\documentclass[11pt, a4paper]{article}
\usepackage[a4paper, total={6in, 8in}]{geometry}
\usepackage{amsthm}
\usepackage{amsmath,amssymb,amsfonts,nicefrac,verbatim}
\usepackage[pdfencoding=auto, psdextra]{hyperref}
\usepackage{color}

\usepackage{bm}
\usepackage{bbm}
\usepackage[toc]{appendix}
\usepackage{enumitem}

\newtheorem{thm}{Theorem}[section]
\newtheorem{theorem}[thm]{Theorem}

\newtheorem{rconj}[thm]{Conjecture}
\newtheorem{lemma}[thm]{Lemma}
\newtheorem{corollary}[thm]{Corollary}
\newtheorem{claim}[thm]{Claim}

\newtheorem{definition}[thm]{Definition}
\newtheorem{remark}[thm]{Remark}

\newtheorem{fact}[thm]{Fact}

\makeatletter
\newtheorem*{rep@theorem}{\rep@title}
\newcommand{\newreptheorem}[2]{%
\newenvironment{rep#1}[1]{%
\def\rep@title{\bf #2 \ref*{##1} \text{(Restated)} }%
\begin{rep@theorem} }%
{\end{rep@theorem} } }

\newtheorem*{rep@claim}{\rep@title}
\newcommand{\newrepclaim}[2]{%
\newenvironment{rep#1}[1]{%
\def\rep@title{\bf #2 \ref*{##1} \text{(Restated)} }%
\begin{rep@claim} }%
{\end{rep@claim} } }

\newtheorem*{rep@lemma}{\rep@title}
\newcommand{\newreplemma}[2]{%
\newenvironment{rep#1}[1]{%
\def\rep@title{\bf #2 \ref*{##1} \text{(Restated)}}%
\begin{rep@lemma} }%
{\end{rep@lemma} } }

\makeatother
\newreptheorem{theorem}{Theorem}
\newrepclaim{claim}{Claim}
\newreplemma{lemma}{Lemma}

\newcommand{\remove}[1]{}
\newcommand\E{\mathop{\mathbb{E}}}
\newcommand\card[1]{\left| {#1} \right|}
\newcommand\sett[2]{\left\{ \left. #1 \;\right\vert #2 \right\}}

\newcommand\set[1]{{\left\{ #1 \right\}}}
\newcommand\Prob[2]{{\Pr_{#1}\left[ {#2} \right]}}

\newcommand\Expect[2]{{\mathop{\mathbb{E}}_{#1}\left[ {#2} \right]}}

\newcommand\cExpect[3]{{\mathbb{E}_{#1}\left[ \left. #3 \;\right\vert #2 \right]}}
\newcommand\norm[1]{\left\| #1 \right\|}
\newcommand\inner[2]{\langle{#1},{#2}\rangle}
\newcommand\eps{\varepsilon}
\renewcommand\geq{\geqslant}
\renewcommand\leq{\leqslant}
\renewcommand\ge{\geqslant}
\renewcommand\le{\leqslant}

\newcommand{\B}{\mathcal{B}}
\newcommand{\A}{\mathcal{A}}

\newcommand{\SO}{\textup{SO}}
\newcommand{\Spin}{\textup{Spin}}
\newcommand{\Sp}{\textup{Sp}}
\newcommand{\SU}{\textup{SU}}
\newcommand{\U}{\textup{U}}

\newcommand{\sign}{\textrm{sign}}
\newcommand{\parenth}[1]{{\left( #1 \right)}}
\renewcommand{\O}{\textup{O}}

\renewcommand{\epsilon}{\eps}

\newcommand\T{\mathrm{T}}
\newcommand\Trow{\T_{{\sf row}}}
\newcommand\Tcol{\T_{{\sf col}}}

\newcommand{\rom}[1]{\uppercase\expandafter{\romannumeral #1\relax}}
\newcommand{\gnote}[1]{{\marginpar{\color{red}$\rightarrow$\footnotesize { #1}}}}

\title{Product Mixing in Compact Lie Groups}
 \author{David Ellis\thanks{Department of Mathematics, University of Bristol.}
 	\and
  Guy Kindler\thanks{Department of Computer Science, Hebrew University of Jerusalem.}
 	\and
 	Noam Lifshitz\thanks{Einstein Institute of Mathematics, Hebrew University of Jerusalem. Supported by the Israel Science Foundation (grant no.~1980/22).} 
     \and
 	Dor Minzer\thanks{Department of Mathematics, Massachusetts Institute of Technology. Supported by a Sloan Research
 Fellowship, NSF CCF award 2227876 and 
 NSF CAREER award 2239160.}
 }
\date{\vspace{-5ex}}
\begin{document}
\global\long\def\rb{\mathrm{R}_{V}}%
\global\long\def\trow{\mathrm{T}_{\sf{row}}}%

\global\long\def\tcol{\mathrm{T}_{\sf{col}}}%
\global\long\def\lA{\mathrm{L}_{U}}%

\maketitle

\begin{abstract}
If $G$ is a group, we say a subset $S$ of $G$ is {\em product-free} if the equation $xy=z$ has no solutions with $x,y,z \in S$.\remove{In 1985, Babai and S\'{o}s \cite{bs} asked, for a finite group $G$, how large a subset $S\subseteq G$ can be if it is product-free. The main tool (hitherto) for studying this problem has been the notion of a {\em quasirandom group}.} For $D \in \mathbb{N}$, a group $G$ is said to be {\em $D$-quasirandom} if the minimal dimension of a nontrivial complex irreducible representation of $G$ is at least $D$. Gowers showed that in a $D$-quasirandom finite group $G$, the maximal size of a product-free set is at most $|G|/D^{1/3}$. This disproved a longstanding conjecture of Babai and S\'os from 1985.

 For the special unitary group, $G=\SU(n)$, Gowers observed that his argument yields an upper bound of $n^{-1/3}$ on the measure of a measurable product-free subset. In this paper, we improve Gowers' upper bound to  $\exp(-cn^{1/3})$, where $c>0$ is an absolute constant. In fact, we establish something stronger, namely, {\em product-mixing} for measurable subsets of $\SU(n)$ with measure at least $\exp(-cn^{1/3})$; for this product-mixing result, the $n^{1/3}$ in the exponent is sharp.
 
 Our approach involves introducing novel hypercontractive inequalities, which imply that the non-Abelian Fourier spectrum of the indicator function of a small set concentrates on high-dimensional irreducible representations.
 Our hypercontractive inequalities are obtained via methods from representation theory, harmonic analysis, random matrix theory and differential geometry. We generalize our hypercontractive inequalities from $\SU(n)$ to an arbitrary $D$-quasirandom compact connected Lie group for $D$ at least an absolute constant, thereby extending our results on product-free sets to such groups. 

We also demonstrate various other applications of our inequalities to geometry (viz., non-Abelian Brunn-Minkowski type inequalities), mixing times, and the theory of growth in compact Lie groups. A subsequent work due to Arunachalam, Girish and Lifshitz uses our inequalities to establish new separation results between classical and quantum communication complexity.
\end{abstract}
\newpage 

\tableofcontents

\section{Introduction}

A subset $\A$ of a group $G$ is said to be {\em product-free} if $gh \notin \A$ for all $g,h\in \A$. The study of product-free subsets of groups has attracted significant attention over the past three decades. In
1985, Babai and S\'os \cite{bs} considered the problem of determining the largest size of a product-free set in a finite group $G$. They conjectured that exists an absolute 
positive constant $c_0 >0$ such that any finite group $G$ has a product-free set of size at 
least $c_0|G|$. In the Abelian case, this is quite easy to see, and had previously been observed by Erd\H{o}s, in an unpublished communication to Babai and S\'os. (In the cyclic case $(\mathbb{Z}_n,+)$, one can take a `middle-third' construction, viz., $\{x \in \mathbb{Z}_n:\ n/3 < x \leq 2n/3\}$, as a large product-free set, and one can reduce to the cyclic case by observing that any finite Abelian group has a nontrivial cyclic quotient, and that the preimage of a product-free set under a quotient map is also product-free and of the same measure.) The exact answer in the Abelian case was given by Green and 
Ruzsa \cite{gr} in 2003: the largest product-free subset of a finite Abelian group $G$ has 
size $c |G|$, where the function $c = c(G) \in [2/7,1/2]$ was explicitly determined by 
Green and Ruzsa. The general Babai-S\'os conjecture was disproved in 2008 by Gowers 
\cite{gowers}, who showed that if $G$ is a finite group such that the minimal dimension of 
a nontrivial irreducible complex representation of $G$ is equal to $D$, then any product-free subset of $G$ has size at most $D^{-1/3}|G|$. It remains to observe that the quantity 
$D = D(G)$ is unbounded over finite non-Abelian groups $G$. For example, for the projective 
special linear group $\textup{PSL}_2(\mathbb{F}_q)$ (for $q$ an odd prime power), we have $D(\textup{PSL}_2(\mathbb{F}_q)) = (q-1)/2$, so the measure of a product-free subset of $\textup{PSL}_2(\mathbb{F}_q)$ is at most $O(q^{-1/3})$, which tends to zero as $q$ tends to infinity. 
\remove{ [David: 
before, $SU_n$ was mentioned here in place of PSL, but this is problematic as $\SU(n)$ isn't 
a finite group and the Babai-S\'{o}s conjecture only refers to finite groups. I'll rewrite the 
next paragraph slightly, too, if that's OK.\gnote{Sure!}]}

Gowers observed that his argument also implies that if $G$ is an (infinite) compact group for which the minimal dimension of a nontrivial irreducible complex continuous representation is equal to $D$, then the maximal Haar measure of a measurable, product-free set in $G$ is at most $D^{-1/3}$. For $\SU(n)$ we have $D(\SU(n)) = n$, implying an upper bound of $n^{-1/3}$ on the measure of a measurable product-free subset of $\SU(n)$. However, Gowers conjectured that for $\SU(n)$, the true answer is exponentially small in $n$. Indeed, as Gowers states, it seems difficult to come up with an example better than the following. Recall that group $\SU(n)$ acts on the complex unit sphere $\{v \in \mathbb{C}^n: \|v\|_2=1\}$, and take $\cal A$ to be the set of all matrices $A\in \SU(n)$ such that the real part of $\langle Ae_1,e_1 \rangle$ is less than $-1/2$. As noted by Gowers, it follows from the triangle inequality that this set is product-free, and it is easy to check that the measure of $\cal A$ is $2^{-\Omega(n)}$.

In this work, we make progress towards proving Gowers' conjecture. Specifically, we improve Gowers' upper bound by a stretched exponential factor, viz., from $n^{-1/3}$ to $e^{-cn^{1/3}}$. 
\begin{theorem}
\label{thm:sun-stronger}
There exists an absolute constant $c>0$ such that the following holds. Let $n \in \mathbb{N}$ and let $\A \subset \SU(n)$ be Haar-measurable and product-free. Then $\mu(\A) \leq \exp(-cn^{1/3})$.
\end{theorem}

\remove{
The construction of $\cal A$ above can be viewed as a special case (for $\SU(n)$) of a method of Kedlaya \cite{kedlaya1} for constructing (fairly large) product-free sets in a group $G$ that acts on a set $X$. Fix a subset $B\subseteq X$ and an element $x\in X$, and define the corresponding \emph{Kedlaya set} $K_{x,B}$ by  
\[
K_{x,B}:=\set{g \in G:\, gx\in B\ \text{and } g(B)\subseteq X \setminus B}.
\]
It clear that $K_{x,B}$ is product-free. Such constructions are known to be optimal in some cases, as we will describe below.
 We conjecture that the following holds for every compact group.

\begin{rconj}
    There exists an absolute constant $C>0$ such that the following holds. Let $G$ be a compact group equipped with its Haar probability measure $\mu$. Then there exists a Kedlaya set $K$ in $G$, such that every measurable product-free set in $G$ has Haar measure at most $C\cdot \mu(K)$.
\end{rconj}
David: it turns out that this conjecture is trivially true. Indeed, if B is any product-free set, and X is the group G itself, acted on from the left by G in the natural way, and x is the identity element of G, then the Kedlaya set K_{x,B} is precisely the set B itself. One might be able to make the conjecture nontrivial by imposing some kind of smallness condition on X, but it isn't clear exactly how to do this at present. So I suggest removing the conjecture for the time being.
}
\subsection{Quasirandomness for groups, and mixing.}

Gowers' bound for product-free sets relies on a relationship between spectral gaps and dimensions of irreducible representations, a relationship which was first discovered by Sarnak and Xue \cite{sarnak1991bounds}. In fact, Gowers' proof uses a beautiful connection between the problem and a purely representation-theoretic notion that Gowers called {\em quasirandomness} (due to a rough equivalence with the graph-quasirandomness of certain Cayley graphs, an equivalence which we shall explain below). For a group $G$ we denote by $D(G)$ the minimal dimension of a non-trivial complex irreducible continuous representation of $G$. (Henceforth, for brevity, we will use the term {\em representation} to mean continuous representation.) For $d \in \mathbb{N}$, we say that a group $G$ is {\em $d$-quasirandom} if $D(G) \geq d$.\footnote{To avoid confusion with the quasirandomness parameter for graphs, it might have been less ambiguous to call this notion `$d$-group-quasirandomness', but as the latter is rather cumbersome we have opted for the above shorter formulation; we hope that this will not cause the reader confusion, in the sequel.} Denoting by $\alpha(G)$ the largest possible density $\frac{|A|}{|G|}$ of a product-free set $A\subseteq G$ (if $G$ is a finite group), Gowers showed that for any finite group $G$, $\alpha(G) \le D(G)^{-1/3}$. 
Since $D(G)$ can be arbitrarily large (as is the case for the alternating groups, which have $D(A_n) = n-1$ for all $n \geq 7$, and the groups $\textup{PSL}_2(\mathbb{F}_q)$ as mentioned above, and for many other natural infinite families of finite groups), this disproved the conjecture of Babai and S\'os. 

For finite groups, the quasirandomness parameter gives an almost complete description of the maximal size of a product free set. Pyber (see \cite{gowers}) used the Classification of Finite Simple Groups, together with a construction of Kedlaya, showing that  $\alpha(G)\ge D(G)^{-C}$ where $C>0$ is an absolute constant. Nikolov and Pyber \cite{np} later improved this to  $\alpha(G)\ge \frac{1}{CD(G)}$. This established a remarkable fact, namely that the purely representation theoretic quasirandomness parameter $D(G)$ is polynomially related to the the combinatorial quantity $\alpha(G)$. 
\begin{equation}\label{eq:Nikolov pyber Gowers}
(CD(G))^{-1}\le \alpha(G)\le D(G)^{-1/3}    
\end{equation}

For compact connected Lie groups we obtain the 
following general variant of Theorem~\ref{thm:sun-stronger}, which gives an upper bound on the size of a product-free set in the group.\remove{It turns out that our methods are not specific just to $\SU(n)$, indeed, similar arguments work for all simply connected compact Lie groups, and this yields the following.}

\begin{theorem}
\label{thm:son-stronger}
There exists an absolute constant $c>0$ such that the following holds. Let $G$ be a compact connected Lie group, and let $\tilde{G}$ be its universal cover. Let $\A \subset G$ be Haar-measurable and product-free. Let $\mu$ denote the Haar probability measure on $G$. Then $\mu(\A) \leq \exp(-cD(\tilde{G})^{1/3})$.
\end{theorem} 

An elegant argument of Gowers \cite{gowers} (proof of Theorem 4.6, therein) for finite groups, which generalises very easily to the case of compact groups, shows that if $G$ is a compact group then it has a measurable product-free subset of measure at least $\exp(-\Omega(D(G))$. In Section \ref{sec:min-rank} we show that $D(G) = O({D(\tilde{G})^2})$. These two facts combine with Theorem \ref{thm:son-stronger} to give the following analogue of (\ref{eq:Nikolov pyber Gowers}) for compact connected Lie groups.

\begin{corollary}
\label{cor:general-stronger}
There exists an absolute constant $c>0$ such that the following holds. For every compact connected Lie group $G$, \[c D(G)^{1/6} \le \log(1/\alpha(G)) \le \frac{1}{c} D(G).\]
\end{corollary}

\noindent (We remark that our logs will always be taken with respect to the natural basis.) Corollary \ref{cor:general-stronger} says that, as with finite groups, the maximal measure of a measurable product-free set in a compact connected Lie group is controlled by the quasirandomness parameter, but this time the control moves to the exponent.

Quasirandomness (for groups) was a crucial ingredient in the `Bourgain--Gamburd expansion machine', which is a three-step method for obtaining spectral gaps for Cayley graphs (see e.g.\ Tao \cite{tao2015expansion}, for an exposition). Briefly, this `machine' proceeds as follows: one first shows that the graph has high girth, then one shows that there are no `approximate subgroups' in which a random walk could be entrapped, and then quasirandomness is used (together with with the trace method) to finally obtain a spectral gap. Quasirandomness (for groups) has many other applications, such as in bounding the diameters of Cayley graphs (see e.g.\ the survery of Helfgott \cite{helfgott2015growth}).

The term `quasirandomness' was used (for groups) by Gowers, due to the following connection with the (now classical) notion of quasirandomness for graphs. (There are, of course, now notions of quasirandomness for a huge variety of combinatorial and algebraic structures; roughly speaking, these say the structure behaves in a random-like way, in an appropriate sense.) We now need some more terminology. The \emph{normalized adjacency matrix} $A_H\in \mathbb{R}^{V\times V}$ of a $d$-regular graph $H=(V,E)$ has $(i,j)$-th entry equal to $1/d$ if $\{i,j\}\in E$, and equal to zero otherwise. The graph $H$ is said to be {\em $\epsilon$-quasirandom} if all the nontrivial eigenvalues of $A_H$ are at most $\epsilon$ in absolute value (here, `nontrivial' means having an eigenvector orthogonal to the constant functions). 

One of the striking consequences of $d$-quasirandomness for a finite group $G$, is that it implies that Cayley graphs of the form $Cay(G,S)$ are $(1/\textup{poly}(d))$-quasirandom, whenever $S$ is a dense subset of $G$. The fact that this only relies on density considerations and does not require any assumption on the structure of $S$, makes the notion of quasirandomness for groups rather powerful.

More generally, applications of quasirandomness for a group $G$ can often be (re)phrased as follows. Suppose that $G$ is $d$-quasirandom, and that we have a linear operator $T:L^2(G) \to L^2(G)$ whose nontrivial eigenvalues we want to bound (in absolute value) from above; suppose further that $T$ commutes with either the left or the right action of $G$ on $L^2(G)$. (In Gowers' proof, slightly rephrased, the operator $T$ could be viewed as $B^*B$, where $B$ is the bipartite adjacency matrix of the bipartite Cayley graph with vertex-classes consisting of two disjoint copies of $G$, and where the edges are all pairs of the form $(g,sg)$ for $g \in G$ and $s \in S$, $S$ being a product-free set in $G$.) Then by the commuting property, each eigenspace of $T$ is a nontrivial representation of $G$, and therefore has dimension at least $d$; it follows that each nontrivial eigenvalue of $G$ has multiplicity at least $d$. But the sum of the squares of the eigenvalues of $T$ is equal to $\textup{Trace}(T^2)$, and this yields the bound $d|\lambda|^2 \leq \textup{Trace}(T^2)$ for all nontrivial eigenvalues $\lambda$ of $T$. This is often called the Sarnak-Xue trick, as it was first employed in \cite{sarnak1991bounds}

Bourgain and Gamburd used their `expansion machine' (alluded to above) to show that taking two uniformly random elements $a,b \in \text{SL}_2(\mathbb{F}_p)$ is sufficient for the Cayley graph $\mathrm{Cay}\left(\text{SL}_2(\mathbb{F}_p, \{ a,b,a^{-1},b^{-1}\}\right)$ to be an expander with high probability, $p$ tending to infinity. 
It is a major open problem in the theory of Cayley graphs to obtain a similar result in the unbounded-rank case, for example for $\text{SL}_n(\mathbb{F}_p)$ where $p$ is fixed and $n$ tends to infinity. One of the properties that breaks down when one attempts to use the Bourgain--Gamburd expansion machine in the case of unbounded rank, is the dependence of the quasirandomness parameter on the cardinality of the group. Specifically, in order for the Bourgain--Gamburd expansion machine to work effectively for a group $G$, the quasirandomness parameter $D(G)$ needs to be polynomial in the cardinality of $G$. In the unbounded rank case, this no longer holds. For example, $D(SL_n(\mathbb{F}_p))\le p^n$ (consider the representation of dimension $p^n$ induced by the natural action of $SL_n(\mathbb{F}_p)$ on $\mathbb{F}_p^n$). The situation is even worse for the alternating group $A_n$, as $D(A_n)=n-1$ for $n \geq 7$, and $n-1$ is less than logarithmic in the cardinality of the group.  

\subsection{Ideas and techniques}

To improve on the upper bound of Gowers, we need to find methods for `dealing with' the low-dimensional irreducible representations (more precisely, for dealing with the corresponding parts of the Fourier transform). In this paper, we develop some new techniques for this in the case of compact connected Lie groups. These techniques turn out also to be useful for finite groups; for example, in \cite{ellis2023global}, analogues of some of our methods are developed for the alternating group $A_n$ (where the idea of mixing is replaced by a refined notion, referred to therein as a `mixing property for global sets').

Below we give indications of the new techniques that are used to obtain our improved bounds, and the various areas of mathematics from which they originate.

\remove{-- level d inequalities}

\subsubsection*{Level $d$ inequalities and hypercontractivity}
One of our key ideas is motivated by the (now well-developed) theory of the analysis of Boolean functions. A function \[f\colon\{-1,1\}^n\to \mathbb{R}\] has a {\em Fourier expansion} $f=\sum_{S \subseteq [n]} \hat{f}(S)\chi_S$, where $\chi_S:\{-1,1\}^n \to \{-1,1\}$ is defined by $\chi_S(x):=\prod_{i\in S}x_i$ for each $x \in \{-1,1\}^n$ and $S \subseteq [n]$. The functions $\chi_S$, known as the {\em Fourier-Walsh functions} or {\em characters}, are orthonormal (with respect to the natural inner product on $\mathbb{R}[\{-1,1\}^n]$ induced by the uniform measure). The Fourier expansion gives rise to a coarser orthogonal decomposition, $f=\sum_{d=0}^n f^{=d},$ where \[f^{=d}:=\sum_{|S|=d}\hat{f}(S)\chi_S.\] This is known as the {\em degree decomposition} (as each function $f^{=d}$ is a homogeneous polynomial of total degree $d$ in the $x_i$'s). 

The \emph{level $d$ inequality} for the Boolean cube (essentially due to Kahn--Kalai--Linial \cite{kahn1988proceedings} and Benjamini--Kalai--Schramm \cite{benjamini1999noise}) states that there exists an absolute constant $C>0$, such that for a set $A\subseteq \{-1,1\}^n$ of density $\frac{|A|}{2^n}=\alpha$, if $d\le\log(1/\alpha)$ then the characteristic function $f=1_A$ satisfies $\|f^{=d}\|_2^2\le \alpha^{2}\left(\frac{C\log(1/\alpha)}{d}\right)^d$.
Roughly speaking, the level $d$ inequality says that indicators of small sets are very much uncorrelated with low degree polynomials. 
One of our key ideas in this paper is to generalize the level $d$ inequality from the Boolean cube to the setting of compact connected Lie groups. 

The main tool in the proof of the Boolean level $d$ inequality is  the Bonami--Gross--Beckner hypercontactivity theorem.
It states that the \emph{noise operator} $\mathrm{T}_{\rho} f: = \sum_{d=0}^n \rho^{d}f^{=d}$ is a contraction as an operator from $L_q$ to $L_p$, for all $q>p\ge 1$ provided $0 \leq \rho \le \sqrt{\frac{q-1}{p-1}}$. This immediately implies that $\|f^{=d}\|_q\le \rho^{-d}\|f^{=d}\|_p$ for any function $f$. Roughly speaking, this last inequality says that $L_p$-norm of a low-degree function does not change too drastically with $p$. This is in stark contrast with the behaviour of indicator functions of small sets, $f=1_A$. These satisfy $\|f\|_p=\alpha^{1/p}$, which does change rapidly with $p$. This difference in behaviours can be used to prove the level $d$-inequality, stating that indicators of small sets are essentially orthogonal to the low degree functions. 
\remove{
Their theorem concerns the noise operator that can be defined Fourier analytically via the formula $\mathrm{T}_{\rho} f = \sum_{d=0}^n \rho^{d}f^{=d}$. Bonami, Gross, and Beckner showed that for each $q>p$, each $\rho \le \frac{\sqrt{q-1}}{\sqrt{p-1}}$, and each $f\colon\{-1,1\}^n \to \mathbb{R}$ we have $\|\mathrm{T}_{\rho}f\|_q\le \|f\|_p.$ 

The hypercontractive inequality immediatly implies that  $\|f^{=d}\|_q\le \rho^{-d}\|f\|_p.$ Roughly speaking this inequality says that $L_p$-norm of low degree functions does not change drastically with $p$. This behaviour is in stark contrast with the behaviour of indicators of small sets $f=1_A$. These satisfy $\|f\|_p=\alpha^{1/p}$, which changes rapidly with $p$. This difference in behaviours can be used to prove the level $d$-inequality. 
}

The same proof-concept works hand in hand with the representation theory of compact simple Lie groups. For simplicity, let us restrict our attention (at first) to the group $G=\SO(n)$. For each $d \in \mathbb{N} \cup \{0\}$, we let $V_{\le d}\subseteq L^2(G)$ denote the subspace of $L^2(G)$ spanned by the polynomials of degree at most $d$ in the matrix entries of $X \in G=\SO(n)$; so, for example, $X_{11}X_{22} \in V_{\leq 2}$. We also let $V_{=d}: = V_{\leq d} \cap (V_{\leq d-1})^{\perp}$, for each $d \in \mathbb{N}$. Given $f\in L^2(G)$, we let $f^{\le d}$ denote the orthogonal projection of $f$ onto $V_{\le d}$, and we let $f^{=d}$ denote the orthogonal projection of $f$ onto $V_{=d}$, so that $f^{=d}=f^{\le d} - f^{\le d-1}$. The subspaces $V_{\le d}$ and $V_{=d}$ are two-sided ideals of $L^2(G)$ (i.e., they are closed under both left and right actions of $G$ on $L^2(G)$). Now, if $J$ is a two-sided ideal of $L^2(G)$ and $T:L^2(G)\to L^2(G)$ is a linear operator that commutes with either the left or the right action of $G$ (as will be the case with all the operators we will work with), it follows from the classical representation theory of compact groups (viz., the Peter-Weyl theorem and Schur's lemma) that $T$ has $J$ as an invariant subspace. Hence, such an operator $T$ has each $V_{=d}$ as an invariant subspace, so each eigenspace of $T$ can be taken to be within one of the $V_{=d}$'s. It therefore makes sense to consider quasirandomness relative to the degree decomposition. For each $d \in \mathbb{N}$, we let $D_d$ be the smallest dimension of a subrepresentation of the $G$-representation $V_{=d}$. The obvious adaptation of the Sarnak-Xue trick, described above, then yields that for any eigenvalue $\lambda$ of $T$ with eigenspace within $V_{=d}$, we have $D_d |\lambda|^2 \leq \textup{Trace}(T^2)$. It turns out that $D_d$ grows very fast with $d$, yielding very strong upper bounds on the corresponding $|\lambda|$ for large $d$.

 On the other hand, an ideal level $d$ inequality would imply that if $A$ is an indicator of a small set, then most of its mass lies on the high degrees. This combines with the fast growth of $D_d$ (with $d$) to give a much more powerful form of quasirandomness, one that takes into account the fact that $f$ is $\{0,1\}$-valued, and gives much better bounds.

We remark that the above degree decomposition can be easily extended to all compact linear Lie groups $G \leq GL_n(\mathbb{C})$ by letting $V_{\le d}$ be the space of degree $\le d$ polynomials in the real and imaginary parts of the matrix entries of $X \in G$. (In fact, this notion generalizes fairly easily to arbitrary compact simple Lie groups, even when they are not linear.) As in the $\SO(n)$ case, we let $f^{\le d}$ denote the orthogonal projection of $f$ onto $V_{\le d}$.

We obtain the following level $d$ inequality.       

\begin{thm}\label{first-level-d}
There exists absolute constants $c,C>0$ such that the following holds. Let $G$ be a simple compact Lie group equipped with its Haar probability measure $\mu$. Suppose that $D(G)\ge C$. Let $A\subseteq G$ be a measurable set with $\alpha:=\mu(A)\ge e^{-cD(G)}$. Then for each $d \in \mathbb{N} \cup \{0\}$ with $d \le \log(1/\alpha)$, we have $\|f^{\le d}\|_2^2 \le \alpha^2 \left(\frac{2\log(1/\alpha)}{d}\right)^{Cd}$. 
\end{thm}

When $G$ is simply connected and $d\le c\sqrt{n}$ we are able to obtain an even stronger level $d$ inequality, which is similar to the one on the Boolean cube without the extra $C$ factor in the exponent. This leads to the following.

\begin{thm}\label{thm:second-level-d}
There exists absolute constants $C,c>0$ such that the following holds. Let $G$ be a compact connected Lie group, let $\tilde{G}$ denote its universal cover, and write $n=D(\tilde{G})$. Suppose that $n \geq C$. Let $A\subseteq G$ be a measurable set with $\alpha:=\mu(A)\ge \exp{(-cn^{1/2})}$. Then for each $d \in \mathbb{N} \cup \{0\}$ with $d \le \log(1/\alpha)$, we have $\|f^{\le d}\|_2^2 \le \alpha^2 \left(\frac{C\log(1/\alpha)}{d}\right)^{d}$. 
\end{thm}

   \remove{
   \begin{thm}\label{thm: level-d in O_n}
     There exists absolute constants $\delta>0$ and $C>0$, such that for any
     $f\colon \O\O(n) \to \{0, 1\}$ with expectation $\alpha$, and $d\le \min (\delta n^{1/2}, \log (1/\alpha)/100)$, it holds that
     \[
     \|f^{\le d}\|_{L^2(\mu)}^2 \le \left(\frac{C}{d}\right)^d\alpha^2\log^{d}(1/\alpha ).
     \]
     The same bound holds for the group $\U(n)$
    \end{thm}
    }
    
  It is this second level $d$-inequality that is responsible for the $1/3$ in the exponent of Theorem \ref{thm:son-stronger}. Unfortunately, one would not be able to improve our $1/3$ in the exponent to the (conjectural) right one, merely by strengthening this level $d$-inequality. Indeed, our second level $d$ inequality can be easily seen to be sharp up to the value of the absolute constant $C$, by considering sets of the form $\{A \in \SO(n):\ \langle Ae_1,e_1 \rangle > 1-t\}$ for appropriate values of $t$, when $G=\SO(n)$, for example.
  
\medskip 
Both of our level $d$ inequalities are inspired by the same ideas from the Boolean setting, together with an extra representation theoretic ideas. 
Namely, in order to show a level $d$ inequality, we upper-bound $q$-norms of low degree polynomials in terms of their 2-norms, and then use H\"older's inequality. In the Boolean cube, such upper bounds follow from two facts. The first is that the noise operator $\mathrm{T}_{\rho}$ is hypercontractive. The second is that all the eigenvalues of the restriction of $\mathrm{T}_{\rho}$ to $V_{\le d}$ are bounded from below by $\rho^d$. Our approach is to construct operators on $L^2(G)$ that satisfy the same two properties.

\subsubsection*{Differential geometry and Markov diffusion processes}

Our level $d$ inequalities stem from two techniques for obtaining hypercontractivity. Our first level $d$ inequality, Theorem~\ref{first-level-d}, is obtained via the following method. First, we observe that we may assume without loss of generality that our group $G$ is simply connected. (This is because every compact simple Lie group is a quotient of its universal cover by a discrete subgroup of its centre.) We then make use of classical lower bounds on the Ricci curvature of our (simply connected) compact simple Lie group. The Bakry-Emery criterion \cite{bakry1985diffusions} translates such lower bounds on the Ricci curvature into log-Sobolev inequalities for the Laplace-Beltrami operator $L$. We then apply an inequality of Gross \cite{gross} to deduce a hypercontractive inequality for the operator $e^{-tL}$ from the log-Sobolev inequality. This inequality then allows us to prove our first level $d$ inequality. The operator $e^{-tL}$ is the one corresponding to Brownian motion on $G$. In order to deduce our level $d$-inequality we rely on a formula for the eigenvalues of the Laplacian in terms of a step vector corresponding to each eigenspace. This formula is well-known in the theory of Lie groups; it is given for example in Berti and Procesi \cite{berti-procesi}.  

\subsubsection*{Random walks on bipartite graphs
}
There are two mutually adjoint linear operators that correspond to a random walk on a $d$-regular bipartite graph $B\subseteq L\times R$. We denote those by $T\colon L^2(L)\to L^2(R)$ and $T^*\colon L^2(R)\to L^2(L)$ and they are given by taking expectations over a random neighbour; explicitly, $(Tf)(x) = \mathbb{E}_{y \sim x} f(y)$ for $f \in L^2(L)$ and $x \in R$, and $(T^*g)(y) = \mathbb{E}_{x \sim y}g(x)$ for $g \in L^2(R)$ and $y \in L$. It is easy to see that both operators are contractions with respect to any norm. It turns out that given such a bipartite graph and given a hypercontractive operator $S$ on $R$ one gets for free that the operator $T^*ST$ is hypercontractive. Filmus et al \cite{filmus2020hypercontractivity} used this idea to obtain a `non-Abelian' hypercontractive estimate for `global' functions on the symmetric group, from an `Abelian' hypercontractive result for `global' functions on $(\mathbb{Z}_n)^n$. (Informally, a `global' function is one where one cannot increase the expectation very much by restricting the values of a small number of coordinates.)

In this work, we extend this idea to the continuous domain, by replacing a bipartite graph by a coupling of two probability distributions. Specifically, we consider the probability space $(\mathbb{R}^{n \times n},\gamma)$ of $n$ by $n$ Gaussian matrices (i.e., $\mathbb{R}^{n \times n}$ with each entry being an independent standard Gaussian), and the Haar measure on $\O(n)$. For $(\mathbb{R}^{n \times n},\gamma)$, the Ornstein--Uhlenbeck operator $U_\rho$ is a hypercontractive analogue of the noise operator from the Boolean case. We couple $(\mathbb{R}^{n \times n},\gamma)$ with $\SO(n)$ by applying the Gram--Schmidt operation on the columns of a given Gaussian matrix (flipping the sign of the last column, if necessary, so as to ensure that the determinant is equal to one). We note that essentially the same coupling has been used before, e.g.\ by Jiang \cite{jiang}; however, it has not been used before (to our knowledge) to analyse the distribution of high-degree polynomials in the matrix-entries, which is crucial in our work.

This coupling gives rise to operators $\Tcol$ and $\Tcol^*$, similar to the ones in the discrete case. The hypercontractive inequality for the Ornstein--Uhlenbeck operator $U_\rho$, together with our coupling implies a hypercontractive inequality for the operator $\mathrm{T}_{\rho}':=\Tcol^*U_\rho \Tcol$. We then use a symmetrization trick to obtain an operator $\mathrm{T}_\rho:=\mathbb{E}_{B\sim \mu}R_B^*\mathrm{T}'_{\rho}R_B$, where $R_B$ corresponds to right multiplication by $B$. The symmetrization does not change the hypercontractive properties, which are the same as for $U_\rho$ (see Theorem~\ref{thm:hypercontractivity in O_n}), but it has the advantage of allowing us to analyse more easily the eigenvalues of the operator. 

\subsubsection*{Representation theory}
The hypercontractive inequality for the operator $\mathrm{T}_\rho$ is useful due to the fact that it immediately gives bounds on the norms of eigenfunctions of $\mathrm{T}_{\rho}$. Because of the symmetrization, $\mathrm{T}_{\rho}$ commutes with the action of $G$ from both sides. Therefore, the Peter-Weyl theorem implies that every isotypical\footnote{If $\rho$ is an irreducible representation of $G$ and $V$ is a $G$-module, the {\em $\rho$-isotypical component} of $V$ is the sum of all subrepresentations of $V$ that are isomorphic to $\rho$.} 
component of $L^2(G)$ is contained in an eigenspace of $\mathrm{T}_{\rho}$. 

We eventually show that the eigenvalues of the restriction of $\mathrm{T}_\rho$ to $V_{d}$ are at least $(c\rho)^d$, for some absolute constant $c>0$. This implies that $\mathrm{T}_\rho$ is indeed a good analogue of the noise operator on the Boolean cube, and of the Ornstein--Uhlenbeck operator $U_\rho$, i.e.\ the noise operator on Gaussian space. We obtain this lower bound by showing that each isotypical component contains certain functions that are nice to deal with, functions we call the \emph{comfortable juntas}.

The latter are defined as follows. We define a {\em $d$-junta} to be a function in the matrix entries of $X \in \SO(n)$ that depends only upon the upper-left $d$ by $d$ minor of $X$. Such a $d$-junta is said to be \emph{comfortable} $d$-junta if it is contained in the linear span of the monomials $\{m_{\sigma}:\ \sigma \in S_d\}$, where $m_\sigma : \SO(n) \to \mathbb{R}$ is defined by $m_\sigma(X)=\prod_{i=1}^d X_{i,\sigma(i)}$ for each $X \in \SO(n)$, for each permutation $\sigma \in S_d$.

\subsubsection*{Random matrix theory}
One of the main discoveries of random matrix theory is that the entries of a random orthogonal matrix behave (in an appropriate sense) like independent Gaussians of the same expectation and variance: at least, when one restricts minors of the matrix that are not too large. (In fact, this holds for many different models of random matrices, not just the orthogonal ensemble.) The power of this discovery is of course that a Gaussian random matrix is {\em a priori} much easier to analyse than e.g.\ the random matrix given by the Haar measure on a group.

One way to test that two distributions are similar is to apply a continuous `test function' and take expectations. Usually, for applications in random matrix theory, the test function can be taken to be an arbitrary fixed polynomial. \remove{In fact, real random variables $X_n$ converge in distribution to a real random variable $X$ if and only if for each polynomial $P$, the sequence of expectations $\mathbb{E}[P(X_n)]$ converges to $\mathbb{E}[P(X)]$.} 

When computing the eigenvalues of our operator $\mathrm{T}_{\rho}$ we need to show a similarity in distribution between the upper $d\times d$-minor of $\O(n)$ and the $d \times d$ minor of a random Gaussian matix. For us, however, it is not sufficient to look at a single polynomial of fixed degree. Instead, we need to show a similarity in the distribution with respect to our comfortable $d$-juntas (where $d$ may be as large as $\sqrt{n}$, rather than an absolute constant). Hence, while the philosophy is similar to that of random matrix theory, we require new techniques enabling us to deal with the distributions of polynomials whose degrees may be a function of $n$, indeed up to $\sqrt{n}$.

\subsection{Applications}
In this section we list several applications of our hypercontractive theory: to some problems in group theory, in geometry, and in probability.

To state some of our results, we need some more terminology. If $G$ is a compact connected Lie group, we define $n(G): = D(\tilde{G})$, where $\tilde{G}$ denotes the universal cover of $G$. It is well-known that, for each $m \in \mathbb{N}$, we have $D(\SU(m))=D(\Spin(m))=m$ and $D(\Sp(m))=2m$ (and all these groups are simply connected except for $\Spin(2)$); we also have $D(\SO(m))=m$. Since $\Spin(m)$ is the universal cover of $\SO(m)$ for all $m > 2$, we have $n(\SO(m)) = m$ for all $m > 2$. As we will see in the next section, any compact connected semisimple Lie group $G$ with $D(G)$ at least an absolute constant, can be written in the form $(\prod_{i=1}^{r}K_i)/F$ where each $K_i$ is one of $\SU(n_i),\Spin(n_i)$ or $\Sp(n_i)$ for some $n_i \geq 3$, and $F$ is a finite subgroup of the centre of $\prod_{i=1}^{r}K_i$; the universal cover of such is $\prod_{i=1}^{r}K_i$, and $D(\prod_{i=1}^{r}K_i) = \Theta(\min_i n_i)$. Hence, the quantity $n(G)$ has a very explicit description in terms of the structure of the Lie group $G$.

\subsubsection*{Growth in groups: the diameter problem}

The theory of growth in groups has been a very active area of study in recent decades, and an important class of problem in this area is to determine the diameter of a metric space defined by a group (e.g.,\ the diameter of a Cayley graph of the group). For a compact group $G$ equipped with its Haar probability measure, and 
a measurable generating set $\mathcal{A}\subseteq G$ of measure $\mu$, it is natural to consider the metric space on $G$ where the distance between $x$ and $y$ is defined to be the minimal length of a word in the elements of $\mathcal{A}$ and their inverses which is equal to $xy^{-1}$. The diameters of such metric spaces in the case where $G$ is finite have become a focus of intense study in the last two decades: see e.g.\ the works of Liebeck and Shalev~\cite{liebeck2001diameters},  Helfgott~\cite{helfgott2008growth}, Helfgott and Seress~\cite{helfgott2014diameter},  Pyber and Szabo~\cite{pyber2016growth} and Breuillard, Green and Tao~\cite{breuillard2011approximate}.

For a subset $\mathcal{A}$ of a group $G$ and $t\in\mathbb{N}$, we define   
$$\mathcal{A}^t := \sett{a_1\cdot a_2\cdots a_t}{a_1,a_2,\ldots,a_t\in\mathcal{A}}.$$
The {\em diameter problem for $G$ with respect to $\mathcal{A}$} asks for the smallest positive integer $t$ for which $\A^t=G$. For a compact group $G$  and a real number $0 < \alpha \leq 1$, the {\em diameter problem for sets of measure $\alpha$ in $G$} asks for the minimum possible diameter of a measurable set in $G$ of measure $\alpha$.

In the case where $G$ is a compact and connected group, we note that the diameter of $G$ with respect to any subset $\mathcal{A}$ of positive measure is finite. This follows (almost) immediately from Kemperman's theorem \cite{kemperman}, which states that for any compact connected group $G$ (equipped with its Haar probability measure $\mu$) and any measurable $\mathcal{A},\mathcal{B} \subset G$, we have $\mu(\mathcal{A}\mathcal{B}) \geq \min\{\mu(\mathcal{A})+\mu(\mathcal{B}),1\}$. 

We make the following conjecture, concerning the diameter of large sets.
\begin{rconj}
    Let $G$ be one of $\SU(n),\SO(n),\Spin(n)$ or $\Sp(n)$, and let $\A\subseteq G$ be a measurable subset of measure $\nu$. Then the diameter of $G$ with respect to $\A$ is $O(\nu^{-1/(\ell n)})$, where $\ell=1$ in the case of $\SO(n)$ and $\Spin(n)$, $\ell=2$ in the case of $\SU(n)$, and $\ell=4$ in the case of $\Sp(n)$. In particular, if $\nu \geq e^{-cn}$, then the diameter of $G$ with respect to $\A$ is at most $O_c(1)$.
\end{rconj}
We note that if true, the conjecture is essentially tight, as can be seen for $\SO(n)$ by considering the set 
\[
\mathcal{S}_\epsilon := \set{ X\in  \SO(n)\colon 
\text{the angle between $Xe_1$ and $e_1$ is at most $\epsilon$}
},
\]
For $\eps \leq 1/2$, we have $\mu(\mathcal{S}_{\eps}) = (\Theta(\epsilon))^n$, and the diameter of $\SO(n)$ with respect to $\mathcal{S}_{\epsilon}$ is $\Theta(1/\epsilon)$. If $\pi:\Spin(n) \to \SO(n)$ is the usual (double) covering homomorphism, then the lift $\pi^{-1}(\mathcal{S}_{\epsilon})$ is a subset of $\Spin(n)$ of the same measure as $\mathcal{S}_{\epsilon}$ (using, of course, the Haar probability measure on both groups), and the diameter of $\Spin(n)$ with respect to $\pi^{-1}(\mathcal{S}_{\epsilon})$ is the same the diameter of $\SO(n)$ with respect to $\mathcal{S}_{\epsilon}$, since $\pi(\mathcal{A}^t) = (\pi(\mathcal{A}))^t$ for any subset $\mathcal{A} \subset \Spin(n)$ and any $t \in \mathbb{N}$. Hence, $\pi^{-1}(\mathcal{S}_{\epsilon})$ demonstrates tightness for $\Spin(n)$. The group $\SU(n)$ acts transitively on the unit sphere in $\mathbb{C}^{n}$, which can be identified with $S^{2n-1}$, and the group $\Sp(n)$ acts transitively on the unit sphere in $\mathbb{H}^n$, which can be identified with $S^{4n-1}$; both actions are angle-preserving (in $S^{2n-1}$ and $S^{4n-1}$ respectively). So our above construction for $\SO(n)$ (which comes from the action of $\SO(n)$ on $S^{n-1}$) has the obvious analogues for $\SU(n)$ and $\Sp(n)$, which we conjecture are sharp for those groups. 

\medskip 
We show that for a compact connected Lie group $G$ with $n(G)=n$, for all $\delta >0$ and 
all measurable subsets $\A$ of $G$ with measure at least $2^{-cn^{1-\delta}}$, the diameter of $G$ with respect to $\A$ is at most $O_{\delta}(1)$.

\begin{thm}\label{thm:growth}
For each $\delta>0$ there exist $n_0,k>0$ such that the following holds. Let $n>n_0$ and let $G$ a compact connected Lie group with $n(G)=n$. If $\A \subset G$ is a Haar-measurable set, and
$\mu(\A)\geq 2^{-n^{1-\delta}}$, then $\A^k = G$.
\end{thm}

\subsubsection*{Doubling inequalities for groups}

Theorem \ref{thm:growth} follows from a new lower bound on $\mu(\A^2)$, where $\A\subseteq G$ is a measurable subset of the compact connected Lie group $G$. We prove the following `doubling inequality'.

\begin{thm}\label{thm:Brunn Minkowskii}
There exists absolute constants $C,c>0$ such that the following holds. Let $G$ be a compact connected Lie group with $n(G) = n \geq C$. Let $\mathcal{A}\subseteq G$  be a measurable set with $\mu(\mathcal{A})\ge e^{-cn}$. Then $\mu(\mathcal{A}^2) \ge \mu(\mathcal{A})^{0.1}$.  
\end{thm}

The problem of giving a lower bound on $\mu(\mathcal{A}^2)$ in terms of $\mu(\mathcal{A})$, for $\mathcal{A}$ a measurable subset of a compact group $G$, dates back to the work of Henstock and Macbeath~\cite{henstock1953measure} from 1953, the aforemenentioned bound of Kemperman~\cite{kemperman} from 1964, and the work of Jenkins~\cite{jenkins1973growth} from 1973. Several recent works of Jing, Tran and Zhang have introduced some powerful new methods into the field. For instance, in \cite{jing2021nonabelian}, Jing, Tran and Zhang generalized the Brunn-Minkowskii inequality from $\mathbb{R}^n$ to an arbitrary connected Lie group, using the Iwasawa decomposition to facilitate an inductive approach; their result is essentially sharp for helix-free Lie groups. In \cite{jing2023measure}, they used techniques from $O$-minimal geometry to show that that $\mu(\mathcal{A}^2)\ge 3.99\mu(\mathcal{A})$ for all measurable subsets $\mathcal{A} \subseteq \mathrm{SO}(3)$ of sufficiently small measure. In a forthcoming paper~\cite{jing2023effective} they prove that there exists a function $\delta=\delta(n)$ and an absolute constant $c>0$, such that if $\mathcal{A}\subseteq \mathrm{SO}(n)$ is a measurable set of measure at most $\delta(n),$ then $\mu(\mathcal{A}^2)\ge 2^{cn^{1/10}}\mu(\mathcal{A})$; the function $\delta(n)$ satisfies $\delta(n) \leq 2^{-n^{1+c'}}$ where $c'>0$ is an absolute constant. (For comparison, we note that Theorem \ref{thm:Brunn Minkowskii}, in conjunction with Theorem \ref{thm:mixing time 2} below, imply the existence of an absolute constant $c>0$ such that $\mu(\mathcal{A}^2)\ge \min\{2^{c n^{1/2}}\mu(\mathcal{A}), 0.99\}$ for all measurable subsets $\mathcal{A} \subset \SO(n)$ of measure at least $2^{-cn}$, so our result and that of Jing, Tran and Zhang leave a `gap' between them.) It remains an open problem to determine whether $\mu(\mathcal{A}^2)\ge \min\{2^{n/10}\mu(\mathcal{A}),0.99\}$ for all subsets $\mathcal{A} \subset \SO(n)$. We also mention that very recently (at the same time as the first version of this paper appeared), Machado \cite{machado} showed that if $G$ is a compact Lie group of dimension $d$, and $d'$ is the maximal dimension of a proper closed subgroup of $G$, then for all measurable subsets $\mathcal{A} \subset G$, it holds that
\begin{equation}\label{eq:machado}\mu(\mathcal{A}^2) \geq \left(2^{d-d'}-C_d(\mu(\mathcal{A}))^{2/(d-d')}\right)\mu(\mathcal{A}),\end{equation}
where $C_d$ depends upon $d$ alone. This recovers and improves the above-mentioned result of Jing, Tran and Zhang on $\SO(3)$, and is sharp up to the dependence of $C_d$ upon $d$. It does not, however, imply Theorem \ref{thm:Brunn Minkowskii}, as the dependence of $C_d$ upon $d$ (though not explicitly worked out) is of at least tower-type (so no conclusion can be drawn about sets of measure at least $e^{-\Omega(d)}$, which is the setting of Theorem \ref{thm:Brunn Minkowskii}).

\subsubsection*{Spectral gaps}
We also give the following upper bound on the spectral gaps of the operator corresponding to convolution by $\frac{1_A}{\mu(A)}$. If $G$ is a compact group equipped with its (unique) Haar probability measure $\mu$, for a measurable set $\mathcal{A} \subset G$ we write $x\sim \mathcal{A}$ to mean that $x$ is chosen (conditionally) according to the Haar measure $\mu$, conditional on the event that $x\in \mathcal{A}$.  

 \begin{thm}\label{thm:spectral gap}
  There exist absolute constants $c,C>0$ such that the following holds. Let $G$ be compact connected Lie group and suppose $n:=n(G)\geq C$. Let $\mathcal{A}=\mathcal{A}^{-1}$ be a symmetric, measurable set in $G$ and suppose that $\mu(\mathcal{A})\ge e^{-cn^{1/2}}$. Then the nontrivial spectrum of the operator $T$ defined by  $Tf(x)=\mathbb{E}_{a\sim A}[f(ax)]$ is contained in the interval \[\left[-\sqrt{\frac{C\log{1/\alpha}}{n}}, \sqrt{\frac{C\log{1/\alpha}}{n}}\right].\] 
 \end{thm}

\subsubsection*{Mixing times}
Let $G$ be a compact group, equipped with its (unique) Haar probability measure; then every measurable subset $\A\subseteq G$ of positive Haar measure corresponds to a random walk on $G$. Indeed, we may define a (discrete-time) random walk $R_\A = (X_t)_{t \in \mathbb{N} \cup \{0\}}$ on $G$, by letting $X_0 = \text{Id}$, and for each $t \in \mathbb{N}$, if $X_{t-1}=x$ then $X_t = ax$, where $a$ is chosen uniformly at random from $\A$. In the case where $G$ is finite and $\A$ is closed under taking inverses, this is the usual random walk associated to the Cayley graph $\mathrm{Cay}(G,\A)$. One of the fundamental problems associated to such random walks is to determine their {\em mixing time}. (Following Larsen and Shalev \cite{larsen2008characters}, we say that the \emph{mixing time} of a Markov chain $(X_t)_{t \in \mathbb{N} \cup \{0\}}$ is the minimal non-negative integer $T$ such that the total variation distance between the distribution of $X_{T}$ and the uniform distribution, is at most $1/e$. We note that $1/e$ could be replaced by any other absolute constant $c \in (0,1)$, without materially altering the definition; Larsen and Shalev use the constant $1/e$ as it makes the statement of certain results concerning $S_n$ and $A_n$ more elegant.)

Larsen and Shalev \cite{larsen2008characters} considered the case where $\A$ is a normal set, i.e.\ a set closed under conjugation, and $G$ is the alternating group $A_n$. They showed that for each $\epsilon>0,$ if $\A\subseteq A_n$ of density $\frac{2|\A|}{n!} \ge \mathrm{exp}\left ( -n^{1/2-\epsilon} \right)$, then the mixing time of $R_\A$ is 2, provided that $n\ge n_0(\epsilon)$ is sufficiently large depending on $\epsilon$. Their proof was based upon a heavy use of character theory. Their result is almost sharp, in the sense the number $1/2$ cannot be replaced by any number smaller than $1/2$. We show that a similar phenomenon holds for compact connected Lie groups, even when $\A$ is not a normal set.  

\begin{thm}\label{thm:mixing time 2}
    There exist absolute constants $c,n_0>0$, such that the following holds. Let $G$ be a compact connected Lie group with $n:=n(G)>n_0$. Let $\A\subseteq G$ be a measurable set with Haar measure at least $e^{-c n^{1/2}}$. Then the mixing time of the random walk $R_{\A}$ is 2.
\end{thm}

This result is essentially best possible. For instance, taking $G=\SO(n)$, we may take $\mathcal{A}=\{X \in \SO(n):\ X_{11} > 10/n^{1/4}\}$. It is easy to see that the mixing time of $R_{\mathcal{A}}$ is 3, while $\mu(\mathcal{A}) = \mathrm{exp}(-\Theta(n^{1/2}))$.   

\subsubsection*{Product mixing}
Gowers' proof of his upper bound on the sizes of product-free sets actually establishes a stronger phenomenon, known as \emph{product mixing}. We say that a compact group $G$ (equipped with its Haar probability measure $\mu$) is an $(\alpha,\epsilon)$-\emph{mixer} if for all sets $A,B,C\subseteq G$ of Haar probability measures $\ge \alpha$, when choosing independent uniformly random elements $a\sim A$ and $b\sim B$, the probability that $ab\in C$ lies in the interval $\left(\mu(C)(1-\epsilon),\mu(C)(1+\epsilon)\right)$.  
Gowers' proof actually yields the following statement: there exists an absolute constant $C>0$, such that if $G$ is a $D$-quasirandom compact group, then it is a $(CD^{-1/3},0.01)$-mixer. (The proof is given only for finite groups, but it generalises easily to all compact groups.) For finite groups, Gowers' product-mixing result is sharp up to the value of the constant $C$. Here, we obtain an analogous result for compact connected Lie groups, where the $n^{-1/3}$ moves to the exponent.    

\begin{thm}\label{thm:product mixing intro}
    For any $\epsilon>0$, there exist $c,n_0>0$ such that the following holds. Let $n>n_0$ and let $G$ be a compact connected Lie group with $n:=n(G)>n_0$. Set $\alpha = \mathrm{exp}(-cn^{-1/3}).$ Then $G$ is an $(\alpha, \epsilon)$-mixer.
\end{thm}

This result is sharp up to the dependence of the constants $c = c(\epsilon)$ and $n_0 = n_0(\epsilon)$ upon $\epsilon$. Indeed, we may take $G=\SO(n)$ and let $\mathcal{A}=\mathcal{B}=\{X \in \SO(n):\ X_{11} > 10/n^{1/3}\}$ and 
$\mathcal{C} = \{X \in \SO(n):\ X_{11} < -10 /n^{1/3}\}$, to obtain a triple of sets each of measure $e^{-\Theta(n^{1/3})}$,
such that when choosing $a\sim A$ and $b\sim B$ independently, the probability that $ab\in C$ is smaller than $\tfrac{1}{2}\mu(C)$.

\subsubsection*{Homogeneous dynamics and equidistribution}
Suppose that a compact Lie group $G$ acts on a topological space $X$. The space $X$ is said to be {\em $G$-homogeneous} if $G$ acts transitively and continuously on $X$ (the latter meaning that the action map from $G \times X$ to $X$ is continuous); in this case, $X$ has a unique $G$-invariant probability measure, which is called the Haar probability measure. We obtain the following equidistribution result for homogeneous spaces. 

\begin{thm}\label{thm:mixing}
    For each $\epsilon>0$ there exist $c,n_0>0$ such that the following holds. Let $G$ be a compact connected Lie group with $n(G)=:n >n_0$. Let $X$ be a $G$-homogeneous topological space, and let $\mu_X$ denotes its unique $G$-invariant (Haar) probability measure. Suppose that $\mathcal{A}\subseteq G$ and $\B\subseteq X$ are both measurable sets of Haar probability measures $\ge e^{-cn^{1/2}}$. Let $\nu$ be the probability measure on $X$ which is given by the distribution of $a(b)$, for a uniform random $a\sim \A$ and an (independent) uniform random $b\sim \B$. Then the total variation distance between $\mu$ and 
    $\nu$ is less than $\epsilon$.
\end{thm}

\subsubsection*{$L^q$-norms of low degree polynomials}
We obtain the following upper bounds on the $q$-norms of low degree polynomials (we state the result for $\SO(n)$, for simplicity). 

\begin{thm}\label{thm:Super Bonami in O_n}
    There exist absolute constants $c,C>0$ such that the following holds. Let $q>2$ and let $f\in L^2(\SO(n))$ be a polynomial of degree $d$ in the matrix entries of $X \in \SO(n)$. If $d\le cn$, then
     \[\|f\|_{L^q(\mu)} \le q^{Cd}\|f\|_{L^2(\mu)}.\]
     If moreover, $d\le c\sqrt{n}$, then
    \[\|f\|_{L^q(\mu)} \le (C\sqrt{q})^d\|f\|_{L^2(\mu)}.\]
    \end{thm}

    \subsubsection*{Separation of quantum and classical communication complexity}

Starting with their introduction to computer science in the seminal paper of Kahn Kalai and Linial \cite{kahn1988proceedings}, hypercontractive inequalities have found a huge number of applications in various branches of computer science and related fields (see e.g. \cite{gavinsky2007exponential, mossel2012quantitative, dinur2005hardness, raghavendra2008optimal}, to name but a few). While these applications have generally required hypercontractive inequalities for functions on discrete sets, some applications require continuous domains. For example, in the paper of Klartag and Regev \cite{regev-klartag}, a hypercontractive inequality for functions on the $n$-sphere is used to obtain a lower bound on the number of (classical) communication bits required for two parties to jointly compute a certain function. While with quantum communication, the value of that function can be computed by one party transmitting  only $O(\log n)$ quantum-bits to the other, it was shown in \cite{regev-klartag} that classical communication requires at least $\Omega(n^{1/3})$ bits of communication to be sent, even if parties are allowed to send bits both ways, thereby showing an exponential separation between the power of classical communication and that of one-way quantum communication.

In the field of quantum communication, establishing a significant separation between classical communication and practically implemented modes of quantum communication remains a major open problem. In a forthcoming work, Arunachalam, Girish, and Lifshitz~\cite{arunachalam2023one} apply our hypercontractive inequality for $\SU(n)$ to make a substantial step towards this goal. They used it to obtain an exponential separation between classical communication and a more realistic version of quantum communication, namely the one-clean-qubit model.

\section{Preliminaries: Quasirandomness and min-rank}\label{sec:min-rank}
In this section we show that, for $D$ at least an absolute constant, the universal cover $\tilde{G}$ of a $D$-quasirandom compact connected Lie group $G$ is a product of `classical' (compact, simple, simply connected) Lie groups of the form $\Spin(n),\SU(n),\Sp(n)$. We then make use of this to determine $D(\tilde{G})$, and we show that $D(\tilde{G})\le 4D(G)^2$. 

In what follows, as usual, a {\em compact group} $G$ is a Hausdorff topological group for which the group operations (or equivalently, the map $(g,h) \mapsto gh^{-1}$) are continuous. We recall that a compact group has a unique left-multiplication-invariant probability measure (called the Haar measure), which is also the unique right-multiplication-invariant probability measure. As usual, if $G$ is a compact group, we let
$$L^2(G) = \{f:G \to \mathbb{C}:\ \mathbb{E}_{\mu}[|f|^2] < \infty\} / \sim,$$
where the expectation is with respect to the Haar probability measure $\mu$ on $G$ and the equivalence relation $\sim$ is defined by $f \sim g$ iff $f=g$ $\mu$-almost-everywhere, and we view $L^2(G)$ as a Hilbert space, with the natural inner product,
$$\langle f,g \rangle := \mathbb{E}_{\mu}[\overline{f}g].$$

We make use of the following fact, appearing for example in \cite{procesi} Chapter 10, Section 7.2, Theorem 4, page 380. (Note that the word `Lie' is missing from the statement of this theorem; this omission is clearly just a typographical error.)
\begin{fact}
    Every compact connected Lie group is Lie-isomorphic to a group of the form $(\prod_{i=1}^{r} K_i \times T)/F,$ where each $K_i$ is a simply connected simple compact Lie group (equivalently, $K_i$ is one of $\Sp(n_i)$ for some $n_i \geq 1$, $\Spin(n_i)$ for some $n_i \geq 3$, $\SU(n_i)$ for some $n_i \geq 2$, or the compact form of one of the five exceptional Lie groups, for each $i \in [r]$), $T$ is a finite-dimensional torus (i.e.\ $T = (\mathbb{R}/\mathbb{Z})^m$ for some integer $m$), and $F$ is a finite group contained in the center of $\prod_{i=1}^{r} K_i \times T$, with $F \cap T = \{1\}$. 
\end{fact}
 
\begin{lemma}
    Suppose that $D>1$ and that $G$ is a $D$-quasirandom compact connected Lie group. Then $G$ is semisimple.
\end{lemma}
\begin{proof}
Write  $G= (\prod_{i=1}^{r} K_i \times T)/F$, where each $K_i$ is a simply connected simple compact Lie group, $T$ is a finite-dimensional torus (i.e.\ $T = (\mathbb{R}/\mathbb{Z})^m$ for some integer $m$), and $F$ is a finite group contained in the center of $\prod_{i=1}^{r} K_i \times T$, with $F \cap T = \{1\}$, as in the above fact. Semisimplicity of $G$ is equivalent to $T=\{1\}$. Suppose on the contrary that $T \neq \{1\}$. Let $\pi$ be the projection map from $\prod_{i=1}^{r} K_i \times T$ onto the $T$ component. Since $F \cap T = \{1\}$, we have $\pi(F) = \{1\}$, so the projection $\pi$ induces a (surjective) group homomorphism $\tilde{\pi}$ from $G$ to $T$, and therefore $G$ has a quotient isomorphic to $(\mathbb{R}/\mathbb{Z})^m$ for some integer $m \geq 1$; any nontrivial complex one-dimension irreducible representation of the latter quotient lifts to one of $G$, contradicting the $D$-quasirandomness of $G$ (for any $D > 1$) and proving the lemma.
\end{proof}

We also recall the following standard fact.
\begin{fact}
Every compact semisimple Lie group has finite centre.
\end{fact}

We now show that if $G$ is sufficiently quasirandom, then the exceptional groups do not make an appearance as some $K_i$ when writing $G = (\prod_{i=1}^{r} K_i )/F$. 
\begin{lemma}
Set $D_0=248$. Let $G$ be a compact connected Lie group, and suppose that it is $D$-quasirandom for some $D > D_0$. Then $G = (\prod_{i=1}^{r} K_i )/F$, where each $K_i$ is one of $\Sp(n_i), \Spin(n_i), \SU(n_i)$ for some $n_i \ge \sqrt{D}/2$ and $F$ is a subgroup of the (finite) centre of $\prod_{i=1}^{r}K_i$.     
\end{lemma}
\begin{proof}
By the previous lemma, provided $D_0>1$, $G$ is semisimple. Hence, we may write $G= (\prod_{i=1}^{r} K_i)/F$, where each $K_i$ is one of $\Sp(n_i)$, $\Spin(n_i)$, $\SU(n_i)$ or the compact form of one of the five exceptional Lie groups, for each $i$, and $F$ is a finite group contained in the (finite) center of $\prod_{i=1}^{r} K_i$. As the quotient of a $D$-quasirandom group is $D$-quasirandom, we may project to any one of the components and still obtain a $D$-quasirandom group $K_i/F'$. (In detail, let $\pi_i$ denote projection of $\prod_{j=1}^{r}K_j$ onto the $K_i$ factor; $\pi_i$ induces a surjective homomorphism from $G$ onto $K_i/\pi_i(F)$, and since $F$ is a subgroup of the centre of $\prod_{j=1}^{r}K_j$, $F_i$ is a subgroup of the centre of $K_i$. The group $K_i/\pi_i(F)$ is therefore a quotient of $G$, and so inherits its $D$-quasirandomness.) It is therefore sufficient to consider the case where $G=K_1/F'$. We now note that the adjoint representation of $K_1$ factors through $K_1/F'$ (since $F'$ is contained in the centre of $K_1$), so it can also be viewed as a representation of $K_1/F'$. As the adjoint representation of $K_1$ is not a sum of copies of the trivial representation (this follows from the fact that $K_1$ is non-Abelian), its dimension (which is the same as the dimension of the Lie group $K_1$) is at least $D$. The five exceptional Lie groups, $E_6$, $E_7$, $E_8$, $F_4$ and $G_2$, have dimensions $78$, $133$, $248$, $52$ and $14$ respectively, so $K_1$ cannot equal any of these (since $D>D_0=248$). Hence, $K_1$ is one of  $\Sp(n_1)$, $\Spin(n_1)$ or $\SU(n_1)$. The dimensions of these Lie groups are $n_1(2n_1+1)$, $n_1(n_1-1)/2$ and $n_1^2-1$ respectively, so we obtain $n_1(2n_1+1) \geq D$, which implies that $n_1 \geq \sqrt{D}/2$. This completes the proof of the lemma.
\end{proof}

The following (non-standard) definition will be convenient for us.

\begin{definition}
Let $G$ be a compact, connected, semisimple Lie group and write $G= (\prod_{i=1}^{r} K_i )/F$, where, as above, $K_i$ is one of $\Sp(n_i)$, $\Spin(n_i)$ or $\SU(n_i)$ for each $i$, and $F$ is a finite subgroup of the (finite) centre of $G$. We define the {\em min-rank} of $G$ to be $\min\{n_1,\ldots ,n_r\}$. 
\end{definition}

Using this terminology, the above lemma can be restated by saying that if a compact connected Lie group $G$ is $D$-quasirandom for large enough $D$, then it has min-rank at least $\sqrt{D}/2$.

We remark that the {\em rank} of a Lie group is defined to be the dimension of any one of its Cartan subgroups, so the ranks of $\Sp(n_i)$, $\Spin(n_i)$ and $\SU(n_i)$ are respectively $n_i$, $\lfloor n_i/2 \rfloor$ and $n_i-1$, so in particular are all $\Theta(n_i)$; hence, while the min-rank of $G$ is not exactly the minimum of the ranks of the $K_i$'s (where the $K_i$'s are as above), it is within an absolute constant factor thereof. (We hope this slight abuse of terminology will not cause confusion.)

To establish that $D(\tilde{G})\le 4D(G)^2$ we also need the following. 

\begin{lemma}\label{lem:min-rank and quasirandomness of the universal cover}
    Let $G$ be a compact, connected, semisimple Lie group of min-rank $m$. Then its universal cover $\tilde{G}$ satisfies $D(\tilde{G})\in \{m,2m\}$. 
\end{lemma}
\begin{proof}
Write $G=\prod_{i=1}^r K_i/F$. As the projection map from $\prod_{i=1}^rK_i$ to $G$ is a cover map, and since $\prod_{i=1}^r K_i$ is simply connected, we obtain that $\tilde{G}=\prod_{i=1}^rK_i$. The lemma now follows from the fact that the complex irreducible representations of a product $\prod_{i=1}^{r}K_i$ of finitely many compact groups are tensor products of complex irreducible representations, of the form $\rho_1 \otimes \ldots \otimes \rho_r$ where $\rho_i$ is an complex irreducible representation of $\rho_i$, for each $i$ together with the fact that $D(\SU(n))=D(\SO(n))=n,D(\Sp(n))=2n$.
\end{proof}

Lemma \ref{lem:min-rank and quasirandomness of the universal cover} shows that, when proving the theorems in the introduction, we may replace $D(\tilde{G})$ with the min-rank of $G$.

\section{Good groups and fine groups}
\label{sec:grading-defn}

\subsection{Graded groups and `good' groups.}
In this section we define some basic properties of compact connected groups, which we later use to prove various growth properties. We define \emph{graded} and {\em strongly quasirandom} groups, and \emph{hypercontractive} groups; we say that groups satisfying all these properties are \emph{good}. We also define a somewhat weaker (or, technically, incomparable) notion of a \emph{fine} group. 

The compact, simple, simply connected real Lie groups of large enough rank, i.e $\SU(n)$, $\Sp(n)$ and $\Spin(n)$, are indeed good (this is proved in Section~\ref{sec:coupling_introduce}). We show that goodness is preserved when taking products and quotients (quotients, that is, by closed normal subgroups, as usual), thereby showing that every $D$-quasirandom group is a good graded group, provided $D$ is sufficiently large. 

\begin{definition}[Graded groups]
For $n \in \mathbb{N}$, we say a compact connected group $G$ is {\em $n$-graded} if there exists an orthogonal direct sum, 
\[L^2(G)=\bigoplus^{\lceil n/2 \rceil -1 }_{d=0} V_{=d} \oplus V_{\geq n/2},\]
such that the spaces $V_{=d}$ are invariant under the action of $G$ from both sides, and where $V_{=0}$ contains only the constant functions. For an $n$-graded group and an integer $0 \leq d_0<n/2$, we denote by $V_{>d_0}$ the direct sum 
$V_{>d_0}:=\displaystyle{\bigoplus^{\lceil n/2 \rceil -1 }_{d=d_0+1} V_{=d} \oplus V_{\geq n/2}}$. (Note that we will sometimes write $V_{=d}^{G}$ in place of $V_{=d}$, when we want to stress that the group in question is $G$, e.g.\ if there are several groups involved in our argument.)
\end{definition}
\begin{remark}
    Note that it follows from the definition of an $n$-grading that $V_{\geq n/2}$ is also invariant under the action of $G$ from both sides.
\end{remark}

\paragraph{Grading for the compact simply connected simple Lie groups (of large enough rank).} We note in this section that $\SO(n)$, $\SU(n)$, $\Sp(n)$ and $\Spin(n)$ are all $n$-graded, for $n \geq 3$. For the group $\SO(n)$, for each integer $0 \leq d < n/2$ we define $V_{\leq d}$ to be the subspace of $L^2(\SO(n))$ spanned by degree $\leq d$ multivariate polynomials in the matrix entries of $X \in \SO(n)$ (so, for example, $V_{\leq 2}$ contains the polynomial $X_{11}X_{12}$). For notational convenience, for $0 < r < d/2$ we define $V_{<r}$ to be $V_{\leq d}$, where $d$ is the maximal integer less than $r$. We then set $V_{=0} := V_{\leq 0}$ and $V_{=d}:=V_{\leq d} \cap (V_{\leq {d-1}})^\perp$ for each integer $1 \leq d < n/2$, and we define $V_{\ge n/2}: = (V_{< n/2})^{\perp}$. Note that each $V_{=d}$ is finite-dimensional, but $V_{\ge n/2}$ is infinite-dimensional.

 For the special unitary group we perform a similar construction, except that one views the complex entries of the input matrix as a pair of real numbers. More precisely, we define $V_{\le d}$ to consist of the functions $f$ that are degree $\leq d$ multivariate polynomials in the real and imaginary parts of the matrix entries of $X \in \SU(n)$ (so, for example, $V_{\leq 3}$ contains the polynomial $\text{Re}(X_{11})\text{Im}(X_{11})\text{Re}(X_{12})$). As in the $\SO(n)$ case, we then set $V_{=0}:=V_{\leq 0}$ and $V_{=d}:=V_{\leq d} \cap \parenth{V_{\leq {d-1}}}^\perp$ for each integer $1 \leq d < n/2$, and we set $V_{\geq n/2}:=(V_{< n/2})^{\perp}$.
 
 For the compact symplectic group, $\Sp(n)$, we view it as the group of $n$ by $n$ unitary matrices over the field of quaternions (see Section \ref{sec:spn-constr} for more details), and we define $V_{\leq d}$ to be the vector subspace of $L^2(\Sp(n))$ spanned by degree $\leq d$ multivariate polynomials in the real parts, the $\mathbf{i}$-parts, the $\mathbf{j}$-parts and the $\mathbf{k}$-parts of the matrix entries; we then proceed as in the previous two cases.
 
 For the spin group $\Spin(n)$ things are a little more involved, as it has no straightforward description as a linear group. In order to define the grading we make use of the double covering homomorphism $\pi \colon \Spin(n) \to \SO(n)$ (recall that $\Spin(n)$ is the universal covering group of $\SO(n)$, and that this cover is a double cover, for each $n \geq 3$). The covering homomorphism $\pi$ gives rise to an embedding $i\colon L^2(\SO(n))\to L^2(\Spin(n))$ given by $i f = f \circ \pi$. We then take the grading of the spin group to be  
 \[
 \Spin(n)=\bigoplus_{d< n/2} i(V_{=d}^{\SO(n)})\oplus (i(V_{<n/2}^{\SO(n)}))^\perp,
 \]
 i.e.\ $V_{=d}^{\Spin(n)}= i(V_{=d}^{\SO(n)})$ for each $0\leq d < n/2$, and $V_{ \geq n/2}^{\Spin(n)} = (i(V_{<n/2}^{\SO(n)}))^\perp$.
 
 We remark that the above arguments also imply that $\SO(n)$, $\SU(n)$, $\Sp(n)$ and $\Spin(n)$ are $m$-graded for all $m \in \mathbb{N}$ and all $n \geq 3$, but we will only require the $n$-grading, in the sequel.

 \medskip Next we define a notion of `strong quasirandomness' for graded compact groups. In section~\ref{section: compacts are quasi} we show that the simply connected $n$-graded groups are all $c$-strongly-quasirandom for some absolute constant $c>0$. 

\remove{\gnote{Add an explanation that this overrides the earlier definition}}
\begin{definition}[Strongly quasirandom graded group]
We say an $n$-graded compact group $G$ is $\left((Q_d)_{d=0}^{\lceil n/2 \rceil-1},Q \right)$-{\em strongly-quasirandom} if the minimal dimension of a subrepresentation of $V_{=d}$ (as a left $G$-module) is $\ge Q_d$ for all integers $0 \leq d\leq \lceil n/2 \rceil-1$, and is $\ge Q$ for $V_{\geq n/2}$.

For $c>0$ we say that the $n$-graded compact group $G$ is {\em $c$-strongly-quasirandom} if it is $((Q_d)_{d=1}^{\lceil n/2\rceil-1 },Q)$-strongly-quasirandom when we set $Q_d := \left(\frac{cn}{d}\right)^d$ for $d < cn/(1+c)$, and $Q,Q_d := (1+c)^{cn/(1+c)}$ for $d\ge cn/(1+c).$ \remove{[David: note that with the previous definition the condition for $d = n/200$ say would have been vacuous for small enough $c$, which would have caused problems e.g.\ when passing to products of groups. The slightly messy definition of $Q$ is likewise necessary to ensure good behaviour on passing to products (of groups).] }
\end{definition}

In Section \ref{thm:Every simple group is quasirandom}, we show that all the (infinite families of) compact simply connected simple Lie groups are $c$-strongly quasirandom for some absolute constant $c >0$.
\begin{theorem}\label{thm:Every simple group is quasirandom}
The $n$-graded compact groups $\SU(n), \Sp(n), \Spin(n)$ (for $n \geq 3$) are all $c$-strongly quasirandom for some absolute constant $c>0$, when equipped with our chosen $n$-grading.
\end{theorem}

\begin{definition}[Beckner operator for graded groups]
Let $G$ be an $n$-graded compact group, let $r$ be an integer with $0 \leq r<n/2$, and let $0 \leq \delta\leq 1$. We define the {\em Beckner operator} $T_{\delta, r}:L^2(G) \to L^2(G)$ by 
$T_{\delta,r}(f):= \sum_{i=0}^r \delta^i f^{=i}$, for all $f \in L^2(G)$.
\end{definition}
 
\begin{definition}[Hypercontractive group]
Let $C>0$ and let $r$ be an integer with $0 \leq r < n/2$. We say that an $n$-graded group compact $G$ is $(r, C)$-\emph{hypercontractive}
if for every $q\geq 2$ and every $0 \leq \delta \leq 1/(C\sqrt{q})$, we have $\norm{T_{\delta,r}}_{2\to q}\leq 1$. 
\end{definition}

The following is an easy consequence of hypercontractivity.
\begin{lemma}\label{lem:follows from hypercontractiveness1}
Let $G$ be an $n$-graded $(r,C)$-hypercontractive group, where $0 \leq r < n/2$. Then for every integer  $d\le r$,  every $q\geq 2$ and function $f\in V_{=d}$, we have \[ \|f\|_q\le  \left(C^2q\right)^{d/2}{\|f\|}_2.\] 
\end{lemma}
\begin{proof}
Let $f\in V_{=d}.$ Then $\delta^{d}\|f\|_q =\|T_{\delta,r}f\|_q \le \|f\|_2$ for all $\delta \leq 1/(C\sqrt{q})$; setting $\delta = 1/(C\sqrt{q})$ completes the proof.
\end{proof}

In Section~\ref{sec:coupling_introduce} we show that the compact simple simply connected $n$-graded Lie groups $\Sp(n),\SU(n)$ and $\Spin(n)$ are $(c\sqrt{n},C)$-hypercontractive for some positive absolute constants $C$ and $c$, when they are eqipped with our chosen $n$-grading.

\begin{theorem}\label{thm:every simple group is hypercontractive}
The groups $\Sp(n),\SU(n),\Spin(n)$ are $(c\sqrt{n},C)$-hypercontractive, when equipped with our chosen $n$-grading, for some positive absolute constants $C$ and $c$.
\end{theorem}

\begin{definition}[Good groups]
An $n$-graded compact group $G$ is said to be {\em $(C,c)$-good} if it is $(cn^{1/2},C)$-hypercontractive and $c$-strongly quasirandom.
\end{definition}

The next theorem  follows from Theorem~\ref{thm:every simple group is hypercontractive} and Theorem~\ref{thm:Every simple group is quasirandom}. 

\begin{thm}\label{thm:Every simple group is good}
   For each $n \geq 3$, the $n$-graded groups $\Sp(n),\SU(n),\Spin(n)$ (equipped with our choice of $n$-grading) are $(C,c)$-good, for some absolute positive constants $C$ and $c$.
\end{thm}

We prove Theorem~\ref{thm:Every simple group is good} in Section~\ref{sec:coupling_introduce}. 

\subsection{Fine groups}

We now introduce the closely related notion of a `fine' group. It involves a weaker form of hypercontractivity than with a `good' group, but we manage to prove it for higher values of $d$, viz., up to linear in $n$. 

\begin{definition}[Weakly hypercontractive group]
Let $C>1$ and $1 \leq r < n/2$. An $n$-graded compact group $G$ is $(r,C)$-\emph{weakly hypercontractive} if 
for every function $f\in L^2(G)$ and every $q\geq 2$ and $0 \leq \delta \leq 1/q^C$ we have \[ \|T_{\delta, r}f\|_q\le  \|f\|_2.\]
\end{definition}

The following lemma follows similarly to Lemma~\ref{lem:follows from hypercontractiveness1}.
\begin{lemma}\label{lem:follows from hypercontractiveness2}
Let $G$ be an $n$-graded $(r,C)$-weakly hypercontractive compact group, where $C>1$ and $1 \leq r < n/2$. Then for any integer $d\leq r$, every $q\geq 2$, and any function $f\in V_{=d}$, we have
\[
\norm{f}_q\leq q^{Cd}\cdot \norm{f}_2
\]
    
\end{lemma}
\begin{definition}[Fine groups]
If $c>0$ and $C>1$, an $n$-graded compact group $G$ is said to be $(C,c)$-\emph{fine} if it is both $(cn,C)$-weakly hypercontractive, and $c$-strongly-quasirandom.
\end{definition}

In Section \ref{sec:hyp_from_curv} we show that the (infinite families of) compact simply connected simple groups, viz.\ $\Sp(n),\SU(n)$, and $\Spin(n)$, are all fine, as $n$-graded groups equipped with our chosen gradings.

\begin{theorem}\label{thm:Every simple groups is fine}
For $n \geq 3$, the $n$-graded groups $\Sp(n),\SU(n)$, and $\Spin(n)$ are $(C,c)$-fine, for some absolute constants $C>1$ and $c>0$, when equipped with our chosen $n$-gradings.
\end{theorem}

\subsection{Goodness and fineness are preserved under taking products and quotients.}
In this subsection, we show that goodness and fineness are preserved under taking products, and quotients (quotients, that is, by closed normal subgroups) --- more precisely, that suitable gradings can be defined on products and quotients. This, together with Theorems~\ref{thm:Every simple group is good} and \ref{thm:Every simple groups is fine}, implies that all compact Lie groups of large-enough min-rank are both good and fine.

\begin{lemma}\label{lem:reductionof goodness to the simple case}
Assuming Theorem \ref{thm:Every simple group is good}, the following holds. There exist positive constants $c,C$ and $n_0$ such that if $n>n_0$, then every compact connected Lie group of min-rank $n$ can be equipped with an $n$-grading that makes it $(C,c)$-good, as an $n$-graded group.
\end{lemma}

\begin{lemma}\label{lem:reduction of fineness to the simple case}
Assuming Theorem \ref{thm:Every simple groups is fine}, the following holds. There exist positive constants $c,C$ and $n_0$ such that if $n>n_0$, then every compact connected Lie group of min-rank $n$ can be equipped with an $n$-grading that makes it $(C,c)$-fine, as an $n$-graded group.  
\end{lemma}

\paragraph{Grading for group products.} Let $G$ and $H$ be compact groups, and let $f$ and $g$ be functions on $G$ and $H$ respectively. We write $f\otimes g$ for the function on $G\times H$ given by $(x,y)\mapsto f(x)g(y)$. If $U$ is a closed linear subspace of $L^2(G)$ and $V$ is a closed linear subspace of $L^2(H)$, then we denote by $U\otimes V$ be the Hilbert-space tensor product of $U$ and $V$, i.e.\ the linear subspace of $L^2(G\times H)$ consisting of the closure of the linear span of the set of functions $\{f\otimes g:\ f \in U,\ g \in V\}$. We recall that if $\{u_i\}_{i=1}^{\infty}$ is a Hilbert-space basis for $U$ and $\{v_i\}_{i=1}^{\infty}$ is a Hilbert-space basis for $V$, then $\{u_i \otimes v_j\}_{i,j=1}^{\infty}$ is a Hilbert-space basis for $U \otimes V$.

\begin{definition}
Let $n\ge m$, let $G$ be an $n$-graded compact group and $H$ be an $m$-graded compact group. We give $G\times H$ the structure of an $m$-graded compact group by setting 
\[V_{=d}^{G\times H}  = \bigoplus_{d_1 + d_2 =d} V_{=d_1}^{G}\otimes V_{=d_2}^{H}\]
for $d<m/2$, and $V^{G\times H}_{\geq m/2}=(V^{G\times H}_{\le \lceil m/2\rceil -1})^{\perp}$.
\end{definition}

For products of more than two compact groups, $G_1 \times G_2 \times \ldots \times G_{\ell}$ say, we simply iterate the above definition (noting associativity). Viz., if $G_i$ is $n_i$-graded for $i=1,2,\ldots,\ell$, then we let $n = \min\{n_1,\ldots,n_{\ell}\}$ and we give $G_1 \times \ldots \times G_{\ell}$ the structure of an $n$-graded compact group by setting
\[V_{=d}^{G_1 \times \ldots \times G_{\ell}}  = \bigoplus_{d_1 +\ldots+d_{\ell}=d} V_{=d_1}^{G_1}\otimes\ldots \otimes V_{=d_{\ell}}^{G_{\ell}}\]
for $d<n/2$, and $V^{G_1 \times \ldots \times G_{\ell}}_{\geq n/2}=(V^{G_1 \times \ldots \times G_{\ell}}_{\le \lceil n/2\rceil -1})^{\perp}$.

Our goal is now to show that the product of good groups is good. We make use of the following lemma of Beckner~\cite{Beckner}.
\remove{
\begin{lemma}\label{lem:hypercontractivity is equivalent to the contraction of a certain operator}
Let $C>1,$  $G$ be an $n$-graded compact group and $r\le n/2$. Define operators $T,S$ on $L^2(G)$ given by \[Tf = \sum_{i=0}^{r} (Cq)^{-i/2}f^{=i},\] and \[T'f = \sum_{i=0}^{r} (2Cq)^{-i/2}f^{=i}.\] If $\|T\|_{2\to q}\le 1$, then $G$ is $(r,C)$-hypercontractive. Conversely, if $G$ is $(r,C)$-hypercontractive,  then $\|T'\|_{2\to q}\le 1.$
\end{lemma}
\begin{proof}
The first direction is obvious. Indeed, for each $f\in V_{=i}$ we have \[(Cq)^{-i/2}\|f\|_q = \|Tf\|_q \le \|f\|_2.\]. We may now rearrange to get $\|f\|_q\le (Cq)^{i/2}\|f\|_2.$

As for the other direction, let $f\in L^2(G)$. 
Then by Cauchy-Schwarz
\[
\norm{T'f}_q \leq \sum_{i=0}^r \norm{T'f^{=i}}_q\leq
\sum_{i=0}^r  (4Cq)^{-i/2} \cdot (Cq)^{i/2}\norm{T'f^{=i}}_2
\le \sum_{i=0}^r 2^{-i} \sum 
\]
\end{proof}

The above formulation will allow us to use the following lemma due to Beckner \cite{Beckner}.
}
\begin{lemma}\label{lem:Beckner}
Let $X_1,\ldots, X_r, Y_1, \ldots, Y_r$ be probability spaces. Let $T_1,\ldots T_r$ be linear operators such that $T_i\colon L^2(X_i)\to L^q(Y_i)$ for all $i \in [r]$. Then $\|T_1\otimes \cdots \otimes T_r\|_{2\to q}\le \prod_{i=1}^{r}\| T_i\|_{2\to q}$. 
\end{lemma}

\begin{lemma}\label{lem:hypercontractivity is preserved for products}
    Let $G_1,\ldots, G_l$ be compact groups and suppose that each $G_i$ is $n_i$-graded. Write $n=\min \{n_1, \ldots, n_l\}$. Suppose that each $G_i$ is $(r,C)$-hypercontractive, where $r \in \mathbb{N}$ with $r < n/2$; then $G:=\prod_{i=1}^{l} G_i$ (with the above $n$-grading) is also $(r,C)$-hypercontractive.   
\end{lemma}

\begin{proof}
    Let $q\geq 2$, and let $\delta=\frac{1}{C\sqrt{q}}.$ Let $T_i\colon L^2(G_i)\to L^2(G_i)$, $T\colon L^2(G)\to L^2(G)$ be the appropriate $T_{\delta,r}$ Beckner operators. Then by hypothesis $\|T_i\|_{2\to q}\le 1$ for all $i \in [r]$; our goal is to show that $\|T\|_{2\to q}\le 1.$ This now follows from Lemma \ref{lem:Beckner}  and the fact that $T$ can be decomposed as $T=T_1\otimes \cdots \otimes T_l \circ S,$ where $S$ is given by $f\mapsto f^{\le r}.$ We have 
    \[\| T f\|_{q} \le \|T_1\otimes \cdots \otimes T_l\|_{2\to q} \|Sf\|_2 \le \|f\|_2\] 
    for each $f$. 
\end{proof}

A similar argument works for weakly hypercontractive groups, yielding the following lemma.

\begin{lemma}\label{lem: weak hypercontracti vity is preserved for products}
    Let $G_1,\ldots, G_l$ be compact groups and suppose that each $G_i$ is $n_i$-graded. Write $n=\min \{n_1, \ldots, n_l\}$. Suppose that each $G_i$ is $(r,C)$-weakly hypercontractive, where $0 \leq r < n/2$. Then the $n$-graded group $G:=\prod_{i=1}^{l} G_i$ (with the above $n$-grading) is also $(r,C)$-weakly hypercontractive.  
\end{lemma}

\paragraph{Grading for quotients.} We now define a grading for quotients of graded groups, and show that hypercontractivity (and weak hypercontractivity) is preserved under taking quotients, using this `induced' grading.
\begin{definition}
    Let $G$ be an $n$-graded compact group, let $H$ be a closed normal subgroup of $G$, and let $\pi\colon G\to G/H$ be the quotient map. The $n$-graded structure on $L^2(G/H)$ induced from that on $L^2(G)$ is given by letting $V_{=d}^{G/H}$ consist of all functions $f$ in $L^2(G/H)$ such that $f\circ \pi$ is in $V_{=d}^{G}$. The space $V_{\ge n/2}^{G/H}$ is defined similarly.  
\end{definition}
We note that this definition is consistent with our choices of the gradings of $\Spin(n)$ and of $\SO(n)$.

\begin{lemma}
    The subspaces $V_{=d}^{G/H}$ constitute an $n$-grading of the compact group $G/H$. 
\end{lemma}  
\begin{proof}
    We first note that if
    \[L^2(G)=\left(\bigoplus^{\lceil n/2 \rceil -1 }_{d=0} V^G_{=d} \right)\oplus V_{\geq n/2}\]
    is a grading of $L^2(G)$, then each $V_{=d}^{G}$ is closed (being an orthogonal complement of a subspace), as is $V_{\geq n/2}^G$.
    
    Let $i\colon L^2(G/H)\to L^2(G)$ be given by $i(f)=f\circ \pi$.  We let $i^*$ be its adjoint, which is given explicitly by $i^*(f)(xH) = \mathbb{E}_{h \in H
}[f(xh)]$, where the expectation is taken with respect to the Haar probability measure on $H$. Now $i^*$ commutes with the action of $G$ (from either the left or the right) and therefore the spaces $i^*(V_{=d}^{G})$ and $i^*(V_{\geq n/2}^G)$ are invariant under both the left and the right actions of $G$. Since $i^*$ preserves orthogonality, these spaces are also pairwise orthogonal. Since $i^* \circ i$ is the identity (so $i^*(L^2(G)) = L^2(G/H)$), we obtain that the spaces $i^*(V_{=d}^G), i^*(V_{>n/2})$ constitute a grading of $L^2(G/H)$. The fact that $i^*\circ i$ is the identity also imiplies that $V_{=d}^{G/H}\subseteq i^*(V_{=d}^{G})$ for each $d < n/2$, and that $V_{\geq n/2}^{G/H} \subseteq i^*(V_{\geq n/2}^{G})$. We now claim that $i^*(V_{=d}^G)=V_{=d}^{G/H}$ for all $d < n/2$, and that $i^*(V_{\geq n/2}^G) = V_{\geq n/2}^{G/H}$. To prove this, it suffices to show that $V_{=d}^{G}$ is invariant under $i\circ i^*$ for each $d < n/2$, and that $V_{\geq n/2}^{G}$ is invariant under $i\circ i^*$. (Indeed, the latter implies that $i^*(V_{=d}^{G})\subseteq V_{=d}^{G/H}$ for each $d < n/2$, and that $i^*(V_{\geq n/2}^{G})\subseteq V_{\geq n/2}^{G/H}$.) Now, $i \circ i^*$ is given by $f\mapsto (x \mapsto \mathbb{E}_{h\sim H}[f(xh)])$. Suppose that $f \in V_{=d}^G$. By the right-invariance, each function $f(xh)$ is then in the space $V^G_{=d}$, and as $V^G_{=d}$ is closed, it follows that the average $i\circ i^*$ also lies in $V^G_{=d}$. Exactly the same argument works with $V_{\geq n/2}^G$, proving the claim. This completes the proof of the lemma. 
\end{proof}

\begin{lemma}\label{lem:quotient-hyper}
    Let $G$ be an $n$-graded compact group, and let $H$ be a closed normal subgroup of $G$. Suppose that $G$ is $(r,C)$-(weakly) hypercontractive, where $0 \leq r < n/2$. Then $G/H$ is also $(r,C)$-(weakly) hypercontractive, when equipped with the induced $n$-grading defined above.   
\end{lemma}
\begin{proof}
    Let $i\colon L^2(G/H)\to L^2(G)$ be given by $i(f)=f\circ \pi$, as before. We notice that by definition of the grading and of the Beckner operator, it holds that $T^G_{\delta,r}\circ i = i \circ T^{G/H}_{\delta,r}$. Noting that the map $i$ is an $L^p$-isometric embedding for all $p$, we have that for any $f\in L^2(G/H)$,  
    \[
    \norm{T^{G/H }f}_q= \norm{i\circ T^{G/H}f}_q=\norm{T^G(i(f))}_q\leq \norm{i(f)}_2 = \norm{f}_2.
    \]
    This completes the proof of the lemma.
\end{proof}

If the map $i$ from earlier induces an isomorphism between the spaces $V_{=d}^{G/H}$ and $V_{=d}^{G}$ for all $d \leq r$, then the converse of Lemma~\ref{lem:quotient-hyper} also holds; this will be useful for going from $\SO(n)$ to $\Spin(n)$.

\begin{lemma}\label{lem:quotient-inverse-hyper}
    Let $G$ be an $n$-graded compact group, and let $H$ be a closed normal subgroup of $G$. Equip $G/H$ with the induced (quotient) grading, defined above. Suppose that $G/H$ is $(r,C)$-(weakly) hypercontractive as an $n$-graded group, where $0 \leq r < n/2$. Suppose further that for all $d\le r$, the gradings satisfy $i\left(V_{=d}^{G/H}\right)=V_{=d}^G$. Then $G$ is also $(r,C)$-(weakly) hypercontractive, as an $n$-graded group.   
\end{lemma}
\begin{proof}
    Let us denote the Beckner operator $T_{\delta , r}$ on $L^2(G)$ by $T$, and the corresponding operator on $L^2(G/H)$ by $T'$.  Then we may write $T \circ i  = i \circ T'$. Composing with $i^*$ we obtain 
    \[
        T \circ i \circ i^* = i\circ T'\circ i^*.
    \]
    We now claim that $T\circ i \circ i^* = T$. First note that each space $V^{G}_{=d},V^{G}_{>r}$ is  $i\circ i^*$-invariant.
    We can therefore use fact that the operator $T$ annihilates $V^{G}_{>r}$ to deduce that the operator 
    $T\circ i \circ i^*$ agrees with $T$ on $V^{G}_{>r}$. Our claim will follow once we show that $i\circ i^*$ is the identity on $V^{G}_{\le r}$. 
    To accomplish that we note that $i$ is injective, and  by the hypothesis the restriction of $i$
    to the corresponding $V_{=d}$
    spaces is also surjective, and thus so is its restriction to 
    $V^G_{\leq r}$. As the operator  $i^* \circ i$ is the identity we 
    obtain that the restriction of $ i\circ i^*$ 
    to  $V_{\le r}^{G}$ is the identity as well.    
    
    We can therefore write $T= i \circ T'\circ i^*$, and using the fact that $i$ is an $L^q$-isometric embedding and that $i^*$ contracts 2-norms we obtain:  
    \[ 
    \|T\|_{2\to q} = \|T'\circ i^*\|_{2\to q} \le \|T'\|_{2\to q}\|i^*\|_{2\to 2}\le 1.
    \]
\end{proof}

\paragraph{Conclusion.} We have shown that (weak) hypercontractivity is preserved under taking products and quotients. It is easy to check that $c$-strong-quasirandomness is also preserved under taking products or quotients (using the gradings above), so we immediately obtain Lemma~\ref{lem:reductionof goodness to the simple case}
and Lemma~\ref{lem:reduction of fineness to the simple case}.

\section{Growth in good groups}

\remove{\gnote{add reference:
@misc{non-abelian-brunn-minkowski,
Author = {Yifan Jing and Chieu-Minh Tran and Ruixiang Zhang},
Title = {A nonabelian Brunn-Minkowski inequality},
Year = {2021},
Eprint = {arXiv:2101.07782},
}}}
In this section we prove Theorems \ref{thm:son-stronger}, \ref{thm:second-level-d}, \ref{thm:spectral gap}, \ref{thm:mixing time 2} \ref{thm:product mixing intro} and \ref{thm:mixing}. For now, the reader may consider the objective of proving Theorem \ref{thm:son-stronger} as motivation for what follows. 

Let $G$ be a compact group, and let $\mu$ be the Haar probability measure on $G$. We would like to bound $\mu(A)$ for a set $A\subseteq G$ that is product free. We first note that the property of being product free can be stated in terms of convolutions.

\begin{definition}
For two functions $f,g\in L^2(G)$, we define their {\em convolution} $f*g\in L^2(G)$ by 
\[
f*g(x):=\int f(xy^{-1})g(y) d\mu(y).
\]
\end{definition}

\newcommand{\ang}[1]{{\langle #1 \rangle}}

For $f\in L^2(G)$, we write $T_f$ for the linear operator from $L^2(G)$ to itself defined by $g\mapsto g*f$. 
Observe that if $A \subset G$ is a product-free set of density $\mu(A)=\alpha$, and $f=\frac{1_A}{\alpha}$, then $\langle T_f 1_A,1_A\rangle =0$. If $G$ is an $n$-graded group, we can decompose $g:=1_A$ into its orthogonal projections onto the $V_{=d}$'s, and write  $g=\sum_{d=0}^{\lceil n/2\rceil -1} g^{=d} + g^{\geq n/2}$, where $g^{=d}$ is the orthogonal projection of $g$ onto $V_{=d}$. Noting that $g^{=0}\equiv \alpha$, this allows us to expand 
\[\langle T_f 1_A,1_A \rangle = \ang{T_f g, g } \] 
as a sum of a main term, $\alpha^2$, and other terms of the form 
$\langle T_f g^{=d},g^{=d} \rangle$ or $\langle T_f g^{\ge n/2}, g^{\ge n/2} \rangle$. 
We upper-bound each term using Cauchy-Schwarz:
\[|\langle T_f g^{=d},g^{=d} \rangle| \leq \|T_f\|_{V_{=d}} \|g^{=d}\|_2^2,\] where for a closed subspace $M\leq  L^2(G)$ and a linear operator $T:L^2(G) \to L^2(G)$ we write $\|T\|_M$ for the supremum of $\frac{\|Tv\|_2}{\norm{v}_2}$ over all nonzero $v\in M$.  

Our goal will be to show that these other term make a negligible contribution to the sum compared to the main term $\alpha^2$.  We accomplish that by observing that the space $V_{=d}$ is $T_{f}$-invariant. This shows that the operator $T_{f}^*T_f$ can be diagonalized inside $V_{=d}.$ It also implies that $\|T_f\|_{V_{=d}}^2$ is equal to the maximal eigenvalue of $T_f^*T_f$ inside $V_{=d}$. We then upper bound the maximal eigenvalue of $T_f^* T_f$ inside $V_{=d}$, showing  that these eigenvalues get smaller and smaller as the degree gets larger. Finally, we combine our upper bound on the eigenvalues of $T_f$ with a level $d$-inequality which shows that the $L^2$-mass of $g$ is concentrated on the high degrees. Together, we obtain that the sum of terms $\langle T_f g^{=d},g^{=d} \rangle$ is  indeed negligible.  

Our upper bound on the `degree $d$' eigenvalues of $T_f^* T_f$ follows by combining a level $d$ inequality with a lower bound on the dimension of each eigenspace of $T_f^* T_f.$ We use the fact that each such eigenspace is a subrepresentation of $V_{=d}$ and therefore by strong quasirandomness must have dimension $\ge Q_d,$ for the appropriate quasirandomness parameter $Q_d.$ We upper bound $|\langle T_f g^{\ge n/2}, g^{\ge n/2} \rangle|$ in a similar fashion.

\subsection{Level \boldmath\texorpdfstring{$d$}{d} inequalities and the eigenvalues of convolution operators}
\remove{We now recall some classical facts from the representation theory of compact groups. The Peter-Weyl theorem states that if $G$ is a compact group, equipped with the Haar measure, then $L^2(G)$ has the following decomposition as an orthogonal direct sum:
$$L^2(G) = \widehat{\bigoplus}_{\rho \in \hat{G}}W_{\rho},$$
where $\hat{G}$ denotes a complete set of irreducible unitary representations of $G$ (complete, meaning, with one irreducible representation from each equivalence class), and $W_{\rho}$ is the subspace of $L^2(G)$ spanned by  functions of the form $g \mapsto v^t(\rho(g))u$. These functions are called the \emph{matrix coefficients} of $\rho$. We call the subspaces $W_{\rho}$ the Peter-Weyl ideals of $L^2(G)$. The space $W_{\rho}$ can be decomposed to irreducible left-representations. Each of them is isomorphic to the dual representations $g\mapsto \rho(g^{-1})^{\text{tr}}.$ 

For $f \in L^2(G)$ we write $f^{=\rho}$ for the projection of $f$ onto $W_{
\rho}$. The Peter-Weyl theorem also says that $(f*g)^{=\rho} = f^{=\rho}*g^{=\rho}$ for every $\rho \in \hat{G}$. It also implies that every closed subspace of $L^2(G)$ that commutes with the action of $G$ from both sides can be decomposed as a direct sum of  isotypical components $W_{\rho}$. Therefore there exists a set $\mathcal{L}_d$ of representations of $G$ for which \[V_{=d} =  \widehat{\bigoplus}_{\rho \in \mathcal{L}_d} W_{\rho}.\] If $\rho\in \mathcal{L}_d,$ then we say that the \emph{level} of $\rho$ is $d$. 

Finally, we will also need the following fact 
}
Recall that the {\em Hilbert-Schmidt norm} of a linear operator $T$ on a separable Hilbert space $H$ is defined by
$$\|T\|_{\text{HS}}^2: = \sum_{i=1}^{\infty}\|T(e_i)\|^2_2,$$
where $\{e_i\}_{i=1}^{\infty}$ is any Hilbert-space basis for $H$; if $T$ is a compact operator, then $\|T\|_{\text{HS}}$ is the square root of the sum of the eigenvalues of $T^*T$ (counted with multiplicity). One standard fact (see e.g.\ \cite{conway}, page 267) is the following.

\begin{fact}\label{fact:HS norm of convolution}
    Let $f\in L^2(G)$ and let $T_f$ be the linear operator from $L^2(G)$ to itself defined by $g \mapsto g*f.$ Then $T_f$ is a compact operator, and the Hilbert--Schmidt norm of $T_f$ is equal to the 2-norm of $f$: \[\|T_f\|_{\text{HS}}=\|f\|_2.\]
\end{fact}
We recall our notation for  the norm of an operator on a subspace of $L^2(G)$. 
\begin{definition}
    Let $M\leq L^2(G)$ be a closed subspace. Let $T\colon L^2(G)\to L^2(G)$ be a linear operator; then we write $\|T\|_M: = \sup\{\|Tf\|_2/\|f\|_2:\ f\in M \setminus \{0\}\}$.   
\end{definition}

We now give our upper bound on the level $d$ eigenvalues of $T_f$.
 
\begin{lemma}
\label{lemma:bnp}
Let $G$ be an $n$-graded $((Q_d)_{d=1}^{n/2},Q)$-quasirandom group. Let $f\in L^2(G)$. Then the spaces $V_{=d},V_{\geq n/2}$ are $T_f$-invariant.  Moreover,  ${\|T_f\|}_{V_{=d}} \le \frac{\|f^{=d}\|_2}{\sqrt{Q_d}}$ and ${\|T_f\|}_{V_{\geq n/2}} \le \frac{\|f\|_2}{\sqrt{Q}}$. 
\end{lemma}

\begin{proof}
     We first claim that the subspaces $V_{=d}, V_{\geq n/2}$ are all $T_f$ invariant. To see this, observe that if $U\le L^2(G)$ is a closed subspace that is invariant under the right-action of $G$, then for every $g\in U$, we have $g*f\in U$ as well. (Indeed, let $h \in U^{\perp}$; then
     $$\langle g*f,h \rangle = \int \int g(x y^{-1})f(y)h(x)d\mu(y)d\mu(x) = \int \int g(xy^{-1})f(y)h(x)d\mu(x)d\mu(y)=0,$$
     using Fubini and the fact that for each fixed $y \in G$ the function $x\mapsto g(xy^{-1})f(y)$ lies in $U$. Hence, $g*f \in (U^{\perp})^{\perp}=U$.) Applying this with $U=V_{=d}$, which is a closed subspace invariant under the right action of $G$, we see that the spaces $V_{=d}$ are indeed $T_f$-invariant. Similarly, applying it with $U = V_{\geq n/2}$ (which is also a closed subspace invariant under the right action of $G$), we see that $V_{\geq n/2}$ is also $T_f$-invariant.
     
     Fix $d < n/2$, and let us orthogonally decompose $f$ as  $f=f^{=d}+f'$. Then $T_f = T_{f^{=d}}+ T_{f'}$, by the linearity of convolution. We now assert that $T_{f}$ agrees with $T_{f^{=d}}$ on $V_{=d}$. Essentially the same argument as that above shows that if $U \leq L^2(G)$ is a closed subspace that is invariant under the left-action of $G$, then $g*f' \in U$ for every $f' \in U$ and $g \in L^2(G)$; applying this with $U=V_{=d}^{\perp}$ , we obtain that $T_{f'}g \in V_{=d}^{\perp}$ for every $g \in L^2(G)$. Hence, if $g\in V_{=d}$, then $T_{f'}g=0$, proving our assertion. It follows that ${\|T_f\|}_{V_{=d}}^2$ is the maximal eigenvalue of the operator $T_{f^{=d}}^*T_{f^{=d}}$. On the other hand, Fact \ref{fact:HS norm of convolution} implies that $\|f^{=d}\|_2^2$ is the sum of the eigenvalues of $T_{f^{=d}}^*T_{f^{=d}}$ counted with multiplicity. To complete the proof that ${\|T_f\|}_{V_{=d}} \le \frac{\|f^{=d}\|_2}{\sqrt{Q_d}}$ we show that each such multiplicity is $\ge Q_d$ and so 
     \[Q_d{\|T_f\|}_{V_{=d}}^2 \le \|f^{=d}\|_2^2.\] 
     To lower-bound the multiplicities of the eigenvalues of $T_{f^{=d}}^*T_{f^{=d}}$ inside $V_{=d}$ we note that $T_{f^{=d}}^*T_{f^{=d}}$ commutes with the left-action of $G$, since $T_{f^{=d}}$ does. This implies that the eigenspaces of $T_{f^{=d}}^*T_{f^{=d}}$ are left $G$-submodules of $V_{=d}$. Each such submodule contains an irreducible representation, which has dimension $\ge Q_d$ by hypothesis. 
     The proof that $\|T_f\|_{V_{\geq n/2}} \le \frac{\|f\|_2}{\sqrt{Q}}$ is similar.
\end{proof}

We now upper-bound ${\|T_f\|}_{V_{= d}}$ by proving a corresponding level $d$-inequality.

\begin{thm}\label{thm: level-d in SO_n}
     Let $G$ be an $(r,C)$-hypercontractive group. Let $f\colon G\to \{0, 1\}$ be measurable, and write $\alpha:=\mathbb{E}[f]$. Let $d\in \mathbb{N}$ be such that $0<d\le \min \{\tfrac{1}{2}\log (1/\alpha), r\}$. Then 
     \[
     \|f^{= d}\|_{2}^2 \le \left(\frac{10C^2}{d}\right)^d\alpha^2\log^{d}(1/\alpha ).
     \]
\end{thm}
\begin{proof}
    Let $q = \frac{\log(1/\alpha)}{d}$; note that $q \geq 2$. Let $q'$ be the H\"{o}lder conjugate of $q$, i.e., $q'$ is defined by $\frac{1}{q}+\frac{1}{q'}=1$. Then by $(r,C)$-hypercontractivity and Lemma~\ref{lem:follows from hypercontractiveness1}, we have \[\|f^{=d}\|_2^2 = \langle f, f^{=d} \rangle \le \|f^{=d}\|_q \|f\|_{q'} \le (C^2q)^{d/2} \|f^{=d}\|_2\cdot \alpha^{1-1/q}.\]
    Rearranging, we obtain 
    \[ \|f^{=d}\|_2^2 \le (e^2C^2q)^d \alpha^2. \] The theorem follows by plugging in the value of $q$.
\end{proof}

Theorem \ref{thm: level-d in SO_n} together with Lemma \ref{lem:reductionof goodness to the simple case} also shows that Theorem \ref{thm:Every simple group is good} implies Theorem \ref{thm:second-level-d}.

Combining Lemma \ref{lemma:bnp} with Theorem \ref{thm: level-d in SO_n}, we have the following.

\begin{lemma}\label{lemma:conv bound from goodness}
      For each $c,C>0$ there exist $c',n_0>0$ such that the following holds. Let $n>n_0$ and let $G$ be an $n$-graded $(C,c)$-good group. Let  $A\subseteq G$ be measurable, and suppose that \[\alpha := \mu_G(A)\in (  e^{-c'\sqrt{n}}, c').\] Write  $f=\frac{1_A}{\alpha}$ and $t=\frac{\log(1/\alpha)}{2}$. Then for all $1\leq d\le  t$, we have
      \[{\|T_f\|}_{V_{=d}} \le \left(\frac{C'\log(1/\alpha)}{n}\right)^{d/2},\] 
      where $C':=\frac{10C^2}{c}$. Moreover,
      \[{\|T_f\|}_{V_{>t}} \le \left(\frac{\alpha}{n}\right)^{10}.\] 
\end{lemma}

\begin{proof}
The lemma follows by applying Lemma~\ref{lemma:bnp} with the values of $Q_d$ and $Q$ which are promised by the goodness of $G$, and then bounding $\norm{f^{=d}}_2$ using Theorem~\ref{thm: level-d in SO_n}.

Let us begin with the range $d\leq t$. Since we know that $G$ is $c$-strongly-quasirandom, he have that $G$ is $((Q_d)_{d=1}^{\lceil n/2\rceil -1},Q)$ graded where $Q_d\geq \left(\frac{cn}{d}\right)^d$ for all $d\leq t$. Hence, by Lemma~\ref{lemma:bnp}, we have
\begin{align*}
\norm{T_f}_{V_{=d}}&\leq \frac{\norm{f^{=d}}_2}{\sqrt{Q_d}} \\ & \leq \norm{f^{=d}}_2\cdot \left({\frac{cn}d}\right)^{-d/2} \\ & = \frac{\norm{\alpha\cdot f^{=d}}_2}{\alpha}\cdot \left({\frac{cn}d}\right)^{-d/2} \\ & \leq \left(\frac{d}{cn}\right)^{d/2}\cdot \left(\frac {10C^2}{d}\right)^{d/2}\log(1/\alpha)^{d/2}
\intertext{(using Lemma~\ref{thm: level-d in SO_n})}
&\leq \left(\frac{10C^2\log (1/\alpha)}{cn}\right)^{d/2},
\end{align*}
which is the desired bound. 

Next, consider the range $d\geq t$. For $d$ in this range, and provided that $n_0$ is sufficiently large, we have from $c$-strong-quasirandomness that $Q_d \ge (\frac{cn}{t})^t$. By Lemma \ref{lemma:bnp} we therefore have, by a similar computation to before, 
\[
{\|T_f\|}_{V_{>t}}\le \|f\|_2\left(\frac{cn}{t}\right)^{-t/2}\le n^{-t/4} e^{-21t}\alpha^{-1/2} \le  \left(\frac{\alpha}{n}\right)^{10},
\]
where we used 
\[
\frac{cn}{t}\ge e^{42}n^{1/2},
\]
which holds provided $c'$ is sufficiently small. 
\end{proof}

\begin{lemma}\label{lemma:conv bound from goodness for large sets}
      Let $c,C,c'>0$. Let $G$ be an $n$-graded $c$-strongly-quasirandom group. Let $A\subseteq G$ be measurable, and suppose that \[\alpha := \mu_G(A) \geq  c'.\] Write  $f=\frac{1_A}{\alpha}$. Then for $1\leq d < cn/(1+c)$ we have
      \[{\|T_f\|}_{V_{=d}} \le {(c')}^{-1/2}\left(\frac{d}{cn}\right)^{d/2},\] 
      and for $d\ge cn/(1+c)$ we have \[{\|T_f\|}_{V_{=d}} \le {(c')}^{-1/2}(1+c)^{-cn/(2(1+c))}.\]
\end{lemma}
\begin{proof}
The lemma follows by applying Lemma~\ref{lemma:bnp} with the values of $Q_d$ and $Q$ that are guaranteed by the $c$-strong-quasirandomness of $G$, and then upper-bounding $\|f^{=d}\|_2$ using $\|f^{=d}\|_2\le \|f\|_2 =\alpha^{-1/2}\le {(c')}^{-1/2}$.

\end{proof}

Lemmas \ref{lem:reductionof goodness to the simple case}, \ref{lemma:conv bound from goodness} and \ref{lemma:conv bound from goodness for large sets} together show that Theorem \ref{thm:Every simple group is good} implies Theorem \ref{thm:spectral gap}.

The following is a version of Theorem \ref{thm: level-d in SO_n} that is perhaps easier to comprehend.

\begin{thm}\label{thm: version 2 of level-d in SO_n}
     For each $c,C>0$ there exist $c',C'>0$ such that the following holds. Let $G$ be an $n$-graded, $(c,C)$-good group, let $f\colon G\to \{0, 1\}$ be measurable,
      and suppose that $\alpha:=\mathbb{E}[f] \geq e^{-c'\sqrt{n}}$. Then for all $d \in \mathbb{N}$, we have
     \begin{equation}\label{eq:version 2 of level-d in SO_n}
     \|f^{= d}\|_{2}^2 \le  (C' )^d \alpha^2\log^{d}(e/\alpha).
     \end{equation}
\end{thm}
\begin{proof}
    Provided $C'$ is sufficiently large depending on $c'$, we may (and shall) assume that $\alpha \le c'.$ Indeed, for $\alpha \ge c'$ the right-hand side is greater than one, and we always have the trivial upper bound $\|f^{\le d}\|^2_2 \le \|f\|^2_2 =\alpha$. Let $r=c\sqrt n$, and note that in the case where $d>r$, the claimed upper bound is trivial, as in this case the right-hand side of (\ref{eq:version 2 of level-d in SO_n}) is (again) greater than one.
    
    Let $q = \log(1/\alpha).$ Let $q'$ be the H\"{o}lder conjugate of $q$. For $d\leq r$, we have by $(r,C)$-hypercontractivity and Lemma~\ref{lem:follows from hypercontractiveness1} that \[ \|f^{=d}\|_2^2 = \langle f, f^{=d} \rangle \le \|f^{=d}\|_q \|f\|_{q'} \le (C^2q)^{d/2} \|f^{=d}\|_2\alpha^{1-1/q}.\]
    After rearranging we obtain 
    \[ \|f^{=d}\|_2^2 \le e^2(C^2q)^d \alpha^2=e^2C^{2d}\alpha^2 \log^{d}(1/\alpha).\] 
    This gives the claimed upper bound, provided $C'$ is sufficiently large depending on $C$. 
     
\end{proof}

\subsection{Upper bounds on the measures of product-free sets.}
Let us now show how Lemma~\ref{lemma:conv bound from goodness} implies an upper bound on the measure of a product-free set in a good group. 

\begin{theorem}\label{thm:product free sets in good groups}
    For any $c,C>0$, there exist $c',n_0 >0$ such that the following holds. Let $n>n_0$ and let $G$ be a $(c,C)$-good $n$-graded group. Then every measurable product-free set in $G$ has Haar measure at most $e^{-c' n^{1/3}}$.
\end{theorem}

Before proving Theorem~\ref{thm:product free sets in good groups}, let us note that together with Theorem~\ref{thm:Every simple group is good} it implies Theorem~\ref{thm:son-stronger}. This follows since, by  Lemma~\ref{lem:reductionof goodness to the simple case}, for all $n > n_0$ every compact connected Lie group of min-rank $n$ is a $(c,C)$-good $n$-graded group for some absolute constants $c,C>0$. 

\begin{proof}
    Let $A \subseteq G$ be product-free and measurable; write $\alpha= \mathbb{E}[1_{A}]$ and $t=\frac{\log(1/\alpha)}{2}$. Assume w.l.o.g.\ that $\alpha<c'$, where $c'$ is to be chosen later (if not, then replace $A$ by a smaller product-free set). Let $f=\frac{1_A}{\alpha}$.  Suppose for a contradiction that $\alpha \ge e^{-c'n^{1/3}}$. We have
    \begin{equation}\label{eq:fourier}
    0 = \langle T_f f ,f\rangle = \mathbb{E}^3[f] + \sum _{d=1}^{\lfloor t \rfloor} \langle T_f f^{=d},f^{=d}\rangle + \langle T_f f^{ > t },f^{> t} \rangle.\end{equation}
    Provided $c'$ is sufficiently small, we may now apply Lemma \ref{lemma:conv bound from goodness} and Theorem \ref{thm: level-d in SO_n} with $1_A= \alpha \cdot f$ to obtain, for $C'$ sufficiently large depending on $c$, $C$ and all $1\leq d\le t$,
    \[ | \langle T_f f^{=d},f^{=d} \rangle | \le \|T_f\|_{V_{=d}}\|f^{=d}\|_2^2 \le \left(\frac{C'\log^3(1/\alpha)}{nd^2}\right)^{d/2} \le 100^{-d}, \]
where the last inequality holds provided $c'$ is sufficiently small depending on $C'$. We may also apply Lemma  \ref{lemma:conv bound from goodness} to obtain 
    \[ | \langle T_f f^{> t}, f^{> t} \rangle | \le \|T_f\|_{V_{> t}} \|f\|_2^2 \le \left(\frac{\alpha}{n}\right)^{10}\alpha^{-1}. \]
    As $\mathbb{E}[f]=1$, these two upper bounds contradict (\ref{eq:fourier}). 
    \end{proof}

\subsection{Product mixing}

The proof of Theorem~\ref{thm:product free sets in good groups} in fact  gives the following stronger statement, which implies Theorem \ref{thm:product mixing intro}.

\begin{theorem}\label{thm: Product mixing}
    For any $\epsilon,c,C>0$, there exist $c',n_0>0$ such that the following holds. Let $n>n_0$ and let $G$ be a $(c,C)$-good $n$-graded group. Let $A,B,C$ be measurable subsets of $G$, each with Haar measure at least $e^{-c'n^{1/3}}$. Let $f=\frac{1_A}{\mu(A)},g=\frac{1_B}{\mu(B)},h = \frac{1_C}{\mu(C)}$. Then
    \[|\langle f*g,h \rangle -1| < \epsilon.\]  
\end{theorem}
\begin{proof}
Assume without loss of generality that $B$ has the smallest measure of the three sets. (Note that, while the trilinear form $\mathcal{T}(f,g,h): = \langle f * g,h\rangle$ is not quite symmetric with respect to permuting $f,g$ and $h$, we may swap the positions of $f$ and $g$ or of $g$ and $h$ if we replace some of $A$, $B$ and $C$ by their inverses, meaning\ $A^{-1}: = \{x^{-1}:\ x \in A\}$ etc, which have the same measures. So there is indeed no loss of generality in assuming the above.) Write $\mu(A)=\alpha, \mu(B) =\beta , \mu(C) = \gamma$ and $C'=\frac{10C^2}{c}$. 

 Note that 
$\langle f*g,h\rangle = \langle T_g f  , h \rangle$. First we quickly handle the case where $\beta \ge c'$. Here we may apply Lemma~\ref{lemma:conv bound from goodness for large sets} to obtain that $\|T_g-I_0\|_{2\to 2} < \epsilon c'$, where $I_0$ denotes operator $F \mapsto \mathbb{E}[F]$ which sends a function to the constant function of the same expectation, provided that $n_0$ is sufficiently large. Using the fact that $\mathbb{E}[g]=1$ we then have \[|\langle {T_g f, h} \rangle -1| = |\langle T_gf,h \rangle - \langle I_0f,h\rangle|  \le \|T_g- I_0\|_{2\to 2}\cdot \|f\|_2 \cdot \|h\|_2 < \epsilon\frac{c'}{\alpha^{1/2}\gamma^{1/2}}\le \epsilon,\]
yielding the conclusion of the theorem.

 The proof of the case $\min\{\alpha,\beta,\gamma\} = \beta < c'$ proceeds similarly to the proof of Theorem \ref{thm:product free sets in good groups}. Let $t= \frac{\log(1/\beta)}{2}.$ We have
 \begin{equation}\label{eq:Master}\langle T_g f ,h\rangle - 1 = \langle T_g f ,h\rangle- \langle T_g f^{=0},h^{=0}\rangle = \sum _{d=1}^{\lfloor t \rfloor} \langle T_g f^{=d},h^{=d}\rangle + \langle T_g f^{ > t },h^{> t} \rangle.\end{equation}
 Using Lemma \ref{lemma:conv bound from goodness}, we obtain the upper bound  
 \begin{equation}\label{eq:product mixing large d}
 |\langle T_g f^{> t},h^{> t}\rangle|  \le \|T_g\|_{V^{> t}}\|f\|_2 \|h\|_2 \le \frac{\beta^{10}}{\alpha^{1/2}\cdot \gamma^{1/2}\cdot n^{10}}<\epsilon/2,
 \end{equation}
 provided that $n_0$ is sufficiently large.
For $1\leq d\le t$, we use the upper bound
\begin{equation}\label{eq:cs3} |\langle T_g f^{=d},h^{=d}\rangle| \le \|T_g\|_{V_{=d}}\|f^{=d}\|_2 \|h^{=d}\|_2.\end{equation}
By Lemma \ref{lemma:conv bound from goodness}, we have $\|T_g\|_{V_{=d}} \le \left(\frac{C'\log(1/\alpha)}{n}\right)^{d/2}$, where $C': = 10C^2/c$. By Theorem \ref{thm: version 2 of level-d in SO_n}, we have the upper bound  
\[ \|f^{=d}\|_2^2\le C'^d  \log^d(e/\alpha),\] and similarly for $h$,   
\[ \|h^{=d}\|_2^2\le  C'^d  \log^d(e/\gamma).\]
Substituting the last three bounds into (\ref{eq:cs3}), we obtain
\begin{equation}\label{eq:product mixing small d}
|\langle T_g f^{=d},h^{=d}\rangle| \le  \left(\frac{C'^3\log(e/\alpha)\log(e/\beta)\log(e/\gamma)}{n}\right)^{d/2}\le \epsilon 4^{-d},
\end{equation}
provided that $c'$ is sufficiently small depending on $C$, $c$ and $\epsilon$.

The sum of the contribution from \eqref{eq:product mixing large d} and of those from \eqref{eq:product mixing small d} for $1 \leq d \leq t$, to the right-hand side of (\ref{eq:Master}), is clearly less than $\epsilon$, yielding $|\langle T_g f,h\rangle - 1| < \epsilon$, as required. 
\end{proof}

\subsection {Equidistribution of convolutions}

Let $A\subseteq G$ be a positive-measure subset of a good Lie group, and suppose that $X$ is a $G$-homogeneous topological space (equipped with its $G$-invariant Haar probability measure $\mu_X$), and that $B\subseteq X$ is a positive-measure subset. The next theorem states that as long as the measures of $A$ and of $B$ are not too small, applying a uniformly random element of $A$ to a uniformly random element of $B$ yields an almost uniformly random element of $X$ (meaning, a random element with respect to the Haar probability measure). 
\remove{
and $B$ gives a nearly uniform distribution over $G$
Let $(\Omega,\mu)$ be a probability space and let $A\subseteq \Omega$ be measurable of positive measure. We write  $\mu_A$ for the conditional distribution of $x\sim \mu$, given that $x\in A$. This is the measure that corresponds  to the function $\frac{1_A}{\mu(A)}$. Suppose that $G$ acts on $X$. We identify the uniform measure on $A$ and $B$ with the functions $\frac{1_A}{\mu(A)}$ and $\frac{1_B}{\mu(B)}$. 
}
Note that Theorem \ref{Thm:equidistribution in two steps} below, together with Lemma \ref{lem:reductionof goodness to the simple case}, imply Theorems~\ref{thm:mixing time 2} and \ref{thm:mixing}.

\begin{thm}\label{Thm:equidistribution in two steps}
For each $C,c,\epsilon >0$ there exists $c',n_0>0$, such that the following holds. Let $n>n_0$, let $G$ be an $n$-graded $(C,c)$-good compact connected Lie group, and let $X$ be a $G$-homogeneous topological space (equipped with the $G$-action $(g,x)\mapsto gx$), and let $\mu_X$ denote the $G$-invariant Haar probability measure on $X$. Suppose that $A \subseteq G$ and $B \subseteq X$ are measurable sets of Haar probability measures $\ge e^{-c'\sqrt{n}}$. Let $\mu_A$ denote the Haar probability measure on $G$, conditioned on the event $A$, and let $\mu_B$ denote the Haar probability measure on $X$ conditioned on the event $B$, i.e.\ $\mu_B(Y) = \mu_X(B \cap Y)/\mu_X(B)$ for a measurable set $Y \subseteq X$, and $\mu_A(Z) = \mu_G(A \cap Z)/\mu_G(A)$ for a measurable set $Z \subseteq G$. Then
the total variation distance between $\mu_X$ and the distribution of $ab$ where $a \sim \mu_A$ and $b \sim \mu_B$ independently, is less than $\epsilon$. 
\end{thm}
\begin{proof}
 Consider first the case where $X=G$ and $G$ acts on itself by left multiplication; in this case, since the distribution of $ab$ is $\mu_A * \mu_B$, we need to show that  $\norm{\mu_A*\mu_B-\mu_G}_{\text{TV}}<\epsilon$. We associate $\mu_A$ with the function $f_A = \frac{1_A}{\mu(A)} \in L^2(G)$ and similarly, we associate $\mu_B$ with the function $f_B=1_B/\mu(B) \in L^2(G)$; it follows easily from the Cauchy-Schwarz inequality that $\norm{\mu_A*\mu_B-\mu_G}_{\text{TV}} \le \|f_A*f_B - 1\|_2$. So our aim is now to prove that
 \begin{equation}\label{eq:stp} \|f_A*f_B - 1\|_2 < \epsilon.\end{equation}
 
 In proving (\ref{eq:stp}), we argue that we may assume, without loss of generality, that $\mu(B)\leq c'$. Indeed, if this does not hold, then write $B=\cup_{i\in I}B_i$ as a finite, disjoint union of sets $B_i$ such that $c'/2\leq \mu(B_i)\leq c'$. Once we have proved (\ref{eq:stp}) for sets of measure at most $c'$, we obtain the desired bound for $B$ by convexity, noting that $f_B$ is a convex combination of the functions $f_{B_i}$. 
 
Set $t=\frac{\log(1/\mu(B))}{2}$. We now have 
 \[
 \|f_A*f_B - 1\|_2^2=
 \sum_{d=1}^{t} {\|T_{f_B}f_A^{=d}\|}_2^2 + \|T_{f_B}f_A^{> t}\|_2^2.
 \]
 Now 
 \[ \|T_{f_B}f_A^{=d} \|_2 \le {\|T_{f_B}\|}_{V_{=d}}\|f_A^{=d}\|_2. \]
 and 
 \[\|T_{f_B}f_A^{\ge t}\|_2 \le {\|T_{f_B}\|}_{V_{>t}} \| f_A^{>t}\|_2. \]
 The bound (\ref{eq:stp}) now easily follows from Theorem  \ref{thm: version 2 of level-d in SO_n} and Lemma \ref{lemma:conv bound from goodness}, similarly to in the proof of Theorem \ref{thm: Product mixing}. Indeed, writing $C': = 10C^2/c$, these yield 
 \[
 \|T_{f_B}\|_{V_{=d}}\|f_A^{=d}\|_2 \le \left(\frac{C'^2\log(1/\mu(A))\log(1/\mu(B))}{n}\right)^{d/2}\le \epsilon 4^{-d}
 \]
 for all $1 \leq d \leq t$,
 and
 \[
 \|T_{f_B}\|_{V_{> t}}\|f_A^{> t}\|_2\le \left(\frac{\beta}{n}\right)^{10}\alpha^{-1/2}\le \epsilon/2,
\]
 provided that $n_0$ is sufficiently large and $c'$ sufficiently small depending on $c,C$ and $\epsilon$.
 
 To prove the general case, note that we may choose an arbitrary $x_0\in X$ and set $\tilde{B}=\set{b\in G:\,bx_0\in B}$. If $\tilde{b} \sim \mu_{\tilde{B}}$ then $\tilde{b}x_0 \sim \mu_B$, and if $g \sim \mu_G$ then $gx_0 \sim \mu_X$; it follows that $\|\mu_{ab}-\mu_X\|_{\text{TV}} \leq \|\mu_A * \mu_{\tilde{B}}-\mu_G\|_{\text{TV}}$, where $\mu_{ab}$ denotes the distribution of $ab$ for $a \sim \mu_A$ and $b \sim \mu_B$ (independently). The result for the pair $(A,\tilde{B})$  therefore implies the result for  $(A,B)$.
\end{proof}

\section{Growth in fine groups}
In this section we show that Theorems \ref{first-level-d}, \ref{thm:growth} and \ref{thm:Brunn Minkowskii}  follow from Theorem \ref{thm:Every simple groups is fine}. 

\medskip
First, the theorem below is a variation on Theorem \ref{thm: level-d in SO_n}, with hypercontractivity replaced by weak hypercontractivity.  Together with Lemma~\ref{lem:reduction of fineness to the simple case} it shows that Theorem~\ref{thm:Every simple groups is fine} implies Theorem~\ref{first-level-d}.
\begin{thm}\label{thm: level-d via finess}
     Let $c>0$, $C>1$ and let $G$ be an $(r,C)$-weakly hypercontractive group. Let $f\colon G\to \{0, 1\}$ be measurable,
      and write $\alpha:=\mathbb{E}[f].$ Let $d \in \mathbb{N}$ be such that $0<d\le \min \left(\frac{\log (1/\alpha)}{2}, r\right)$. Then 
     \[
     \|f^{= d}\|_{2}^2 \le \left(\frac{e\log(1/\alpha)}{d}\right)^{2Cd}\alpha^2.
     \]
\end{thm}
\begin{proof}
Let $q=\log(1/\alpha)/d$, and let $q'$ be the H\"{o}lder conjugate of $q$. Then by H\"{o}lder, weak hypercontractivity and  Lemma~\ref{lem:follows from hypercontractiveness2}, we have 
\[
\|f^{= d}\|_2^2 = \langle f,f^{= d} \rangle \le  \|f^{= d} \|_q \|f\|_{q'} \le q^{Cd}\alpha^{1-1/q}\|f^{= d}\|_2. 
\]
The theorem follows by rearranging and substituting $\alpha^{-1/q} = e^d$. (Note that $C>1$, by the definition of weak hypercontractivity.)
\end{proof}

The following lemma is a variant of Lemma \ref{lemma:conv bound from goodness} for fine groups -- the proof is the same as for Lemma \ref{lemma:conv bound from goodness}, only using Theorem~\ref{thm: level-d via finess} instead of Theorem~\ref{thm: level-d in SO_n}. We will make use of it when proving our diameter bounds for fine groups. 

\begin{lemma}\label{lemma:conv bound from finess}
      For each $c>0$ and $C>1$ there exist $c',C', n_0 >0$ such that the following holds. Let $n>n_0$ and let $G$ be an $n$-graded $(C,c)$-fine group. Let  $A\subseteq G$ be measurable, and suppose that \[\alpha := \mu_G(A)\in (e^{-c'n}, c').\] Write  $f=\frac{1_A}{\alpha}$ and $t=\frac{\log(1/\alpha)}{2}$. Then for $1\leq d \le t$ we have \[{\|T_f\|}_{V_{=d}} \le \left(\frac{e\log(1/\alpha)}{d}\right)^{Cd} \cdot \left(\frac{d}{cn}\right)^{d/2}.\]  We also have  \[{\|T_f\|}_{V_{>t}} \le \frac{\alpha^{10}}{n^{10}}.\] 
\end{lemma}

We now show that if $f$ has small expectation, then most of the Fourier mass of $f$ lies on the high degrees. 

\begin{lemma} \label{lem:KKL use of hypercontractity}
For each $c>0$ and $C>1$ there exist $c',n_0>0$ such that the following holds. Let $n>n_0$, let $G$ be a $(C,c)$-fine $n$-graded group, and let $f\colon G\to \{0,1\}$ be measurable. Suppose that $\alpha:=\mathbb{E}[f]\ge e^{-c'n}$, and let $0\leq t\leq \frac{\log(1/\alpha)}{10^6 C^2}$. Then $\|f^{\le t}\|_2^2 \le 2\cdot 2\alpha^{1.99}$. 
\end{lemma}
\begin{proof}
Inserting a factor of $2$ into the bound from Theorem \ref{thm: level-d via finess}, we have for any $0<d\leq t$ that
\[
\|f^{=d}\|_2^2\leq 2^{-d} \left(\frac{2e\log(1/\alpha)}{d}\right)^{2Cd}\alpha^2,
\]
and for $d=0$ we have $\|f^{=0}\|_2^2 = \alpha^2$.
Therefore, $\|f^{\le t}\|$ is upper-bounded by $2 \alpha^2$ multiplied by the maximum of $\left(\frac{2e\log(1/\alpha)}{d}\right)^{2Cd}$ in the range where $d\le \frac{\log(1/\alpha)}{10^6 C^2}$. Taking logs and computing the  derivative with respect to $d$, it is easy to show that the maximum is obtained at the end point when $d=\frac{\log(1/\alpha)}{10^6 C^2}.$ This shows that 
        \[ \| f^{\le t}\|_{2}^{2} \le 2 \alpha^2 (2 \cdot 10^6 \cdot e C^2 )^{\frac{2\log(1/\alpha)}{10^6 C}} \le 2 \alpha^{1.99},\]
        as required.
\end{proof}

For smaller values of $d$, we have better bounds. 

\begin{lemma}\label{lem: Using smaller values of d}
For each $\epsilon,c>0$ and $C>1$, there exist $c',n_0>0$ such that the following holds.  Let $n>n_0$, let $G$ be a $(C,c)$-fine $n$-graded group, and let $f\colon G\to \{0,1\}$ be measurable. Suppose that $\alpha:=\mathbb{E}[f]\ge e^{-c'n}$ and let $d\in \mathbb{N}$ with $0 < d \leq \frac{\log(1/\alpha)}{n^\epsilon}$. Then \[
\|f^{\le d}\|_2^2 \le \alpha^{2 - n^{-\epsilon/2}}.
\]
\end{lemma}
\begin{proof}
Let  $c'$ be sufficiently small and $n_0$ sufficiently large with respect to $\epsilon,c,C$.
        Let $t=\frac{\log(1/\alpha)}{n^{\epsilon}}.$ Similarly to in the proof of Lemma \ref{lem:KKL use of hypercontractity}, it is easy to see that  
    \begin{equation*}
        \|f^{=d}\|_2^2 \le \left( \frac{e\log(1/\alpha)}{t} \right)^{2Ct} \alpha^{2} = (e\cdot n^{\epsilon})^{2C\log(1/\alpha)/n^{\epsilon}}.
        \end{equation*}
        Now 
        \[(e\cdot n^{\epsilon})^{2C\log(1/\alpha)/n^{\epsilon}} =\alpha^{-n^{-\epsilon}\cdot 2C\log(en^{\epsilon})}\le \alpha^{-n^{-\epsilon/2}},
        \]
        provided that $n \geq n_0$.
\end{proof}

\subsection{Product mixing in fine groups}

The next lemma is a version of Theorem \ref {thm: Product mixing}, except for fine instead of for good groups. The only difference in the proof is that we apply Lemma \ref{lemma:conv bound from finess} in place of Lemma \ref{lemma:conv bound from goodness} and Theorem \ref{thm: level-d via finess} in place of Theorem \ref{thm: level-d in SO_n}. 
\begin{lemma}\label{lem:product free mixing for fine groups}
For each $\epsilon,c>0$ and $C>1$, there exists $\delta>0$ such that the following holds. Let $G$ be a $(C,c)$-fine, $n$-graded group, let $A,B,C\subseteq G$ be measurable sets of measures at least $e^{-n^\delta}$, and let $f=\frac{1_A}{\mu(A)}, g=\frac{1_B}{\mu(B)}, h=\frac{1_C}{\mu(C)}$. Then $| \langle f*g,h\rangle - 1|<\epsilon$.  
\end{lemma}

We remark that Lemma \ref{lem:product free mixing for fine groups} is only weaker than the corresponding Theorem \ref{thm: Product mixing} for good groups. We include it even though the groups of interest to us are both good and fine. We decided to include the lemma mainly for aesthetic reasons. Our diameter bounds rely on Lemma \ref{lem:product free mixing for fine groups} and we preferred to show that our diameter bounds hold for fine groups rather than groups that are both good and fine. 

Below we use a trick of Nikolov and Pyber \cite{np} (who observed that product mixing implies an upper bound on the diameter), to bound the diameter in fine groups. 

\begin{corollary}\label{cor: Diameter from product mixing}
For each $c>0$ and $C>1$, there exists $\delta>0$ such that if $G$ is a $(C,c)$-fine group and $\mathcal{A}\subseteq G$ is a measurable set of measure at least $e^{-n^\delta}$, then $\mu(\mathcal{A}^2)>1-e^{-n^{\delta}}$, and $\mathcal{A}^3 = G$. 
\end{corollary}

\begin{proof}
The claim about $\mu(\mathcal{A}^2)$ follows by applying Lemma \ref{lem:product free mixing for fine groups} while taking $A=B=\mathcal{A}$, $C=G\setminus \mathcal{A}^2$ and $\epsilon=1/2$ (in fact, any value of $\epsilon$ less than one, will do). As for the claim about $\mathcal{A}^3$, suppose for a contradiction that $\mathcal{A}^3 \ne G$. Let $x\in G\setminus \mathcal{A}^3.$ Then $\mathcal{A}^2 \cap x\mathcal{A}^{-1} =\varnothing$. This contradicts Lemma \ref{lem:product free mixing for fine groups} when setting $A=B=\mathcal{A}$ and $C=x\mathcal{A}^{-1}$.
\end{proof}

\subsection{Non-Abelian Brunn-Minkowski type inequalities for fine groups}

The following theorem is a restatement of Theorem~\ref{thm:Brunn Minkowskii}. 

\begin{lemma}\label{lem:Brunn minkowskii weak version}
There exist absolute constants $c',n_0>0$ such that the following holds. Let $G$ be a compact connected Lie group with $n:=n(G)> n_0$, and let $A\subseteq G$ be a measurable set of measure at least $e^{-c'n}$. Then $\mu(A^2)\ge \mu(A)^{1/10}$.
\end{lemma}
\begin{proof}
First note that $\mu(A)>\frac{1}{2}$ implies that $A^2=G$, as if $x \in G\setminus A^2$, then $A$ and $xA^{-1}$ are disjoint sets each of measure greater than $1/2$, a contradiction. By Corollary \ref{cor: Diameter from product mixing}, we may also assume that $\mu(A)\le e^{-n^{c'}}$, provided $c'$ is sufficiently small.

    Let $f=\frac{1_A}{\mu(A)}$ and $g=1_{A^2}$. Then we have $\langle f*f,g\rangle = 1.$ On the other hand, by Cauchy--Schwarz, we have 
    \[ |\langle f*f, g\rangle |\le \|f*f\|_2 \|g\|_2 =  \|f*f\|_2 \sqrt{\mu(A^2)}.\] 
    This yields $\mu(A^2)\ge \frac{1}{\|f*f\|_2^2}.$ Let $t=\frac{\log(1/\mu(A))}{10^6C^2}.$
    We have $\|f*f\|_2^2 = \|f^{\le t}* f^{\le t}\|_2^2 + \|f^{\ge t}* f^{\ge t}\|_2^2.$
    By applying Lemma \ref{lem:KKL use of hypercontractity} to $1_A$, we obtain \[ \|f^{\le t} * f^{\le t}\|_2^2 \le \|f^{\le t}\|_2^4\le \alpha^{-0.02}.\]
    By applying Lemma \ref{lemma:conv bound from finess} to $1_A$, we obtain 
    \[ \|f^{>t}*f^{>t}\|_2^2 \le \|T_f\|^2_{V_{>t}}\|f\|^2_2 \le \left(\frac{\alpha}{n} \right)^{20}\alpha^{-1} \le 1.\]
    Combining these two bounds completes the proof. 
\end{proof}

\begin{lemma}\label{lem:Brunn Minkowskii strong version}
For each $\epsilon,c >0$ and $C>1$ there exist $\delta,n_0>0$ such that the following holds. Let $n > n_0$ and let $G$ be a $(C,c)$-fine $n$-graded group. If $A \subset G$ is a measurable set with
\[ \mu(A) := e^{-n^\zeta}\in \left(e^{-n^{1-\epsilon}}, e^{-n^{\epsilon}} \right),\]
then $\mu(A^2) \ge e^{-n^{\zeta - \delta}}.$
\end{lemma}
\begin{proof}
    We may and shall assume, throughout the proof, that $\delta$ is sufficiently small depending on $\epsilon,C$ and $c$, and that $n_0$ is sufficiently large depending on $\delta$. Let $f=\frac{1_A}{\mu(A)}$. As in the proof of the previous lemma, we have \begin{equation}\label{eq:recip}\mu(A^2)\ge \frac{1}{\|f*f\|_2^2}.\end{equation}
    Let $t_1 = \frac{\log(1/\mu(A))}{n^{4\delta}}$ and $t_2 = \frac{\log(1/\mu(A))}{10^6C^2}.$
    We bound $\|f*f\|_2^2$ from above by decomposing it as follows:
    \begin{equation}\label{eq:exp}\|f*f\|_2^2 = \|f^{<t_1}*f^{<t_1}\|_2^2 + \sum_{d=t_1}^{t_2}\|f^{=d}*f^{=d}\|_2^2 + \|f^{>t_2}*f^{>t_2}\|_2^2.\end{equation}
    
    By applying Lemma \ref{lemma:conv bound from finess}, we obtain 
    \[\|f^{>t_2}*f^{>t_2}\|_2 \le \|T_f \|_{V_{>t_2}}\|f\|_2 \le \frac{\alpha^{10}}{n^{10}}\alpha^{-1/2} \le 1.\]
    Applying Lemma \ref{lem: Using smaller values of d} (with $\alpha f$ in place of $f$, and $\epsilon$ taken to be $4\delta$), we have 
    \[\|f^{<t_1}*f^{<t_1}\|_2 \le \|f^{<t_1}\|_2^2 \le \alpha^{-2}\cdot \alpha^{2-n^{-2\delta}} \le \frac{1}{4}\alpha^{-n^{-\delta}}, \]
    provided that $\delta$ is sufficiently small and $n$ is sufficiently large depending on $\delta$. 
    
    Finally, for $t_1<d<t_2$ we combine Lemma \ref{lemma:conv bound from finess} with Theorem \ref{thm: level-d via finess} to obtain the upper bound 
    \begin{align*}
    \|f^{=d}*f^{=d}\|_2 & \le {\|Tf\|}_{V_{=d}} \|f^{=d}\|_2 \le \alpha^2\cdot \left(\frac{e\log(1/\alpha)}{d}\right)^{2Cd} \cdot \left(\frac{cn}{d}\right)^{-d/2} \\ &\le 1 \cdot \left(\frac{e\log(1/\alpha)}{d}\right)^{2Cd} \cdot \left(\frac{c\log(1/\alpha)n^\epsilon}{d}\right)^{-d/2}.     
    \end{align*} 
    We may now use the fact that 
    \[\left(\frac{e\log(1/\alpha)}{d}\right)\le e \cdot n^{4\delta} \] to obtain 
     
    \[ \|f^{=d}*f^{=d}\|_2 \le \left(\frac{e\log(1/\alpha)}{d}\right)^{2Cd}\cdot (cn^{\epsilon})^{-d/2} \le n^{9\delta Cd} n^{-\epsilon d/2} \le \frac{1}{2n},\] 
    provided $\delta$ is sufficiently small.
    Substituting all of these upper bounds into (\ref{eq:exp}) yields 
    $$\|f*f\|_2^2 \leq 2+\tfrac{1}{4}\alpha^{-n^{-\delta}} \leq \alpha^{-n^{-\delta}}$$
    provided $n_0$ is sufficiently large, and substituting this into (\ref{eq:recip}) completes the proof. 

\end{proof}

\subsection{Diameter bounds for fine groups}

\begin{theorem}
For each $\epsilon,c>0$ and $C>1$, there exist $m,n_0>0$, such that the following holds. Let $n>n_0$ and suppose that $G$ is a $(C,c)$-fine $n$-graded group. Let $A\subseteq G$ be a measurable set with $\mu(A)> e^{-n^{1-\epsilon}}$. Then $A^m=G.$
\end{theorem}
\begin{proof}
Note that we may assume $\epsilon$ is as small as we please (depending on $c$ and $C$). First apply Lemma \ref{lem:Brunn Minkowskii strong version} repeatedly ($N$ times, say), until $A^{2^N}$ has measure $\ge e^{-n^\epsilon}$. In other words, $\mu(A^{m_1}) \geq e^{-n^{\epsilon}}$, where $m_1$ depends upon $\epsilon$ alone. We can now apply Corollary \ref{cor: Diameter from product mixing} to obtain $A^{3m_1}=G$, provided that $\epsilon$ is sufficiently small depending on $c$ and $C$.
\end{proof}

\section{The strong quasirandomness of the simply connected compact Lie groups}\label{section: compacts are quasi}
In this section we describe the degree decomposition of the $n$-graded simply connected simple compact Lie groups in terms of their irreducible subrepresentations. We also show that all of them are $c$-strongly-quasirandom, for some absolute constant $c>0$. In addition, we introduce the notion of {\em comfortable $d$-juntas}. These will be important in our proofs. One of the goals of this section is to show that each Peter-Weyl ideal $W_\rho\subseteq V_{=d}$ contains a comfortable $d$-junta. This is useful because any linear operator on $L^2(G)$ that commutes with the action of $G$ from both sides,
has each $W_\rho$ as an eigenspace. The operators we use in Section~\ref{sec:coupling_introduce} will have this commuting property, and so when computing their eigenvalues we can simply consider the action of the relevant operator on a comfortable $d$-junta. \remove{These can be also used to lower bound the dimension of the representations by showing that there are many different $G$-left translates of them that are linearly independent. Such an argument works well when $d$ is small. For larger values of $d$, our proofs involve a technical estimate of a closed formula for the dimension and we defer these to the appendix.  David: this isn't really necessary for lower bounds on dimensions; there are even easier ways to proceed.}

\subsection{The Peter-Weyl theorem}
We now recall some classical facts from the representation theory of compact groups. The Peter-Weyl theorem states that if $G$ is a compact group, equipped with its Haar probability measure, then $L^2(G)$ has the following decomposition as an orthogonal direct sum:
$$L^2(G) = \bigoplus_{\rho \in \hat{G}}W_{\rho},$$
where $\hat{G}$ denotes a complete set of complex irreducible unitary representations of $G$ (here, {\em complete} means having one irreducible representation from each equivalence class of irreducible representations), and $W_{\rho}$ is the subspace of $L^2(G)$ spanned by functions of the form $g \mapsto u^t\rho(g)v$, for $u,v \in V$, where $V$ is the vector space on which $\rho$ acts. The latter functions are known as the \emph{matrix coefficients} or the {\em matrix entries} of $\rho$. The subspaces $W_{\rho}$ are two-sided ideals (meaning, they are closed under both left and right actions of $G$), and they are also topologically closed; in fact, they are precisely the minimal non-zero topologically closed two-sided ideals of $L^2(G)$, and they are therefore irreducible as $G\times G$-modules (the $G \times G$ action being defined in the obvious way, with the first factor acting on $L^2(G)$ from the left and the second from the right). We call them the {\em Peter-Weyl ideals} of $L^2(G)$, though this terminology is non-standard. The space $W_{\rho}$ can be decomposed as a direct sum of $\dim(\rho)$ irreducible left-representations. \remove{Each of them is isomorphic to the dual representations $g\mapsto \rho(g^{-1})^{\text{tr}}.$}

Since the Peter-Weyl ideals $W_{\rho}$ are precisely the minimal closed two-sided ideals of $L^2(G)$, every closed two-sided ideal of $L^2(G)$ can be decomposed as a direct sum of some of the $W_{\rho}$. Let $d \in \mathbb{N} \cup \{0\}$; since $V_{=d}$ is a closed, two-sided ideal of $L^2(G)$, there exists a set $\mathcal{L}_d$ of irreducible representations of $G$ such that 
\[V_{=d} =  \bigoplus_{\rho \in \mathcal{L}_d} W_{\rho}.\]
If $\rho\in \mathcal{L}_d$ for some integer $0 \leq d < n/2$, we say that the \emph{level} of $\rho$ is equal to $d$. (Note that, since the $V_{=d}$ are pairwise orthogonal, the sets $\mathcal{L}_d$ are pairwise disjoint.)

\subsection{Weyl's construction for $\SO(n)$, and its applications}
\label{sec:weyl}

Our goal is now to show that for the group $\SO(n)$ each $\rho\in \mathcal{L}_d$ has dimension at least $(\tfrac{cn}{d})^d$ for some absolute constant $c>0$, for each $d < n/2$. We will also show that the other irreducible representations of $\SO(n)$ all have dimension at least exponential in $n$, i.e.\ at least $\exp(c' n)$ for some absolute constant $c'>0$.

We briefly recall Weyl's construction of the irreducible representations of $\SO(n)$. For more detail on Weyl's construction, the reader is referred for example to the book \cite{fulton-harris} of Fulton and Harris. (We note that, though the description in \cite{fulton-harris} is of  $\SO(n, \mathbb{C})$, the irreducible representations of $\SO(n):=\SO(n,\mathbb{R})$ are in a dimension-preserving one-to-one correspondence with those of its complexification $\SO(n,\mathbb{C})$.) We start by describing the irreducible representations of $\O(n): = O(n,\mathbb{R})$. Let $V = \mathbb{R}^n$ denote the standard representation of $\O(n)$, defined by $\rho_V(g)(v) = g\cdot v$ --- meaning, multiplication of the matrix $g$ with the column-vector $v$. (We note, for later, that the restriction of this representation to $\SO(n)$ is known as `the standard representation of $\SO(n)$'.) For a partition $\lambda = (\lambda_1,\ldots,\lambda_{\ell})$ of some non-negative integer, let $d = \sum_{i=1}^{\ell} \lambda_i$. Consider the group algebra of the symmetric group on $d$ elements, $\mathbb{R}[S_d]$, with the standard basis $\{e_g:\ g \in S_d\}$, and with multiplication defined by $e_{g} e_h = e_{gh}$ for $g,h \in S_d$. (Where there is no risk of confusion, we will sometimes write $g$ in place of $e_g$, as an element of $\mathbb{R}[S_d]$, as is usual practice.)
Let $T$ be the standard Young tableau of shape $\lambda$ with the numbers $1,2,\ldots,\lambda_1$ (in order) in the first row, the numbers $\lambda_1+1,\lambda_1+2,\ldots,\lambda_1+\lambda_2$ (in order) in the second row, and so on. Also, let $P$ be the subgroup of $S_d$ stabilising each of the rows of $T$ (as sets), let $Q$ be the subgroup of $S_d$ stabilising each of the columns of $T$ (as sets), and let
$$c_{\lambda} = \left(\sum_{g \in P}e_g\right)\left(\sum_{g \in Q}\sign(g)e_g\right)$$
be the {\em Young symmetrizer} of $\lambda$ corresponding to $T$. 
The group $S_d$ acts on $V^{\otimes d}$ from the right, permuting the factors:
$$(v_1 \otimes v_2 \otimes \ldots \otimes v_d)g = v_{g(1)} \otimes v_{g(2)} \otimes \ldots \otimes v_{g(d)},$$
and, extending linearly, so does $\mathbb{R}[S_d]$. 

We define the {\em Weyl module} $\;\mathbb{S}_{\lambda}(V): = V^{\otimes d}c_{\lambda}$. Clearly, $\mathbb{S}_{\lambda}(V)$ is a left $\O(n)$-submodule of $V^{\otimes d}$. It is reducible in general. However, we can obtain an irreducible left $\O(n)$-module by considering
$\mathbb{S}_{[\lambda]}(V) := V^{[d]}c_{\lambda}$, where
$V^{[d]}$ is defined to be the intersection of the kernels of all $\binom{d}{2}$ linear maps on $V^{\otimes d}$ of the form
\[v_1 \otimes v_2 \otimes \ldots \otimes v_d \mapsto \langle v_i,v_j\rangle v_1 \otimes v_2 \otimes \ldots \otimes v_{i-1} \otimes v_{i+1} \otimes \ldots \otimes v_{j-1} \otimes v_{j+1} \otimes \ldots \otimes v_d.\]
Such linear maps are called \emph{contractions}. It turns out that when the sum of the lengths of the first two columns of the Young diagram of $\lambda$ is greater than $n$, we have $S_{[\lambda]}(V)=\{0\}$. The other modules $\mathbb{S}_{[\lambda]}(V)$ (corresponding to those partitions $\lambda$ such that the sum of the first two columns of the Young diagram of $\lambda$ is at most $n$) form a complete set of pairwise inequivalent irreducible complex representations of $\O(n)$.

Weyl's construction for $\SO(n)$ requires only one additional ingredient. We say two partitions $\lambda$ and $\mu$ are {\em associated} if the sum of the lengths of the first column of $\lambda$ and the first column of $\mu$ is equal to $n$, and the $i$th column of $\lambda$ has the same length as the $i$th column of $\mu$ for each $i > 1$. If $\lambda$ and $\mu$ are a pair of distinct associated partitions, then $\mathbb{S}_{[\lambda]}(V)$ and $\mathbb{S}_{[\mu]}(V)$ restrict to isomorphic representations of $\SO(n)$. If $\lambda$ is self-associated (which happens iff $n$ is even and the first column of $\lambda$ has length $n/2$), then $\mathbb{S}_{[\lambda]}(V)$ restricts to a direct sum of two isomorphic irreducible representations of $\SO(n)$; if $\lambda$ is not self-associated, then $\mathbb{S}_{[\lambda]}(V)$ restricts to an irreducible representation of $\SO(n)$, and if $\lambda'$ is the partition associated to $\lambda$, then $\mathbb{S}_{[\lambda']}(V)$ restricts to the same irreducible representation of $\SO(n)$. In the latter case, it is customary to choose (as the representative of its equivalence class), the partition with first column of length less than $n/2$. Note that, importantly for us, for any partition $\lambda$ with $\sum_i \lambda_i < n/2$, $\mathbb{S}_{[\lambda]}(V)$ is irreducible as an $\SO(n)$-representation, as well as being irreducible as an $\O(n)$-representation, and moreoever, as $\lambda$ ranges over partitions of integers less than $n/2$, the $S_{[\lambda]}(V)$ are pairwise inequivalent as $\SO(n)$-representations, as well as being pairwise inequivalent as $\O(n)$-representations.

\subsubsection*{An alternative definition of the level of a representation}

The purpose of this section is to show that for a partition $\lambda$ with $\sum_i \lambda_i < n/2$, the level of the irreducible representation $\mathbb{S}_{[\lambda]}(V)$ of $\SO(n)$ is equal to $\sum_i \lambda_i$. 

For $0\leq d< n/2$, define (as above) $\mathcal{L}_d : = \{\rho \in \widehat{\SO(n)}:\ \rho \text{ has level }d\}$, and define $\tilde{\mathcal{L}}_d$ to be the set of irreducible representations of $\SO(n)$ (up to equivalence) that have the form $S_{[\lambda]}(V)$, where $\sum_{i}\lambda_{i}=d$. We wish to show that $\mathcal{\mathcal{L}}_d=\tilde{\mathcal{L}}_d$ for all $0 \leq d < n/2$. 

\begin{lemma}\label{lemma:v tensor d}
Let $V = \mathbb{R}^n$ be the standard representation of $\SO(n)$ and let $0 \leq d < n/2$. Then all irreducible $\SO(n)$-subrepresentations of $V^{\otimes d}$ are elements of $\cup_{i=0}^{d}\tilde{\mathcal{L}}_i$. 
\end{lemma}
\begin{proof}
We prove this lemma by induction on $d$. The case where $d=0$ is trivial. Let $d \in \mathbb{N}$ and assume the statement of the lemma holds whenever $d$ is replaced by some $d'<d$. Recall that, since $V^{\otimes d}$ can be expressed a sum of $V^{[d]}$ and some other modules all isomorphic to $V^{\otimes (d-2)}$, all the irreducible $\SO(n)$-subrepresentations of $V^{\otimes d}$ appear either as $\SO(n)$-subrepresentations of the module $V^{[d]}$ or as $\SO(n)$-subrepresentations of $V^{\otimes (d-2)}$. By the induction hypothesis, it therefore suffices to show that each irreducible $\SO(n)$-subrepresentation of $V^{[d]}$ is an element of $\tilde{\mathcal{L}}_i$ for some $i$. Let $c_{\lambda,T}$ be the Young symmetrizer corresponding to the Young tableau $T$ (not necessarily the standard one) of shape $\lambda$. It is well-known that the $c_{\lambda,T}$ (as $\lambda$ ranges over all partitions of $d$ and $T$ over all Young tableaus of shape $\lambda$) spans a subspace of $\mathbb{R}[S_d]$ containing the class functions; in particular, we may write $\text{Id} \in \mathbb{R}[S_d]$ as a real linear combination of the $c_{\lambda,T}$'s. It follows that $V^{[d]}$ is a sum of left $\O(n)$-modules of the form $V^{[d]}c_{\lambda,T}$, and clearly $V^{[d]}c_{\lambda,T}$ is isomorphic to $V^{[d]}c_{\lambda} = \mathbb{S}_{[\lambda]}(V)$, as either a left $\O(n)$-module or a left $\SO(n)$-module. This completes the proof of the lemma.
\end{proof}

The following lemma implies that $\mathcal{L}_d=\tilde{\mathcal{L}}_d$ for all $0 \leq d < n/2$.

\begin{lemma}\label{lem: degrees of So_n in terms of reps}
If $0 \leq d<n/2$, then
\[ V_{=d} = \bigoplus_{\rho \in \tilde{\mathcal{L}}_d} W_{\rho}.\]
\end{lemma}
\begin{proof}
    We prove the lemma by induction on $d$. For $d=0$, the statement is trivial. Suppose now that $d>0$. Since the spaces $W_\rho$ are pairwise orthogonal, the induction hypothesis reduces our task to showing that \[ V_{\le d} = \bigoplus_{i=0}^{d}\bigoplus_{\rho \in \tilde{\mathcal{L}}_d} W_{\rho}. \] 
    Let us write $\tilde{V}_{\le d} :=\bigoplus_{i=0}^{d}\bigoplus_{\rho \in \tilde{\mathcal{L}}_d} W_{\rho}$.
    We now use Lemma \ref{lemma:v tensor d}, namely that all the irreducible subrepresentations of $V^{\otimes d}$ are elements of $\tilde{\mathcal{L}}_i$ for some $i\le d$.   
    The matrix coefficients of the representation $V^{\otimes d}$ include the entries of the matrix $X^{\otimes d}$, where $X \in \SO(n)$ is the input matrix. These are exactly the degree-$d$ monomials in the entries of $X$. Decomposing $V^{\otimes d}$ into irreducible representations, we see that all the homogeneous degree-$d$ polynomials belong to $\tilde{V}_{\le d}$; using the induction hypothesis again, all polynomials of degree at most $d-1$ belong to $\tilde{V}_{\leq d-1} \subset \tilde{V}_{\leq d}$, and therefore $V_{\le d}\subseteq \tilde{V}_{\leq d}$.
    The reverse inclusion ($\tilde{V}_{\le d}\subseteq V_{\leq d}$) follows immediately\remove{that converse follows from the fact that each $S_{\lambda}(V)$ is a subrepresentation of $V^{\otimes d}$ and therefore its matrix coefficients are spanned by the matrix coefficients of $V^{\otimes d},$ which are all degree $d$ polynomials. 
 To prove the converse we observe} from the fact that if $\sum_{i}\lambda_i = d$, then the matrix coefficients of $\mathbb{S}_{[\lambda]}(V)$ are homogeneous degree-$d$ polynomials in the entries of the input matrix. 
\end{proof}

We can now extend the notion of level to all the representations of $\SO(n).$

\begin{definition}
    Let $\rho$ be an irreducible representation of $\SO(n)$ and let $\lambda$ be the corresponding partition, whose Young diagram has first column of length at most $n/2$. Then we define the {\em level} of $\rho$ to be $\sum_i \lambda_i$.
\end{definition}

\subsubsection*{Comfortable $d$-juntas}
We now digress a little and show that Weyl's construction implies that each Peter-Weyl ideal $W_{\rho}$ contains a certain `nice' function. This will be used later, in Section \ref{sec:coupling_introduce}. 

\remove{Given a linear space of $n\times n$-matrices we denote by $x_{ij}$ the map $X\mapsto e_i^tXe_j$, where $e_i$ is the standard basis. }

Our `nice' functions are as follows.

\begin{definition}
    The {\em comfortable $d$-juntas} on $\SO(n)$ are the functions on $\SO(n)$ of the form 
    \[X \mapsto \sum_{\sigma \in S_d} a_{\sigma}x_{1,\sigma(1)}\cdots x_{d,\sigma(d)}\]
    for $a_{\sigma} \in \mathbb{R}.$
\end{definition}

We remark that we use the term `junta' here, by analogy with juntas in the theory of Boolean functions (on $\{0,1\}^n$), because functions of the above form depend only upon the upper $d\times d$ minor, though we stress that we will be interested in the case where $d$ is polynomial in $n$ (e.g.\ $d \sim \sqrt{n}$), rather than just the case of $d$ fixed and $n$ large.

Letting $e_1,\ldots, e_n$ be the standard orthonormal basis of $\mathbb{R}^n$, since $\langle e_i,e_j\rangle = 0$ for all $i,j \in [d]$ we have $e_1 \otimes \ldots \otimes e_d \in V^{[d]}$. Therefore $(e_1 \otimes  \ldots \otimes e_d)c_{\lambda} \in S_{[\lambda]}(V),$  and thus the function $P_\lambda$ in $L^2(\SO(n))$ defined by
$$P_\lambda(X) := \langle X((e_1 \otimes  \ldots \otimes e_d)c_{\lambda}),e_1 \otimes \ldots \otimes e_d  \rangle$$
is a matrix coefficient of $\mathbb{S}_{[\lambda]}(V)$. \remove{Indeed, the map $\langle \cdot, e_1 \otimes \cdots \otimes e_d\rangle$ is a functional on $\mathbb{S}_{[\lambda]}(V)$.} Moreover, the function $P_{\lambda}$ is clearly a comfortable $d$-junta: writing $c_{\lambda} = \sum_{\sigma \in S_d} \epsilon_\sigma \sigma$, where $\epsilon_\sigma \in \{-1,0,1\}$ for each $\sigma \in S_d$, we have
$$P_\lambda(X) = \sum_{\sigma \in S_d}\epsilon_\sigma \prod_{i=1}^{d}x_{i, \sigma (i)}.$$
Moreover, we clearly have $P_{\lambda}(\text{Id})=1$, so $P_{\lambda}$ is a non-zero element of $L^2(\SO(n))$. We obtain the following conclusion, upon which we rely crucially in the sequel.
\begin{fact}
\label{fact:weyl-useful}
Let $0 \leq d< n/2 $. For each irreducible representation $\rho \in \mathcal{L}_d$ of $\SO(n)$, the Peter-Weyl ideal $W_{\rho}$ contains a nonzero comfortable $d$-junta. \end{fact}

\remove{
\begin{proof}
    As we saw above, the function $P_{\lambda}$ lies in $W_{\rho}$. To complete the proof we must show that it is nonzero. But its value at the identity matrix is clearly one, so we are done. \remove{Let $\pi\colon \SO(n)\to \mathbb{R}^{d\times d}$ be the projection from $\SO(n)$ to the upper $d\times d$-minor. Then the image of $\pi$ is easily seen to have a nonempty interior inside $\mathbb{R}^{d\times d}$. Now the polynomial $p_{\lambda}$ is nonzero as a polynomial on $\mathbb{R}^{d\times d}$. It therefore cannot vanish on an a nonempty open set. This shows that $P_\lambda$ does not vanish on the image of $\pi$ and it therefore does not vanish on $\SO(n)$.}  
\end{proof}
}

\subsection{Obtaining strong quasirandomness for $\SO(n)$.}

The following lower bound on the dimension of an irreducible representation of $\SO(n)$ follows immediately from the analysis in \cite{samra-king} of Weyl's original dimension formulae \cite{weyl}. (We note that our comfortable $d$-junta machinery could be used to easily obtain a slightly weaker lower bound of $\binom{\lfloor n/2 \rfloor }{d}$. We use such an argument later, when showing strong quasirandomness for $\SU(n)$.)

\begin{lemma}
\label{lemma:lb1son}
If $\rho$ is an irreducible representation of $\SO(n)$ of level $d \leq n$, then
$$\dim(\rho) \geq \frac{(n-d)^d}{d!}.$$
\end{lemma}

We also need the following lower bound, whose proof is deferred to the Appendix.
\begin{lemma}
\label{lemma:lb2son}
Let $n \geq 5$. If $\rho$ is an irreducible representation of $\SO(n)$ of level $d\geq n/2$, then
$$\dim(\rho) \geq \exp(n/32).$$
\end{lemma}

Lemmas \ref{lemma:lb1son} and \ref{lemma:lb2son} immediately give strong quasirandomness.

\begin{theorem}\label{thm:son is quasirandom}
    For each $n \geq 2$, the $n$-graded group $\SO(n)$ is $c$-strongly-quasirandom, for some absolute constant $c>0$.
\end{theorem}

\subsection{Obtaining strong quasirandomness for $\Spin(n)$.}

The strong quasirandomness of the group $\Spin(n)$ follows from the fact that it is a double cover of $\SO(n)$. 

\begin{thm}\label{thm:Spin is quasirandom}
For each $n \geq 3$, the $n$-graded group $\Spin(n)$ is $c$-strongly-quasirandom, for some absolute constant $c>0$.
\end{thm}
\begin{proof}
\remove{The representations of the spin group $\Spin(n)$ include those of $\SO(n)$ together with additional ones called spin representations. Recall that we set $V_{=d}^{\Spin(n)}$ to consist of the functions that factor through $V_{=d}^{\SO(n)}.$ This implies that for all the spin representations $\rho$ we obtain that the Peter-Weyl ideal $W_{\rho}$ is contained in $V_{>n/2}.$ Therefore, in order to show that the spin group is $c$-strongly-quasirandom we only have to show that all the spin representations have dimension exponential in $n$. This is established in the Appendix. (See Section \ref{subsec:Quasirandomness of spin representations})}

Recall that the spin group $\Spin(n)$ is the double-cover of $\SO(n)$ for all $n \geq 2$. It is a simply-connected real Lie group for all $n \geq 3$, so its complex irreducible representations are in an explicit (and dimension-preserving) one-to-one correspondence with those of its Lie algebra. It has the same Lie algebra as $\SO(n)$; this Lie algebra $\mathfrak{so}(n,\mathbb{R})$ is simple for all $n \geq 5$, and its complexification $\mathfrak{so}(n,\mathbb{C})$ is likewise simple (for all $n \geq 5$), so by e.g.\ \cite{fulton-harris} (26.14), for all $n \geq 5$ the complex irreducible representations of $\mathfrak{so}(n,\mathbb{R})$ are restrictions of unique complex irreducible representations of $\mathfrak{so}(n,\mathbb{C})$. The dimensions of the complex irreducible representations of $\text{Spin}(n)$ (for all $n \geq 5$) are therefore given by equations (24.29) and (24.41) in \cite{fulton-harris} (pages 408 and 410). For all $n=2k+1 \geq 5$ odd, we have
$$\dim(\rho_{\lambda}) = \prod_{1\leq i < j \leq k} \frac{\lambda_i-\lambda_j-i+j}{j-i}\prod_{1 \leq i \leq j \leq k}\frac{\lambda_i+\lambda_j+2k+1-i-j}{2k+1-i-j},$$
where the $k$-tuple $\boldsymbol{\lambda} = (\lambda_1, \lambda_2,\ldots \lambda_k)$ ranges over all $k$-tuples defined by
$$\lambda_i = a_i+a_{i+1}+\ldots+a_{k-1}+\tfrac{1}{2}a_k$$
for some $(a_i)_{i=1}^{k} \in (\mathbb{N} \cup \{0\})^k$. The case of $a_k$ even corresponds to irreducible representations of $\text{Spin}(n)$ that are also irreducible representations of $\SO(n)$; the dimensions of these were bounded previously. The case of $a_k$ odd corresponds to `new' irreducible representations of $\text{Spin}(n)$, but the above equation implies that any such has dimension at least $2^{\Omega(n)}$. This calculational check is deferred to the Appendix (see Section \ref{subsec:Quasirandomness of spin representations}). 
\end{proof}

\subsection{Two descriptions of the compact symplectic group \boldmath\texorpdfstring{$\Sp(n)$}{Sp(n)}}
\label{sec:spn-constr}

We now turn to the case of the compact symplectic group. This group has two common descriptions, and both will be useful for us.

We start with its description as the unitary group over the field of quarternions, $\mathbb{H}$. A shorthand is useful at this point: a {\em quaternionic matrix} is a matrix with quaternion entries. The conjugate $\bar{x}$ of a quaternion $x=a+ib+jc+kd$ is defined, as usual, by $\bar{x} = a-ib-jc-kd$. The conjugate $\bar{M}$ of a quaternionic matrix $M$ is defined by $(\bar{M})_{s,t} = \overline{M_{s,t}}$ for all $s,t$, and the Hermitian conjugate $M^*$ is defined by $M^* = (\bar{M})^t$. We say an $n$ by $n$ quaternionic matrix is {\em unitary} if
$$M^*M = I = MM^*;$$
we note that the second equality above is equivalent to the first. For $n \in \mathbb{N}$, the \emph{quarternionic unitary group} $U_n(\mathbb{H})$ is defined to be the group of $n$ by $n$ quaternionic unitary matrices. This is one way of viewing $\Sp(n)$.

A second way of viewing $\Sp(n)$ is as follows. Note that an $n$ by $n$ quaternionic matrix can be written in the form $A+jB$, with $A$ and $B$ being complex $n$ by $n$ matrices. If $A+jB$ and $C+jD$ are two quaternionic matrices, then their product satisfies
\[(A+jB)(C+jD)=AC-\bar{B}D + j(BC + \bar{A}D),\]
and the Hermitian conjugate of the quaternionic matrix $A+jB$ satisfies
$(A+jB)^*=A^*-jB^t$. It follows that the map
$$\Phi:\mathbb{H}^{n \times n} \to \mathbb{C}^{2n \times 2n}$$
from $n$ by $n$ quaternionic matrices to $2n$ by $2n$ complex matrices defined by
$$\Phi(A+jB) = \begin{pmatrix} A & -\bar{B}\\ 
B & \bar{A} 
\end{pmatrix}$$
is an (injective) ring homomorphism that preserves Hermitian conjugates. It is easily checked that $M \in \mathbb{H}^{n \times n}$ is unitary if and only if $\Phi(M) \in \mathbb{C}^{2n \times 2n}$ is unitary (writing $M = A+jB$, either condition holds iff $A^*A+B^*B = I$ and $A^t B = B^t A$ both hold). Finally, defining
\[\Omega: = \begin{pmatrix} 0 & I_n \\ -I_n & 0\end{pmatrix},\]
we observe that an arbitrary unitary matrix
$$X = \begin{pmatrix} A & C\\ 
B & D \end{pmatrix} \in \mathbb{C}^{2n \times 2n}$$
satisfies $X^t \Omega X = \Omega$ (or, equivalently, $X^t \Omega = \Omega X^*$) if and only if $C = -\bar{B}$ and $D = \bar{A}$, i.e.\ if and only if $X$ lies in $\Phi(U_n(\mathbb{H}))$. It follows that
$$\Phi(U_n(\mathbb{H})) = \{X \in \mathbb{C}^{2n \times 2n}:\ X^t \Omega X = \Omega\} \cap U_{2n}(\mathbb{C}).$$

Hence, we can also view $\Sp(n)$ as
$$\{X \in \mathbb{C}^{2n \times 2n}:\ X^t \Omega X = \Omega\} \cap U_{2n}(\mathbb{C}),$$
i.e., as the intersection of the compact Lie group $U_{2n}(\mathbb{C})$ with
\[\Sp(2n, \mathbb{C}):=\{X \in \mathbb{C}^{2n \times 2n}: X^t \Omega X =  \Omega \}.\]
(The group $\Sp(2n,\mathbb{C})$ is known as the {\em complex symplectic group}, and it is the complexification of $\Sp(n)$.) It can be checked (e.g.\ by taking the Pfaffian of the equation $X^t \Omega X = \Omega$) that any element of $\Sp(2n,\mathbb{C})$ has determinant one, so $U_{2n}(\mathbb{C})$ can be replaced by $\SU_{2n}(\mathbb{C}) = \SU(2n)$ in the above identification:
$$\Sp(n)=\{X \in \mathbb{C}^{2n \times 2n}:\ X^t \Omega X = \Omega\} \cap \SU(2n).$$

\subsection{Weyl's construction in $\Sp(n)$, and its applications}
Weyl's construction for $\Sp(n)$ is similar to his construction for $\O(n)$.
Our exposition follows Fulton and Harris \cite{fulton-harris}, as before.

The standard representation of $\Sp(n)$ is given by the second description above, regarding the elements of $\Sp(n)$ as $2n$ by $2n$ complex unitary matrices, and taking their natural (left) action on $\mathbb{C}^{2n}$. 

Let $V=\mathbb{C}^{2n}$ denote the standard representation of $\Sp(n)$. Similarly to in the $\O(n)$ case, we have contraction maps $\psi_{i,j}\colon V^{\otimes d}\to V^{\otimes (d-2)}$, given by
\[\psi_{i,j}:\ v_1\otimes \cdots \otimes v_d \mapsto Q(v_i,v_j)v_1\otimes \cdots \hat{v}_i\otimes \cdots \otimes \hat{v}_j\otimes \cdots \otimes v_d,\]
where $\hat{v}$ denotes that $v$ is omitted from the tensor product, and $Q(v,w): = v^t \Omega w$, where $\Omega$ is the matrix above.

 Let $V^{\langle d\rangle}$ be the intersection of the kernels of all the contractions $\psi_{i,j}$; since the elements of $\Sp(n)$ preserve the skew-symmetric form $Q$, $V^{\langle d \rangle}$ is a left $\Sp(n)$-module. Moreover, $V^{\langle  d\rangle}$ is acted upon by $S_d$ from the right, by permutation of the factors. (The fact that this action preserves $V^{\langle d \rangle}$ follows from the skew-symmetry of $Q$.) For a partition $\lambda \vdash d$, we define $S_{\langle \lambda \rangle}(V)$ to be the representation (or left $\Sp(n)$-module) $V^{\langle d \rangle}c_{\lambda}$, where $c_\lambda$ is the Young symmetrizer defined in Section \ref{sec:weyl}. It turns out that $S_{\langle \lambda \rangle}(V)$ is nonzero precisely when (the Young diagram of) $\lambda$ has at most $n$ rows, and the nonzero left $\Sp(n)$-modules of this form constitute a complete set of pairwise inequivalent complex irreducible representations of $\Sp(n)$. Note that the situation here is, if anything, even simpler than that for $\SO(n)$, where we have to worry, if only momentarily, about distinct $S_{[\lambda]}(V)$ (for two distinct values of $\lambda$) restricting to the same irreducible representation of $\SO(n)$.

Recall that the {\em level} of an irreducible representation $\rho$ of $\Sp(n)$ was defined to be the non-negative integer $d$ such that $\rho \in \mathcal{L}_d$, where
$$V_{=d}^{\Sp(n)} = \bigoplus_{\rho \in \mathcal{L}_d} W_{\rho}.$$
Let $\lambda$ be a partition such that $\sum_i \lambda_i < n/2$. Our next aim is to show that the level of the irreducible representation $S_{\langle \lambda \rangle}(V)$ is equal to $\sum_i \lambda_i$, as in the case of $\SO(n)$. 

Earlier, in Section \ref{sec:grading-defn}, we wrote our input quarternion matrix as $A+iB+jC+kD$, where $A,B,C,D$ are $n$ by $n$ real matrices, and we defined $V_{\le d}$ to consist of the degree $\leq d$ polynomials in the entries of the matrices $A,B,C$ and $D$. Alternatively, we may write the input matrix as $A+jB,$ where $A$ and $B$ are $n$ by $n$ complex matrices; then $V_{\le d}$ may be viewed as the vector space of degree $\le d$ polynomials in the entries of $A,B\in \mathbb{C}^{n\times n}$ and their complex conjugates. Finally, we may use the identification $\Phi$ above (in Section \ref{sec:spn-constr}) of a quarternion matrix $A+jB$ with the $2n\times 2n$ complex matrix $\begin{pmatrix} A & -\bar{B}\\ 
B & \bar{A}
\end{pmatrix}$ to obtain that $V_{\le d}$ simply consists of polynomials of degree at most $d$ in the entries of $\begin{pmatrix} A & -\bar{B}\\ 
B & \bar{A}
\end{pmatrix}$; this is precisely what we need to generalise our $\SO(n)$ proofs.

For each $0 \leq d <n/2$, let $\tilde{\mathcal{L}}_d$ denote the set of irreducible representations of $\Sp(n)$ (up to equivalence) of the form $\mathbb{S}_{\langle \lambda \rangle}(V)$, where $\lambda$ is a partition of $d$. Using the fact that, as with $\SO(n)$, any irreducible $\Sp(n)$-subrepresentation of $V^{\otimes d}$ appears either as an $\Sp(n)$-subrepresentation of the module $V^{\langle d \rangle}$ or as an $\Sp(n)$-subrepresentation of $V^{\otimes (d-2)}$, the proof of Lemma \ref{lemma:v tensor d} generalizes straightforwardly to give:

\begin{lemma}
Let $V = \mathbb{C}^{2n}$ be the standard representation of $\Sp(n)$ and let $0 \leq d < n/2$. Then all irreducible $\Sp(n)$-subrepresentations of $V^{\otimes d}$ are elements of $\cup_{0\leq i \leq d}\tilde{\mathcal{L}}_i$. 
\end{lemma}

Then, taking the input matrix $X$ in the form $\begin{pmatrix} A & -\bar{B}\\ 
B & \bar{A}
\end{pmatrix}$, the proof of Lemma \ref{lem: degrees of So_n in terms of reps} generalises straightforwardly to give:

\begin{lemma}\label{lem:symplectic degree decomposition works}
Let $0 \leq d<n/2$. Then $V_{= d}^{\Sp(n)}= \bigoplus_{\rho \in \tilde{\mathcal{L}}_d} W_{\rho}$. 
\end{lemma}

\remove{In fact, the conclusion of the previous lemma holds under the weaker hypothesis that $0 \leq d \leq n$, though we do not need this fact.}

As in the $\SO(n)$ case, we may now extend the notion of level to all irreducible representations of $\Sp(n)$.
\begin{definition}
Let $\rho$ be an irreducible representation of $\Sp(n)$. The {\em level} of $\rho$ is the integer $d$ such that $\rho$ is isomorphic to $S_{\langle \lambda \rangle}(V)$, where $\lambda$ is a partition of $d$.
\end{definition}

\subsubsection*{Comfortable $d$-juntas}
Viewing $\Sp(n)$ as $U_n(\mathbb{H})$, we say a function $f:\Sp(n) \to \mathbb{C}$ is a {\em comfortable $d$-junta} if it is a complex linear combination of monomials of the form
$$M \mapsto \prod_{\ell=1}^{d} (M_{\ell,\sigma(\ell)})_{q_\ell\text{-part}}$$
where $\sigma \in S_d$ and $q_1,\ldots, q_{d}\in \{i,j,k,\textbf{real}\}$.

An easy variant of our argument for $\SO(n)$, yields the following. 
\begin{lemma}\label{lem: In sp_n every isotypical component contains a d-junta}
    Let $\rho$ be a representation of level $d \leq n$ of $\Sp(n)$. Then $W_\rho$ contains a non-zero comfortable $d$-junta.  
\end{lemma}
\begin{proof}
    We first adopt the second perspective on $\Sp(n)$, viewing it as
    $$\{X \in \mathbb{C}^{2n \times 2n}:\ X^t \Omega X = \Omega\} \cap \SU(2n).$$
    Letting $e_1,\ldots, e_{2n}$ be the standard orthonormal basis of $\mathbb{C}^{2n}$, since $Q(e_i,e_j) = 0$ for all $i,j \in [n]$ we have $e_1 \otimes \ldots \otimes e_d \in V^{\langle d \rangle}$. Therefore $(e_1 \otimes  \ldots \otimes e_d)c_{\lambda} \in S_{\langle \lambda \rangle}(V)$, and thus the function $P_\lambda$ in $L^2(\Sp(n))$ defined by
$$P_\lambda(X) := \langle X((e_1 \otimes  \ldots \otimes e_d)c_{\lambda}),e_1 \otimes \ldots \otimes e_d  \rangle$$
is a matrix coefficient of $\mathbb{S}_{\langle \lambda\rangle}(V)$. Note that this is exactly the same function as we exhibited for $\SO(n)$, except that its domain is
$$\{X \in \mathbb{C}^{2n \times 2n}:\ X^t \Omega X = \Omega\} \cap \SU(2n) \cong \Sp(n)$$
rather than $\SO(n)$. We claim that the function $P_{\lambda}$ is a comfortable $d$-junta. Indeed, writing $c_{\lambda} = \sum_{\sigma \in S_d} \epsilon_\sigma \sigma \in \mathbb{R}[S_n]$, where $\epsilon_\sigma \in \{-1,0,1\}$ for each $\sigma \in S_d$, we have
$$P_\lambda(X) = \sum_{\sigma \in S_d}\epsilon_\sigma \prod_{\ell=1}^{d}x_{\ell,\sigma (\ell)}.$$
Writing $X = \begin{pmatrix} A & -\bar{B}\\ 
B & \bar{A}
\end{pmatrix}$, we see that each $x_{\ell,\sigma(\ell)}$ appearing above is actually $a_{\ell,\sigma(\ell)}$ (since $d \leq n$ and $\sigma \in S_d$); writing each $a_{\ell,\sigma(\ell)}$ as a sum of its real and imaginary parts and expanding, we see that $P_{\lambda}$ is indeed a complex linear combination of monomials of the required form (in fact, with each $q_\ell$ being either $\textbf{real}$ or $i$, and with the coefficient of each monomial being in $\{\pm 1, \pm i\}$).

The function $P_{\lambda}$ is non-zero element of $L^2(\Sp(n))$ since (as with $\SO(n)$) we have $P_\lambda(\text{Id}) = 1$. 
\end{proof}

\subsection{Obtaining strong quasirandomness for $\Sp(n)$.}

 The dimensions of the complex irreducible representations of $\Sp(n)$ are given by equation (24.19) in \cite{fulton-harris}. These representations are in one-to-one correspondence with partitions of non-negative integers, whose Young diagrams have at most $n$ rows; the dimension of the irreducible representation $\rho_{\lambda}$ corresponding to the partition $\lambda = (\lambda_1,\ldots,\lambda_n)$ (with $\lambda_1 \geq \ldots \geq \lambda_n$) is given by
$$\dim(\rho_{\lambda}) = \prod_{1 \leq i < j \leq n}\frac{\lambda_i-\lambda_j-i+j}{j-i} \prod_{1 \leq i \leq j \leq n} \frac{\lambda_i+\lambda_j+2n+2-i-j}{2n+2-i-j}.$$
As with the (odd) special orthogonal groups, we define the {\em level} of the representation $\rho_{\lambda}$ to be the number of cells in the Young diagram of $\lambda$ (i.e., it is the non-negative integer of which $\lambda$ is a partition). The proof of the following lemma is similar to that of Lemma~\ref{lemma:lb2son}. We defer it to the Appendix. 
\begin{lemma}\label{lemma:lb2spn}
If $\rho_{\lambda}$ is an irreducible representation of $\Sp(n)$ of level $d \geq n/2$, then $\dim(\rho_{\lambda}) \geq \exp(n/16)$.
\end{lemma}
Again, as with the special orthogonal groups, the following lower bound is immediate from the analysis in \cite{samra-king} of Weyl's dimension formulae \cite{weyl}.

\begin{lemma}
\label{lemma:lb1spn}
If $\rho$ is an irreducible representation of $\Sp(n)$ of level $d \leq n$, then
$$\dim(\rho) \geq \frac{(n-d)^d}{d!}.$$
\end{lemma}

These yield the following.
\begin{thm}\label{thm:Sp is quasirandom}
For each $n \geq 2$, the $n$-graded group $\Sp(n)$ is $c$-strongly-quasirandom, for some absolute constant $c>0$.
\end{thm}

\subsection{Weyl's construction for $\SU(n)$, and its applications.}
The final group to consider is the special unitary group. We start by relating our degree decomposition of $L^2(\SU(n))$ to the decomposition of $L^2(\SU(n))$ into Peter-Weyl ideals.     

\subsubsection*{Degree decomposition}

Earlier, we defined $V_{\le d}$ to consist of the polynomials of (total) degree at most $d$ polynomials in the real parts and the imaginary parts of the entries of the input matrix $X \in \SU(n)$. Equivalently, $V_{\le d}$ consists of the polynomials of (total) degree at most $d$ in the entries of the input matrix and their complex conjugates.

\subsubsection*{Weyl's construction for \SU(n)}

We now recall Weyl's construction of the irreducible representations of $\SU(n)$, and deduce from it the consequences we need. (As before, for more detail on Weyl's construction, the reader is referred to Fulton and Harris \cite{fulton-harris}, noting that the complex irreducible representations of $\SU(n)$ are the same as those of $\text{SL}(n,\mathbb{C})$, since $\SU(n)$ is a maximal compact subgroup of $\text{SL}(n,\mathbb{C})$.) Let $V = \mathbb{C}^n$ denote the standard representation of $\SU(n)$, defined by $\rho_V(g)(v) = g\cdot v$. For a partition $\lambda$ with at most $n-1$ parts, let $d = d(\lambda) = \sum_{i} \lambda_i$. Let $T$ be the standard Young tableau of shape $\lambda$ (defined in Section \ref{sec:weyl}) and let
$c_{\lambda}$ be the corresponding Young symmetrizer (also defined in Section \ref{sec:weyl}). We define the corresponding {\em Weyl module} by $\mathbb{S}_{\lambda}(V): = V^{\otimes d}c_{\lambda}$. Clearly, $\mathbb{S}_{\lambda}(V)$ is a left $\SU(n)$-submodule of $V^{\otimes d}$. The modules $\mathbb{S}_\lambda(V)$, as $\lambda$ ranges over all partitions with at most $n-1$ parts, constitute a complete set of pairwise inequivalent complex irreducible representations of $\SU(n)$. (Unlike in the cases of $\SO(n)$ and $\Sp(n)$, we do not need to pass to a subrepresentation of the Weyl module; the latter is already irreducible as a left $\SU(n)$-module.)

Unlike in the case of $\Sp(n)$, however, the complex conjugates of the entries of the input matrix are no longer matrix coefficients of the standard representation. Instead, they are matrix coefficients of the dual of the standard representation, i.e.\ they are the entries of the matrix $(A^{-1})^{t}=\bar{A}$. (Recall that the dual of a representation $\rho$ is the representation $\rho^*$ defined by $\rho^*(g)=(\rho(g^{-1}))^t$.) We note that $(\mathbb{S}_{\lambda}(V))^* \cong \mathbb{S}_{\lambda}(V^*) = (V^{*})^{\otimes d} c_{\lambda}$.

We have three notions of level for a representation $\rho$ of $\SU(n)$, and our goal is to show that they agree when the level is $<n/2$.

Recall that we defined $V_{=0} := V_{\leq 0}$ (the space of constant functions), and $V_{=d} = V_{\leq d} \cap V_{\leq d-1}^{\perp}$ for each $d \in \mathbb{N}$. Since $V_{=d}$ is closed under both left and right actions of $\SU(n)$ for each $d \in \mathbb{N} \cup \{0\}$, there exists a set $\mathcal{L}_d$ of irreducible representations of $\SU(n)$ such that
$$V_{=d}^{\SU(n)} = \bigoplus_{\rho \in \mathcal{L}_d}W_\rho.$$
If $\rho \in \mathcal{L}_d$ for some $0 \leq d < n/2$, we define the {\em level} of $\rho$ to be $d$. (Note that, since the $V_{=d}$ are pairwise orthogonal, the sets $\mathcal{L}_d$ are pairwise disjoint.)

In addition, we define the \emph{tensor level} of an irreducible representation $\rho$ of $\SU(n)$ to be the minimal non-negative integer $d$ such that there exist $d_1,d_2 \in \mathbb{N} \cup \{0\}$ with $d_1+d_2=d$ and with $\rho$ being isomorphic to a subrepresentation of $V^{\otimes d_1}\otimes V^{*\otimes d_2}$. We write $\bar{\mathcal{L}}_d$ for the set of irreducible representations of $\SU(n)$ (up to equivalence) that have tensor level equal to $d$.

The third notion of level will soon be described in terms of the Young diagram/partition corresponding to the representation. 

The following analogue of Lemma \ref{lemma:v tensor d} is immediate, and implies that for $0 \leq d < n/2$, having level $d$ is equivalent to having tensor level $d$.

\remove{
\begin{lemma}
   If $0 \leq d_1+d_2<n/2$, then all irreducible $\SU(n)$-subrepresentations of $V^{\otimes d_1} \otimes (V^*)^{\otimes d_2}$ are elements of $\cup_{i=0}^{d_1+d_2}\bar{\mathcal{L}}_i$.
\end{lemma}
\begin{proof}
    
\end{proof}
}

\begin{lemma}
If $0 \leq d<n/2$, then $V_{=d}^{\SU(n)}=\bigoplus_{\rho\in \bar{\mathcal{L}}_d}W_{\rho}$.
\end{lemma}

\subsubsection*{Step vectors}
For an irreducible representation $\rho=\mathbb{S}_{\lambda}(V)$ of $\SU(n)$ with corresponding partition $\lambda=(\lambda_1,\ldots,\lambda_{n-1})$, write $a_i:=\lambda_i-\lambda_{i+1}$ for each $i \in [n-1]$. We call the vector $(a_1,\ldots, a_{n-1})$ the \emph{step vector} of the representation $\rho$ (or, abusing terminology slightly, the {\em step vector} of the partition $\lambda$). We order such vectors with respect to the lexicographic ordering, i.e.\ $\lambda >_{\text{lex}} \lambda'$ iff $\lambda_i > \lambda_i'$ where $i = \min\{j:\lambda_j \neq \lambda_j'\}$. For a (not necessarily irreducible) representation $\rho$ of $\SU(n)$, its step vector is defined to be the lexicographically largest step vector of an irreducible subrepresentation of $\rho$.

The step vector is better-behaved than the corresponding partition, with respect to taking duals and tensors. The dual of $\rho$ has the reversed step vector $(a_{n-1},\ldots ,a_1)$,
corresponding to the partition $(a_{1}+\cdots +a_{n-1},\cdots,a_{1}+a_{2} ,a_{1})$. Moreover, if $\rho_1$ and $\rho_2$ are two representations whose step vectors are $(a_1,\ldots, a_{n-1})$ and $(b_1,\ldots ,b_{n-1})$, then the step vector of their tensor product is  $(a_1+b_1,\ldots , a_{n-1}+b_{n-1})$.

We say that an irreducible representation $\rho = \mathbb{S}_{\lambda (V)}$ of $\SU(n)$ is \emph{efficient} if $\lambda_{\lfloor\frac{n}{2}\rfloor}=0$. Equivalently, $\rho$ is efficient if its step vector $w=(a_1,\ldots ,a_{n-1})$ has the property that its \emph{second half} $w'':=(a_{\lceil n/2 \rceil}, \ldots, a_{n-1})$ consists of zeros. 
We say that $\rho$ is \emph{dually-efficient} if its \emph{first half} $w':=(a_1, \ldots,\ldots,  a_{\lceil n/2 \rceil-1})$ consists of zeroes. Alternatively, in terms of the corresponding partition $\lambda$, the irreducible representation $\rho$ is dually-efficient if $\lambda_1=\lambda_2 = \cdots = \lambda_{\lceil n/2 \rceil -1}$. (Note that the dual to each efficient representation is dually-efficient, and if $n$ is odd, the converse also holds. For $n$ even, our definition leads to a somewhat arbitrary choice of how to handle the middle part of the step vector, but this does not matter when the level is smaller than $n/2$.)
We call the partition $\alpha$ with step vector $w'$ the {\em efficient truncation} of $\lambda$ and the partition $\beta$ with step vector $w'$ the {\em dually-efficient truncation} of $\lambda$. 
 
 \begin{definition}
 We define the \emph{total level} of a representation $\rho=\mathbb{S}_{\lambda}(V)$ with step vector $(a_1,\ldots,a_{n-1})$ to be \[\sum_{i=1}^{n-1} a_i \min\{i,n-i\}.
 \]    
 \end{definition}

For each $d \in \mathbb{N} \cup \{0\}$, let $\tilde{\mathcal{L}}_d$ denote set of irreducible representations of $\SU(n)$ with total level $d$. Our next aim is to show that $\mathcal{L}_d=\tilde{\mathcal{L}}_d$ for all $0 \leq d < n/2$. For this, we first need the following.

\begin{lemma}\label{lem:Littlewood Richardson}
Let $\mathbb{S}_\lambda(V)$ be an irreducible representation of $\SU(n)$ of total level $d$, with $\lambda$ having $\alpha$ as its efficient truncation and $\beta$ as its dually-efficient truncation. Then the representation $\mathbb{S}_\alpha(V)\otimes \mathbb{S}_{\beta}(V)$ can be decomposed as a direct sum of one copy of the representation $\mathbb{S}_{\lambda}(V)$ and some other irreducible representations, all of which have total level less than $d$. 
\end{lemma}

\begin{proof}
This follows from the Littlewood-Richardson rule, e.g.\ as given in Fulton and Harris \cite[Section A.8]{fulton-harris}. Indeed, by the Littlewood-Richardson rule, writing $\alpha = (\alpha_1,\ldots,\alpha_{\ell})$, the irreducible constituents of $\mathbb{S}_{\beta}(V) \otimes \mathbb{S}_{\alpha}(V)$ are exactly those $\mathbb{S}_{\lambda}(V)$ such that the Young diagram of $\lambda$ can be produced by the following algorithm. Take the Young diagram of $\beta$, and first add $\alpha_1$ new boxes to the rows (in such a way as to produce the Young diagram of another partition, but with no two of the $\alpha_1$ added boxes being added to the same column), and place a `1' in each of these $\alpha_1$ new boxes. Then add a further $\alpha_2$ boxes to the rows (again in such a way as to produce the Young diagram of another partition, but with no two of the $\alpha_2$ added boxes being added to the same column), and place a `2' in each of these $\alpha_2$ new boxes. Continue in this way (so that at the last step, $\alpha_{\ell}$ new boxes are added). Now consider the sequence of length $\alpha_1 +\ldots+\alpha_{\ell}$ formed by concatenating the reversed rows of newly-added boxes, and check that if one looks at the first $t$ entries in this sequence (for any $t$ between 1 and $\alpha_1+\ldots+\alpha_{\ell}$), the integer $p$ appears at least as many times as the integer $p+1$ among these first $t$ entries, for any $1 \leq p < \ell$. If this `concatenation' condition holds, keep the Young diagram / partition; if not, reject it.

It is easy to check that the only way of performing this algorithm in such a way as to obtain a partition of level at least $d$, is to produce the (Young diagram of the) partition $\lambda$ itself: the first $\alpha_1$ new boxes must all be added to the first row of the Young diagram of $\beta$, the second $\alpha_2$ new boxes must all be added to the second row, and so on. Indeed, if at the $j$th stage (when adding $\alpha_j$ new boxes containing the integer $j$), any box is added to a row above the $j$th row, then (inductively) one sees that the concatenation condition would be violated, and moreover if some new box is added to a column of depth greater than $n/2$, then clearly, at the end of the process, less than $E$ cells will be in columns of depth at most $n/2$, and moreover less than $F$ cells will be missing from columns of depth greater than $n/2$, so the irreducible constituent of $\mathbb{S}_{\beta}(V) \otimes \mathbb{S}_{\alpha}(V)$ which is obtained, will have level less than $d$.
\end{proof}

We note the following consequence.

\begin{lemma}\label{lemma:index}
Let $\lambda$ be a partition with $n-1$ rows. Let $\alpha$ be its efficient truncation and $\beta$ its dually-efficient truncation. Write $E=\sum_{i=1}^{\lceil n/2\rceil}i\cdot a_i$ and $F= \sum_{i=1}^{\lceil n/2\rceil - 1}i a_{n-i}$. Then   
the representation $\mathbb{S}_\alpha(V)$ is a subrepresentation of $V^{\otimes E}$, the representation $\mathbb{S}_{\beta}(V)$ is a subrepresentation of $(V^{*})^{\otimes F}$, and the representation $\mathbb{S}_{\lambda}(V)$ is a subrepresentation of 
\[V^{\otimes E}\otimes (V^{*})^{\otimes F}.\]
\end{lemma}
\begin{proof}
The first assertion, concerning $\mathbb{S}_\alpha(V)$, is immediate from the construction of the Weyl module, since $\alpha$ is a partition of the integer $E$. The second assertion, concerning $\mathbb{S}_\beta(V)$, follows from the first, together with the fact that taking duals reverses the step vector. The third assertion, concerning $\mathbb{S}_{\lambda}(V)$, now follows from the previous lemma.  
\end{proof}

\begin{lemma}\label{The notions of level agree}
For all $0 \leq d < n/2$, we have $\mathcal{L}_d=\tilde{\mathcal{L}}_d$.
\end{lemma}
\begin{proof}
By Lemma \ref{lemma:index}, we have $\tilde{\mathcal{L}}_d\subseteq \mathcal{L}_d$ for all $0 \leq d < n/2$. It therefore suffices to show that for each $0\leq d<n/2$,
\[\bigcup_{i=0}^{d}\mathcal{L}_i \subseteq \bigcup_{i=0}^{d}\tilde{\mathcal{L}}_i.\] 
So suppose that $\rho$ is an irreducible representation of $\SU(n)$ of level at most $d$. Our goal is to show that the total level of $\rho$ is at most $d$. Let $E+F\le d$ be such that $\rho$ is a subrepresentation of $V^{\otimes E}\otimes (V^{*})^{\otimes F}$. Just as in the proof of Lemma \ref{lemma:v tensor d}, we may decompose $V^{\otimes E}$ into a direct sum of submodules of the form $\mathbb{S}_\alpha(V)$ with $\alpha \vdash E$, and likewise we can decompose $(V^{*})^{\otimes F}$ into a direct sum of submodules of the form $\mathbb{S}_\beta(V^*)$ with $\beta \vdash F$. Hence, we may assume that $\rho$ is isomorphic to a subrepresentation of $\mathbb{S}_{\alpha}(V)\otimes \mathbb{S}_{\beta}(V^*) \cong \mathbb{S}_{\alpha}(V) \otimes (\mathbb{S}_{\beta}(V))^*$ for some $\alpha \vdash E$ and $\beta \vdash F$. Now as $d<n/2$, both $\mathbb{S}_\alpha(V)$ and $\mathbb{S}_\beta(V)$ are efficient, so $(\mathbb{S}_{\beta}(V))^*$ is dually efficient, and therefore we may apply Lemma \ref{lem:Littlewood Richardson} to deduce that the total level of $\rho$ is at most the sum of the total levels of $\mathbb{S}_{\alpha}(V)$ and $(\mathbb{S}_\beta(V))^*$, which is $E+F$, as required.     
\end{proof}

\subsubsection*{Comfortable $d$-juntas}

We say that a function $f:\SU(n) \to \mathbb{C}$ is a {\em comfortable $d$-junta} if it is a complex linear combination of monomials of the form
$$X \mapsto \prod_{j=1}^{d} (x_{j,\sigma(j)})_{q_j\text{-part}}$$
where $\sigma\in S_d$ and $q_1,\ldots,q_d\in \{\textbf{imaginary}, \textbf{ real}\}$.
\remove{
Let $I,J\subseteq [n]$. We denote by $A[I\times J]$ the minor of $A$ corresponding to the sets $I,J.$ We will also write $[m,n]$ for the interval of integers between $m$ and $n$.
}

We start by showing that the comfortable $d$-juntas are in $V_{=d}$, i.e. that they are degree $d$ polynomials that are orthogonal to all polynomials of degree $\le d-1$. (Here, the degree is in terms of the matrix entries and their complex conjugates, or equivalently, in terms of the real and the imaginary parts of the matrix entries.)

\begin{lemma}\label{lem: Comfortable polynomials are pure}
    Every comfortable $d$-junta belongs to $V_{=d}$.
\end{lemma}
\begin{proof}
    Every comfortable $d$-junta clearly lies in $V_{\le d}$. It suffices to show that any such is orthogonal to all polynomials (in the matrix entries and their complex conjugates) of degree $\le d-1$. Let $T$ be one of the degree $d$ monomials that appear in our comfortable $d$-junta of degree $\le d$. Let $S$ be an arbitrary monomial of degree $\le d-1$. It suffices to show that $S$ and $T$ are orthogonal.
    Let $k_1,\ldots,k_d$ be the rows of the matrix-entry variables appearing in the monomial $T$. Since $S$ is a monomial of degree less than $d$, not all of these rows can appear amongst the variables in $S$; without loss of generality, assume row $k_1$ does not. Let $k_0$ be a row that appears neither in the variables in $S$ nor those in $T$ (such exists, since $d \leq n/2$, so $d+(d-1) < n$). Let $U_0$ be the diagonal unitary matrix with ones everywhere on the diagonal except in the $(k_0,k_0)$ and $(k_1,k_1)$ entries, with a $-i$ in the $(k_0,k_0)$ entry and an $i$ in the $(k_1,k_1)$ entry. If $X$ is distributed according to the Haar measure on $\SU(n)$, then so is $U_0X$, but multiplying $X$ by $U_0$ simply multiplies $\overline{S}T$ by $i$, so $\mathbb{E}_X[\overline{S}(X)T(X)] = \mathbb{E}_X[\overline{S}(U_0 X) T(U_0 X)] =  i\mathbb{E}_X[\overline{S}(X)T(X)]$ and therefore $\mathbb{E}_X[\overline{S}T] = 0$ as required.
\end{proof}

For the proof of the next lemma we need another definition. We say a polynomial $p:\SU(n) \to \mathbb{C}$ in the matrix entries of $X \in \SU(n)$ is {\em weakly-comfortable} if it is a complex linear combination of monomials of the form
$$X \mapsto \prod_{k=1}^{\ell} x_{i_k,j_k},$$
where $i_1,\ldots,i_{\ell}$ are distinct and $j_1,\ldots,j_{\ell}$ are distinct.

\begin{lemma}
\label{lemma:su_n weyl-useful}
Let $\rho$ be a representation of $\SU(n)$ of total level $d$, where $0 \leq d\leq n$. Then $W_{\rho}$ contains
a comfortable $d$-junta. \remove{Moreover, this junta spans an irreducible left submodule of $W_\rho$.}
\end{lemma}
\begin{proof}
Let $\rho$ be an irreducible representation of $\SU(n)$ of total level $d$. Let $\lambda$ be the corresponding integer partition (with at most $n-1$ rows). Let $\alpha$ (respectively $\beta$) be its efficient part and dually efficient part respectively. Note that the Young diagram of $\beta^*$ has $F$ cells, and all are in columns of depth at most $n/2$. As in the $\SO(n)$ case, there exists a weakly-comfortable, homogeneous polynomial $P$ of total degree $E$, such that $P \in W_{\rho_{\alpha}}$, namely the polynomial
$$P_\alpha(X) := \langle X((e_1 \otimes  \ldots \otimes e_d)c_{\alpha}),e_1 \otimes \ldots \otimes e_d  \rangle,$$
where $c_{\alpha}$ is the Young symmetrizer corresponding to $\alpha$; writing $c_{\alpha} = \sum_{\sigma \in S_d} \epsilon_\sigma \sigma$, where $\epsilon_\sigma \in \{-1,0,1\}$ for each $\sigma \in S_d$, we have
$$P_\lambda(X) = \sum_{\sigma \in S_d}\epsilon_\sigma \prod_{i=1}^{d}x_{i \sigma (i)}.$$

Similarly, there exists a non-zero, weakly-comfortable, homogeneous polynomial $Q$ of total degree $F$, such that $Q \in W_{\rho_{\beta^*}}$. Since $E+F = d \leq n$, we may further take $P$ to depend only upon matrix entries in the top $E$ by $E$ minor, and $Q$ to depend only upon matrix entries in the minor $[E+1,E+F]\times [E+1, E+F]$.

 Note that $P\overline{Q}$ is a comfortable $d$-junta and is spanned by the matrix coefficients of $\rho_\alpha \otimes \rho_\beta$. We claim that, in fact, $P \overline{Q}$ is spanned by the matrix coefficients of $\rho_{\lambda}$. This follows immediately from the fact that, firstly, by Lemma \ref{lem:Littlewood Richardson}, $\rho_\alpha \otimes \rho_\beta$ can be decomposed into a direct sum of a copy of $\rho_{\lambda}$ and some other irreducible representations of total level less than $d$, and that secondly, by Lemma \ref{lem: Comfortable polynomials are pure}, the comfortable $d$-junta $P\overline{Q}$ is orthogonal to $V^{\SU(n)}_{\leq d-1}$.

\remove{We now define a map $\psi\colon V^{* \otimes E}\otimes V^{\otimes F} \to L^2(\SU(n))$ by identifying $V=(V^*)^*$ with linear functionals over the conjugates of the input and $V^*$ with linear functionals over $\mathbb{C}^n.$ When we tensorize such a map we obtain multilinear polynomials in the inputs and the conjugates. We then take such a multilinear polynomial and plug in the first $d$ columns of the input matrix to obtain a function in $L^2(G)$. It is easy to see that $\psi$  is a $G$-morphism. We may now restrict $\psi$ to $S_{\alpha}(V^*)\otimes S_{\beta}(V)$, and look at the image of $\psi(P\otimes Q).$ First we note that by expanding to real and imaginary parts  $\psi(P\otimes Q)$ is a comfortable $d$-junta.

On the other hand, by Lemma \ref{lem:Littlewood Richardson},  $\psi(P\otimes Q)$ lies in the sum of $V_{d_1}$ with one copy of $\rho^*$ inside $L^2(G)$. By Lemma \ref{lem: Comfortable polynomials are pure} the function $\psi(P\otimes Q)$ is orthogonal to all polynomials of degree $\le d-1$. This shows that $\psi(P\otimes Q)$ spans an irreducible representation isomorphic to $\rho^*$ and it therefore lies in $W_\rho.$}
\end{proof}

\subsection{Obtaining strong quasirandomness for $\SU(n)$.}
\begin{lemma}
    Let $k,n \in \mathbb{N}$ with $k \leq n/2$, and let
    $$P = P(X_{11},X_{12},\ldots,X_{kk}) \in \mathbb{C}[X_{11},X_{12},\ldots,X_{kk}] \setminus \{0\}$$
    be a multivariate polynomial in the variables $(X_{i,j})_{i,j \in [k]}$ that is not the zero polynomial. Let $\pi:\SU(n) \to \mathbb{C}^{k \times k}$ denote projection onto the top-left $k$ by $k$ minor of a matrix in $\SU(n)$. Then $P \circ \pi$ cannot vanish on all of $\SU(n)$.
\end{lemma}
\begin{proof}
The image of $\pi$ is easily seen to have a nonempty interior inside $\mathbb{C}^{d\times d}$ (indeed, a suitably small open neighbourhood of $0$ is contained in the image of $\pi$). Since $P$ is a nontrivial polynomial in the variables $X_{11},X_{12},\ldots,X_{kk}$, it cannot vanish on all of a nonempty open subset of $\mathbb{C}^{k \times k}$.
\end{proof}

The following lemma is 
analogous to Lemma~\ref{lemma:lb1son}, and allows us to lower-bound the dimensions of irreducible representations of not-too-large degree. We prove it by considering comfortable $d$-juntas. 

\begin{lemma}
\label{lemma:lb1sun}
If $\rho$ is an irreducible representation of $\SU(n)$ of level $d$ where $0 \leq d < n/2$, then
$$\dim(\rho) \geq \binom{\lfloor n/2 \rfloor }{d}.$$
\end{lemma}
\begin{proof}
Let $P$ be a non-zero comfortable $d$-junta contained in $W_{\rho}$. Recall that the left action of $\SU(n)$ on $L^2(\SU(n))$ is defined as follows: for $A \in \SU(n)$ and $f \in L^2(\SU(n))$, we define $L_A f \in L^2(\SU(n))$ by $L_A f(X) = f(AX)$. Since $P \in W_{\rho}\setminus \{0\}$, the set $\{L_A P: A \in \SU(n)\}$ is contained in a left submodule $V$ of $L^2(\SU(n))$ which is isomorphic to the representation $\rho$. In particular, for any even permutation $\sigma$, $L_{A(\sigma)}P \in V$, where $A(\sigma)$ is the permutation matrix corresponding to $\sigma$; explicitly, $(A_{\sigma})_{i,j} = 1_{\{\sigma(i)=j\}}$ for each $i,j \in [n]$. Note that $L_{A(\sigma)}P$ is the polynomial obtained from $P$ by replacing the variable $X_{i,j}$ with the variable $X_{\sigma(i),j}$, for all $i,j \in [d]$; write $\sigma P: = L_{A(\sigma)}P$, for brevity. For each subset $S\in \binom{\lfloor n/2 \rfloor }{d}$, choose an even permutation $\sigma_S$ sending $[d]$ to $S$. The polynomials $\sigma_S P$ are linearly independent as polynomials (as for any two distinct sets $S \neq S'$, the set of monomials appearing in $\sigma_S P$ is disjoint from the set of monomials appearing in $\sigma_{S'}P$); moreover, each polynomial $\sigma_S P$ depends only upon variables in the top left $[\lfloor n/2 \rfloor]$ by $[\lfloor n/2 \rfloor]$ minor. It follows from the previous lemma, applied with $k = \lfloor n/2\rfloor$, that the polynomials $\sigma_S P$ are linearly independent as elements of $L^2(G)$, and therefore $\dim(V) \geq {\lfloor n/2\rfloor \choose d}$, as required.  

\end{proof}
We remark that above method for lower-bounding the dimensions of the irreducible representations of $\SU(n)$ works just as well for $\SO(n)$ and $\Sp(n)$, so we could have used it in place of Lemmas \ref{lemma:lb1son} and \ref{lemma:lb1spn} to give an alternative, self-contained proof of the $c$-strong-quasirandomness of $\SO(n)$ and of $\Sp(n)$.

The following lemma is analogous to Lemma~\ref{lemma:lb2son} and lower-bounds the dimensions of irreducible representations of $\SU(n)$ with high levels. We defer the proof till the Appendix. 
\begin{lemma}\label{lem:su2lb}
  There exists an absolute constant $c>0$ such that if $d\ge n/4$ and $\rho$ is an irreducible representation of $\SU(n)$ of level $d$, then
  ${\sf dim}(\rho)\geq 2^{cn}$.
\end{lemma}

Lemmas \ref{lemma:lb1sun} and \ref{lem:su2lb} immediately yield the strong quasirandomness of $\SU(n)$. 
\begin{lemma}
For each $n \geq 2$, the group $\SU(n)$ is $c$-strongly-quasirandom as an $n$-graded group, where $c>0$ is an absolute constant.
\end{lemma}

\section{Simply connected compact Lie groups are fine}\label{sec:hyp_from_curv}
In this section, we prove Theorem~\ref{thm:Every simple groups is fine}. The proof has two parts. In the first part, we identify a natural noise operator $U_{\delta}$ on $L^2(G)$ for the groups $G=\SU(n),\Sp(n),\Spin(n)$ which is guaranteed to satisfy a certain hypercontractive inequality, thanks to the fact that the Ricci curvature of these groups is bounded from below. (We note that this noise operator $U_{\delta}$ is not quite the same as the Beckner operator $T_{\delta, r}$ that we defined earlier.) The second part consists of inferring the weak hypercontractivity of the operator $T_{\delta, r}$ from the hypercontractive inequality for $U_{\delta}$. We accomplish that by analyzing the eigenvalues of $U_\delta$ and showing that they are all larger than the eigenvalues of the operator $T_{\delta^C,r}$ for some absolute constant $C>1$. This will allow us to write $T_{\delta^C,r}=U_{\delta}S$ for a linear operator $S$ on $L^2(G)$ satisfying $\|S\|_{2\to 2}\le 1.$ We will thus have 
\[
\|T_{\delta^C,r}\|_{2\to q}\le \|S\|_{2\to 2}\cdot \|U_{\delta}\|_{2\to q} \le 1,
\]
as needed.

\subsubsection*{The hypercontractive inequality}
Here we rely on concepts from differential geometry, such as a Riemannian metric, the Laplace--Beltrami operator, and the Ricci curvature/tensor. We use the notation of 
Anderson, Guionnet and Zeitouni \cite[Sections E and F]{zeitouni-book}, and we refer the reader to that work for more details.
\remove{ is a compact, connected, $N$-dimensional Riemann manifold (equipped with Riemannian metric $g$, which is simply an inner product $g_p$ on the tangent space to $M$ at $p$, for each $p \in M$), we let $\Delta := \nabla \cdot \nabla$ denote the Laplace--Beltrami operator on the space of twice-differentiable complex-valued functions on $M$. The Laplace--Beltrami operator $\Delta$ is, as in the case of $\mathbb{R}^n$ with the Euclidean metric, the divergence of the gradient; this is given, in local coordinates, by
$$\Delta : = \nabla \cdot \nabla = \sum_{i,j=1}^{N}g^{ij} \nabla_i \nabla_j = \sum_{i,j=1}^{N} g^{ij}\left( \frac{\partial^2}{\partial x^i \partial x^j} - \sum_{k=1}^{N} \Gamma_{ij}^{k}\frac{\partial}{\partial x^k}\right),$$
where $g_{ij} = g(\partial / \partial x^i,\partial/\partial x^j)$, $(g^{ij})$ is the matrix inverse to $(g_{ij})$, $\Gamma_{ij}^{k}$ denotes the {\em Cristoffel symbol}, defined by
$$\nabla_i (\partial / \partial x^j) = \sum_{k=1}^{N} \Gamma_{ij}^{k}( \partial / \partial x^k),$$
and $\nabla_i := \nabla_{\partial/\partial x^i}$ denotes the {\em Levi-Civita connection} (or {\em covariant derivative}).
}

The compact Lie groups are equipped with the structure of a bi-invariant Riemannian manifold $(M,g)$; this is unique up to normalization if the compact Lie group is simple. Once a normalization is set, and denoting by $\Delta$ the Laplace--Beltrami operator, it is also known that  the Hilbert space $L^2(M)$ has an orthonormal basis of eigenvectors of $\Delta$, that $\Delta$ is self-adjoint and negative semidefinite, and that $0$ is a simple eigenvalue of $\Delta$ (with the constant functions as corresponding eigenvectors). For a given compact Lie group $G$ with a Riemann manifold structure, we let $u_0,u_1,u_2,\ldots$ be such a basis, with $0 = \lambda_0 > \lambda_1 \geq \lambda_2 \geq \ldots$ being the corresponding eigenvalues, so that $\lambda_i <0$ for all $i \geq 1$, and with $u_0$ being the constant function with value 1. For any $f \in L^2(M)$, we may write $f$ uniquely in the form
$$f=\sum_{i=0}^{\infty} c_iu_i,$$
where $c_i \in \mathbb{R}$ for each $i \geq 0$ (we have $c_i = \langle f,u_i \rangle$ for each $i \geq 0$). For $\delta >0$, we define the {\em noise operator} $U_{\delta}$ by
$$U_\delta:\ L^2(M) \to L^2(M);\quad U_\delta(f) = \sum_{i = 0}^{\infty} c_i \delta^{-\lambda_i} u_i,$$
for $f=\sum_{i=0}^{\infty} c_iu_i$.
We note that $U_{e^{-t}}$ is, in fact, the heat kernel corresponding to $\Delta$, which is the averaging operator with respect to the Brownian motion on the corresponding manifold.

For a Riemaniann manifold $(M,g)$ and  $C>0$, we say that $(M,g)$ has {\em Ricci curvature bounded from below by $C$} if for all points $p \in M$, the Ricci tensor $\textrm{Ric}_p(\cdot,\cdot)$ at $p$ satisfies
$$\text{Ric}_p(X,X) \geq C g_p(X,X)$$
for all tangent vectors $X$ at $p$.

The hypercontractive inequality we need is the following.
\begin{thm}\label{thm:hyp-general}
Let $C>0$ and let $(M,g)$ be a compact, connected Riemann manifold whose Ricci curvature is bounded from below by $C$, let $2 \leq p \leq q$ and let $f \in L^p(M)$. Then
$$\|U_{\delta}(f)\|_q \leq \|f\|_p\quad \forall\ 0 \leq \delta \leq \left(\frac{p-1}{q-1}\right)^{1/C}.$$
\end{thm}
As explained in Klartag and Regev \cite{regev-klartag} the Bakry-Emery criterion yields a log-Sobolev inequality (as given for example in  \cite[Corollary 4.4.25]{zeitouni-book} applied with $\Phi=0$), which implies a hypercontractive inequality by a theorem of Gross \cite[Theorem 6]{gross}). 

\remove{An intuitive explanation is as follows: as in the Euclidean case, by the Divergence Theorem, the Laplace--Beltrami operator, applied to a function $f$ and evaluated at a point $p$, measures how much the value of $f$ at a point $p$ differs from the `average' value of $f$ on small metric spheres centred at $p$. The higher the eigenvalue, the greater the local variation of the eigenfunction, hence the need to smooth by higher powers of $\delta$ (as indeed the operator $U_{\delta}$ does), in order to control the $q$-norm in terms of the $p$-norm.
}

\paragraph{Utilizing the Riemannian structure.} As mentioned above, the compact simple Lie groups have a unique (up to normalization) bi-invariant Riemannian manifold structure. In order to set it up one needs to assign an inner-product on the tangent space at the identity $\text{Id}$ of $G$, i.e., on the Lie algebra of $G$. (The inner product on all the other tangent spaces is then determined by using a push-forward with respect to left multiplication by the appropriate group element.)

The tangent space of $\Spin(n)$ at the identity is the Lie algebra of $\SO(n)$. As in \cite{zeitouni-book}, we equip it with the usual Euclidean / Hilbert-Schmidt norm, i.e. the norm of a matrix is the sum of the squares of its entries. This norm gives rise to a bi-invariant metric when applying push-forward maps to define the norm on all the other tangent spaces. The norm on the Lie algebras of $\SU(n)$ and $\Sp(n)$ is defined similarly by taking the sum of squares of the components of each entry.  

It is well-known (see~\cite[4.4.30]{zeitouni-book}), that the Ricci curvature of the simply connected compact Lie groups $\Spin(n),\SU(n), \Sp(n)$ is bounded from below by $(n-2)/4, \frac{n}{2},n+1$ respectively.  

\remove{
It is easy to check, and well-known, that the Ricci curvature of $\SO(n)$, equipped with the standard (bi-invariant) metric inherited from Euclidean space (and where the tangent space is identified with the Lie algebra of $\SO(n)$, i.e.\ the space of real skew-symmetric $n$ by $n$ matrices, equipped with the usual Euclidean norm $\|A\|_2^2 := \sum_{i,j=1}^{n} A_{ij}^2$ on $\mathbb{R}^{n\times n}$), is bounded from below by $(n-2)/4$.
}

Using Theorem~\ref{thm:hyp-general}, we can prove Theorem~\ref{thm:Every simple groups is fine}, assuming the following lemma. 
\begin{lemma}\label{lem: eigenvalues for the Laplacian}
There exists an absolute constant $C$, such that the following holds. Let $G$ be either $\SU(n), \Sp(n)$ or $\SO(n)$. Let $d \in \mathbb{N}$ with $0 \leq d<n/2 $ and let $\rho \in \mathcal{L}_d$ be a representation of level $d$. Then the corresponding eigenvalue of the Laplace--Beltrami operator for $G$ satisfies $\lambda_{\rho} \ge -C n d$.
\end{lemma}

The proof of this lemma uses a formula for the eigenvalue $\lambda_{\rho}$ in terms of the fundamental weight corresponding to $\rho$, which is well-known and appears e.g.\ in Berti and Procesi \cite{berti-procesi}. We defer the proof to the Appendix.

\begin{proof}[Proof of Theorem~\ref{thm:Every simple groups is fine}]
We have already established the strong quasirandomness for all the groups $G$ of the form $\Spin(n), \Sp(n),$ and $\SU(n)$. It remains to establish their weak hypercontractivity. By Lemma \ref{lem:quotient-inverse-hyper}, it suffices to prove that $\SO(n),\Sp(n)$ and $\SU(n)$ are weakly hypercontractive.

By Theorem \ref{thm:hyp-general} there exists an absolute constant $C_1$, such that   setting $\delta = \left(\frac{1}{q-1}\right)^{\frac{1}{nC_1}},$ we have  $\|U_{\delta}\|_{2\to q}\le 1.$ Let $\rho \in \mathcal{L}_d$ be of level $d$. Then by Lemma \ref{lem: eigenvalues for the Laplacian} we obtain that the eigenvalues of $U_\delta$ given by $\delta^{-\lambda_\rho}$ are $\ge q^{-Cd}$ for some absolute constant $C$. This implies that we may write $T_{q^{-C},r} = U_{\delta} \circ S$, where all the eigenvalues of $S$ are $\le 1.$ Hence, 
\[
\left\|T_{q^{-C},r}\right \|_{2\to q}\le \left \|U_{\delta}\right \|_{2\to q}\left\|S\right\|_{2\to 2}\le 1.
\]
This completes the proof of the theorem.
\end{proof}

\remove{
Indeed, in our case 
if $f\colon \SO(n)\to\mathbb{R}$ is a function and $\rho$ is a representation of $\SO(n)$, then $U_{\delta} f_{\rho} = \delta^{\lambda_{\rho}} f_{\rho}$ and 
so
\[
\norm{f_{\rho}}_2^2
=\inner{f}{f_{\rho}}
=\delta^{-\lambda_{\rho}}\inner{f}{U_{\delta} f_{\rho}}
\leq
\delta^{-\lambda_{\rho}}\norm{f}_{q/(q-1)}\norm{U_{\delta} f_{\rho}}_q,
\]
where we used H\"{o}lder's inequality. Taking $\delta = \left(\frac{1}{q-1}\right)^{1/C}$ we conclude, using Theorem~\ref{thm:hyp-general}, that
\[
\norm{f_{\rho}}_2^2\leq (q-1)^{\frac{\lambda_{\rho}}{C}}\norm{f}_{q/(q-1)}\norm{f_{\rho}}_2,
\]
and re-arranging gives
\[
\norm{f_{\rho}}_2^2 \leq (q-1)^{\frac{2\lambda_{\rho}}{C}}\norm{f}_{q/(q-1)}^2.
\]
Choosing $q = \log(1/\alpha)$ for $\alpha = \E[f]$ gives that 
\[
\norm{f_{\rho}}_2^2
\leq \alpha^2 \log^{\frac{2\lambda_{\rho}}{C}}(\alpha)
\leq \alpha^2 \log^{\frac{8\lambda_{\rho}}{n-2}}(\alpha),
\]
where we used the lower bound on $C$ the Ricci curvature of $\SO(n)$.

Therefore, Theorem~\ref{thm:level-d inequality_from_curv} would follow
once we establish an upper bound on $\lambda_{\rho}$.  
}

\section{Showing that \boldmath\texorpdfstring{$\Sp(n), \Spin(n)$}{Sp(n), Spin(n)} and \texorpdfstring{$\SU(n)$}{SU(n)} are good}\label{sec:coupling_introduce}

We now outline our coupling-based approach to showing that $\SO(n),\Sp(n)$ and $\SU(n)$ are good. (By Lemma \ref{lem:quotient-inverse-hyper}, the goodness of $\Spin(n)$ will follow from that of $\SO(n)$.) We focus on the case of $\SO(n)$, and discuss the necessary adaptations for $\SU(n)$ and $\Sp(n)$ in the Appendix.

\subsubsection*{Our coupling approach for proving hypercontractivity} 
Our approach is based on constructing a coupling between matrices $X$ sampled according to the Haar measure on $\SO(n)$ and matrices $Y\in\mathbb{R}^{n\times n}$
whose entries are independent standard Gaussians, with the intuition that the distributions of $\sqrt{n} X$ and $Y$ are `locally' close to one another.

We use this coupling to define a noise operator on $L^2(\SO(n),\mu)$: first we use the coupling to move a function to Gaussian space, then we apply the well-known Gaussian noise operator (the Ornstein--Uhlenbeck operator), and then we go back to $\SO(n)$ via the coupling. It turns out that a noise operator defined this way inherits the hypercontractive property
from the Gaussian noise operator (this is easy to see), so we get hypercontractivity `for free'. The real work of the proof is to show
that the eigenvalues of our noise operator on $L^2(\SO(n),\mu)$ that correspond to functions in $V_{=d}$ are not too small. We show 
these eigenvalues are at least $2^{-O(d)}$, provided $d\leq \delta\cdot n^{1/2}$, for a sufficiently small absolute constant $\delta>0$.

\subsection{The Gaussian noise operator, a.k.a.\ the Ornstein--Uhlenbeck operator}

In this section, we recall the definition of the Gaussian noise operator and several of its properties. For simplicity of notation, we present this theory for functions
in $L^2(\mathbb{R}^n,\gamma)$, however everything applies more generally to functions in $L^2(\mathbb{R}^{n\times m}, \gamma)$. (Here and elsewhere, we abuse
notation slightly and denote by $\gamma$ a Gaussian distribution, where the domain is clear from context.)
\begin{definition}
  For $\rho\in [0,1]$, we define $U_{\rho}\colon L^2(\mathbb{R}^{n},\gamma)\to L^2(\mathbb{R}^{n},\gamma)$ by
\[
U_\rho f (X)=\mathbb{E}_{Y\sim \gamma}[f(\rho X +\sqrt{1-\rho^2})Y].
\]
\end{definition}
It is a well known fact that $U_{\rho}$ is hypercontractive~\cite{Nelson}:
\begin{thm}\label{thm:Gaussian_noise_hypercontract}
 Let $f\colon \mathbb{R}^n \to \mathbb{R}$ and let $0\leq \rho \le \frac{1}{\sqrt{q-1}}$. Then $\|U_{\rho}f\|_{L^q(\gamma)} \le \|f\|_{L^2(\gamma)}$.
\end{thm}
Below we use $U_\rho$ to construct an operator $\mathrm{T}_{\rho}$ over $L^2(\sqrt{n}\SO(n))$ which is hypercontractive,
and on which we have lower bounds on the eigenvalues corresponding to low-degrees, thereby showing that the group $\SO(n)$ is good.

\subsection{Constructing the noise operator $\mathrm{T}_\rho$}
\label{sec:gs}
In this section, we design our noise operator $\mathrm{T}_{\rho}$ on $L^2(\SO(n))$. 
En route, we define auxiliary operators that act on both $L^2(\mathbb{R}^{n\times n}, \gamma^{n\times n})$ and $L^2(\sqrt{n}\SO(n), \mu)$.

\subsubsection*{Left and right multiplication by matrices from $\SO(n)$}
\begin{definition}
  For a matrix $U\in\SO(n)$, we define the operator $L_U$ acting both on $L^2(\mathbb{R}^{n\times n}, \gamma^{n\times n})$ and $L^2(\sqrt{n}\SO(n), \mu)$, as follows.
  For a function $f\colon \mathbb{R}^{n\times n}\to\mathbb{R}$, the function $L_U f\colon \mathbb{R}^{n\times n}\to\mathbb{R}$ is defined by
  \[
  L_U f(X) = f(UX).
  \]
  For a function $f\colon \sqrt{n}\SO(n)\to\mathbb{R}$, we similarly define $L_U f(X) = f(UX)$.
\end{definition}
We similarly define the operator $R_V$ corresponding to right multiplication.
\begin{definition}
  For a matrix $V\in\SO(n)$, we define the operator $R_V$ acting both on $L^2(\mathbb{R}^{n\times n}, \gamma)$ and $L^2(\sqrt{n}\SO(n), \mu)$, as follows.
  For a function $f\colon \mathbb{R}^{n\times n}\to\mathbb{R}$, the function $R_V f\colon \mathbb{R}^{n\times n}\to\mathbb{R}$ is defined by
  \[
  R_V f(X) = f(XV).
  \]
  For a function $f\colon \sqrt{n}\SO(n)\to\mathbb{R}$, we similarly define $R_V f(X) = f(XV)$.
\end{definition}

\subsubsection*{The Gram--Schmidt operators.}
Next, we define the operators $\Trow, \Tcol$ that capture our coupling and map $L^2(\SO(n),\mu)$ to $L^2(\mathbb{R}^{n\times n}, \gamma)$, as well as
their adjoint operators that go in the reverse direction. To do so, we use the Gram-Schmidt process.

Fix a matrix $X\in\mathbb{R}^{n\times n}$ and let $c_1,\ldots,c_n$ be its columns. Provided $\det(X) \neq 0$ (which for a Gaussian matrix happens with probability one), we may apply the Gram-Schmidt process on $(c_1,\ldots,c_n)$
to get an orthonormal set of vectors $\tilde{c}_1,\ldots,\tilde{c}_n$. Abusing notation slightly, we define the matrix ${\sf GS}_{{\sf col}}(X)\in \sqrt{n} \SO(n)$ as the matrix whose $i^{\text{th}}$
column is $\sqrt{n}\tilde{c}_i$ for all $i < n$ and whose $n$th column is either $\sqrt{n}\tilde{c}_n$ (if $\det(X)>0$) or $-\sqrt{n}\tilde{c}_n$ (if $\det(X)<0$); this is of course a (column-) dilation of the (column-) Gram-Schmidt matrix corresponding to $X$. Since the Gram-Schmidt process preserves the sign of the determinant, this matrix ${\sf GS}_{{\sf col}}(X)$ is indeed in $\sqrt{n}\SO(n)$. 

Similarly, letting $r_1,\ldots,r_n$ be the rows of $X$, we let $\tilde{r}_1,\ldots,\tilde{r}_n$ be the resulting set of vectors
by applying the Gram-Schmidt process on $(r_1,\ldots,r_n)$ and define the matrix ${\sf GS}_{{\sf row}}(X)$ as the matrix whose $i^{\text{th}}$
row is $\sqrt{n}\tilde{r}_i$ for all $i < n$, and whose $n$th row is either $\sqrt{n}\tilde{r}_n$ (if $\det(X)>0$) or $-\sqrt{n}\tilde{r}_n$ (if $\det(X)<0$). 

The dilated Gram--Schmidt processes above define couplings $(X,{\sf GS}_{{\sf col}}(X))$ and $(X,{\sf GS}_{{\sf row}}(X))$
between $\gamma$ and $\mu$, and we use these to define the operators $\Trow$ and $\Tcol$:
\begin{definition}
  We define $\Trow\colon L^2(\sqrt{n}\SO(n), \mu)\to L^2(\mathbb{R}^{n\times n}, \gamma)$ as follows.
  For a function $f\colon \sqrt{n}\SO(n)\to\mathbb{R}$, we define $\Trow f\colon \mathbb{R}^{n\times n}\to\mathbb{R}$ by
  \[
  \Trow f(X) = f({\sf GS}_{{\sf row}}(X)).
  \]
\end{definition}
\begin{definition}
  We define $\Tcol\colon L^2(\sqrt{n}\SO(n), \mu)\to L^2(\mathbb{R}^{n\times n}, \gamma)$ as follows.
  For a function $f\colon \sqrt{n}\SO(n)\to\mathbb{R}$, we define $\Tcol f\colon \mathbb{R}^{n\times n}\to\mathbb{R}$ by
  \[
  \Tcol f(X) = f({\sf GS}_{{\sf col}}(X)).
  \]
\end{definition}

\subsubsection*{The operator $\mathrm{T}_{\rho}$.}
The operators $\Tcol$ and $\Trow$ allow us to move from $L^2(\mu)$ to $L^2(\gamma)$. We can also go in the reverse direction, using their adjoints. It is easy to see, using Jensen's inequality, that $\Tcol^* U_\rho \Tcol$ has the same hypercontractive properties as $U_\rho$. Thus, it is natural 
to consider the operator $\Tcol^* U_\rho \Tcol$ as an analogue of the Gaussian noise operator, for $\sqrt{n}\SO(n)$. We do not know, however, how to bound from below the eigenvalues of $\Tcol^* U_\rho \Tcol$ so as to deduce
Theorem~\ref{thm:Super Bonami in O_n}. The reason is that to bound its eigenvalues, 
we (naturally) need some information about the eigenvectors corresponding to them, and we only know how
to obtain such information from classical
representation facts about $\SO(n)$. To use these facts (so as to ensure the eigenspaces are `nice', and easy to analyse), it is necessary that our 
operator commutes with the action of $\SO(n)$ {\em from both sides}. For the operator $\Tcol^* U_\rho \Tcol$ above, one can show that it commutes with multiplication from the left, i.e.\ with the operators $L_U$, but unfortunately, it does not commute with multiplication from the right. To overcome this, we obtain commutation with the action of $\SO(n)$ from the right with an averaging trick, which is an analogue of the famous Weyl unitary trick.
\begin{definition}
  We set $\mathrm{T}_\rho = \Expect{V\sim \SO(n)}{R_V^*\Tcol^*U_{\rho}\Tcol R_V}$.
\end{definition}

The following result asserts that $\mathrm{T}_{\rho}$ is hypercontractive, and it also gives lower bounds on its eigenvalues (which are required for deducing
our level $d$ inequalities).
\begin{thm}\label{thm:hypercontractivity in O_n}
    For each $\rho \in (0,1)$, the operator $\mathrm T_{\rho}$ is self adjoint on $L^2(\sqrt{n}\SO(n),\mu)$ and has the following properties:
    \begin{enumerate}
        \item $\mathrm{T}_\rho$ commutes with both left and right multiplication by matrices from $\SO(n)$.
        \item If $\rho \le \frac{1}{\sqrt{q-1}}$, and $f\in L^2(\sqrt{n}\SO(n),\mu)$, then $\| \mathrm{T}_{\rho} f\|_{L^q(\mu)} \le \|f\|_{L^2(\mu)}$.
        \item There exist absolute constants $\delta>0, C>1$, such that if $d\le \delta n^{1/2}$ and  $f\in V_d$, then
        \[
        \|\mathrm{T}_{\rho}f\|_{L^2(\mu)}\ge C^{-d}\rho^{d}\|f\|_{L^2(\mu)}.
        \]
    \end{enumerate}
    \end{thm}

    Let use show how Theorem \ref{thm:hypercontractivity in O_n} immediately implies that $\SO(n)$ and $\Spin(n)$ are good.

\begin{theorem}
    There exist absolute constants $c,C>0$ such that the $n$-graded groups $\SO(n)$ and $\Spin(n)$ are $(C,c)$-good.
\end{theorem}
\begin{proof}[Proof, given Theorem \ref{thm:hypercontractivity in O_n}]
    The strong quasirandomness of $\Spin(n)$ and $\SO(n)$ was already established in Section \ref{section: compacts are quasi}. It remains to show that they are $(cn^{1/2},C)$-hypercontractive for absolute constants $C,c>0$. Let $C$ be the constant in the statement of Theorem \ref{thm:hypercontractivity in O_n} and let $c=\delta$ where $\delta$ is the constant in the statement of Theorem \ref{thm:hypercontractivity in O_n}. By Lemma \ref{lem:quotient-inverse-hyper}, it suffices to show that $\SO(n)$ is $(cn^{1/2},C)$-hypercontractive. Let $T= T_{1/(C\sqrt{q}),\ cn^{1/2}}$ be the Beckner operator. By the monotonicity (in $\delta$) of $\|T_{\delta,r}\|_{2 \to q}$, it suffices to show that $\|T\|_{2\to q}\le 1$. Let us show (using Theorem \ref{thm:hypercontractivity in O_n}) that we may write $T= T_{1/\sqrt{q-1}}\circ S$, where $\|S\|_{2\to 2}\le 1$. Indeed, as $T_{1/\sqrt{q-1}}$ is self-adjoint and commutes with the action of $\SO(n)$ from both sides, it has the Peter-Weyl ideals $W_\rho$ as its eigenspaces.     
    By Part 3 of Theorem \ref{thm:hypercontractivity in O_n}, the eigenvalues of $T_{1/\sqrt{q-1}}$ corresponding to a representations $\rho$ of level $d$ are at least $(C\sqrt{q})^{-d}$, and thus are at least the corresponding eigenvalue of $T$ on $W_\rho$. This shows that the desired operator $S$ exists. We may now apply Part 2 of Theorem \ref{thm:hypercontractivity in O_n} to obtain: 
    \[
    \left\|T\right\|_{2\to q}\le \left\|T_{\frac{1}{\sqrt{q-1}}}\right\|_{2\to q}\left\|S\right\|_{2\to 2}\le 1.
    \]
\end{proof}
\subsubsection*{The noise operator $\mathrm{T}_{\rho}$ commutes with the action of $\SO(n)$ from both sides}
In this section, we establish Part 1 of Theorem \ref{thm:hypercontractivity in O_n}. We phrase it as a lemma.
\begin{lemma}\label{lem:commuting}
  The operator $\mathrm{T}_{\rho}$ commutes with the action of $\SO(n)$ from both sides.
\end{lemma}
\begin{proof}
  We show the commuting from the left and the right separately.
  \paragraph{Commuting from the left.}
  For this, it suffices to show that for each $U, V\in\SO(n)$, the operators $L_U$
  and $R_V^*\Tcol^*\T_{\rho}\Tcol R_V$ commute. It is easy to see that $L_U$ and $R_V$ commute. Hence,
  it suffices to show that $L_U$ and $\Tcol^*\T_{\rho}\Tcol$ commute.

  First note  that if $X\in\mathbb{R}^{n\times n}$
  is a matrix, it holds that  ${\sf GS}_{{\sf col}}(UX) = U{\sf GS}_{{\sf col}}(X)$. It follows that 
  \[
  L_U \Tcol^*f(X)=
  \Tcol^*f(UX)
  =\Expect{Y: {\sf GS}_{{\sf col}}(Y) = UX}{f(Y)}
  =\Expect{Z: {\sf GS}_{{\sf col}}(Z) = X}{f(UZ)}
  =\Tcol^* L_U f(X),
  \]
  so $L_U$ commutes with $\Tcol^{*}$. The adjointness immediately implies that $L_U$ also commutes with the operator $\Tcol.$ It is easy to see directly that $L_U$ commutes with the operator $U_{\rho}$, and so it commutes with the composition  $\Tcol^*\T_{\rho}\Tcol$.

  \paragraph{Commuting from the right.} Fix $V'\in \SO(n)$, then
  \[
   R_{V'} \mathrm{T}_\rho
   =
   \Expect{V\sim \SO(n)}{R_{V'} R_V^*\Tcol^*\T_{\rho}\Tcol R_V}
   =
   \Expect{V\sim \SO(n)}{R_{V V'^{t}}^*\Tcol^*\T_{\rho}\Tcol R_V}
  \]
  Making the change of variables $V\leftarrow V V'^{t}$, we obtain
  \[
  R_{V'} \mathrm{T}_\rho
  =\Expect{V\sim \SO(n)}{R_{V}^*\Tcol^*\T_{\rho}\Tcol R_{VV'}}
  =\Expect{V\sim \SO(n)}{R_{V}^*\Tcol^*\T_{\rho}\Tcol R_{V}}R_{V'}
  =\mathrm{T}_{\rho} R_{V'}.
  \qedhere
  \]
\end{proof}

\subsection{The noise operator $\mathrm{T}_{\rho}$ is hypercontractive}

In this section, we prove part 2 of Theorem \ref{thm:hypercontractivity in O_n}. To do so, we first
adopt a different point of view of the couplings defined by $\Trow$ and $\Tcol$ that will often be easier for us to work with.

\subsubsection*{The `Gaussian maker distribution'}
Rather than going from $Y\sim \gamma$ to $X\sim \mu$ by applying the Gram--Schmidt process on its columns and dilating by $\sqrt{n}$ (and flipping the sign of the last column if necessary), we can go the other way and
construct $Y$ from $X$. This is accomplished as follows. We define a pair of independent random variables $(X,G)$ such that $XG$ is distributed according to $\gamma$ and $X\sim \mu$. We call the distribution of $G$ the \emph{Gaussian maker distribution} and we abbreviate it to GMD.

\begin{definition}
  We define the {\em Gaussian maker distribution} to be the distribution of the upper-triangular matrix $G=(g_{ij})$ constructed as follows. First, independently choose one-dimensional Gaussians $g_{ij}\sim N(0,\frac{1}{n})$ of expectation zero and variance $1/n$, for each $i<j$. For each $i<n$, we independently choose $g_{ii}$ to be $1/\sqrt{n}$ times the (Euclidean) length of an $(n-i+1)$-dimensional Gaussian $z\sim (\mathbb{R}^{n-i+1},\gamma)$. We also independently choose $g_{nn}$ to be a standard Gaussian random variable, $z\sim N(0,1)$. Finally, we set $g_{ij}=0$ for all $j<i$.
\end{definition}

It is clear that when we sample a matrix $Y\sim \gamma$ and apply the (dilated) Gram-Schmidt process in Section \ref{sec:gs}, we get a matrix $X$ which is $\sqrt{n}$ times a matrix sampled from the Haar measure on $\SO(n)$ (provided we condition on the probability-one event that $\det(Y) \neq 0$, of course). We would like to show that $X^{-1}Y\sim \text{GMD}$, independently of $X$.
Indeed, this follows by choosing the columns of $X$ and $Y$ one after another. The first column of $Y$ is uniformly distributed according to $(\mathbb{R}^n,\gamma)$.
By rotational symmetry, its length and its normalization are independent, and therefore $X G e_1 = g_{11} X e_1$ is indeed distributed as $Y e_1$.
Note that $X e_2$ is independent of $g_{11}$, as it is a uniformly random unit vector orthogonal to $X_1$.
Thus, completing $X_1$ to a basis arbitrarily we obtain, by rotational invariance of the Gaussian distribution, that the correlation of $Y e_2$ with $X e_1$ is normally distributed.
After we present $Ye_2$ with respect to an extension of $\frac{1}{\sqrt{n}} Xe_1$ to an orthonormal basis, we see that the last $n-1$ coordinates are distributed as a random $n-1$ dimensional Gaussian.
This shows that indeed $Ye_2$ is distributed as $g_{22} X e_2 +g_{12} X e_1$. Continuing in this fashion column by column (being a little careful with the last column), we see that indeed $X^{-1}Y$ is distributed according to $\text{GMD}$, independently of $X$.
We thus have the following formulae for the adjoint operators $\Trow^{*},\Tcol^{*}$:
\begin{lemma}\label{lem:formula_for_adjoint}
For each $f \in L^2(\mathbb{R}^{n\times n},\gamma)$, we have
\begin{align*}
\Tcol^*f(X)&=\Expect{G\sim \text{GMD}}{f(XG)}& \forall X \in \SO(n),\\
\Trow^*f(X)&=\Expect{G\sim \text{GMD}}{f(G^t X)}& \forall X \in \SO(n).
\end{align*}
\end{lemma}

In addition to Lemma~\ref{lem:formula_for_adjoint} being useful on its own, it allows us to extend the domain of $\Tcol^*f$ to $\mathbb{R}^{n\times n}$.
Indeed, abusing notation, we shall often view $\Tcol^*f$ as a function from $\mathbb{R}^{n\times n}$ to $\mathbb{R}$, by using the above formula:
\[
\Tcol^*f(X)=\Expect{G\sim \text{GMD}}{f(XG)}.
\]
This allows us to view $\Tcol^{*}$ as an operator on Gaussian space $L^2(\mathbb{R}^{n \times n},\gamma)$, and we note that if $f\colon\mathbb{R}^{n\times n}\to\mathbb{R}$
has degree at most $d$, then $\Tcol^{*} f$ also has degree at most $d$, and thus $V_d$ is an invariant subspace of $\Tcol^{*}$.

Another consequence of Lemma~\ref{lem:formula_for_adjoint} is that the operators $\Tcol^{*}, \Trow^{*}$ do not increase $q$-norms:
\begin{fact}\label{fact:tcol_preserve}
  The following hold for all $q\geq 1$:
  \begin{enumerate}
    \item The operators $\Tcol$ and $\Trow$ preserve $q$-norms.
    \item The operators $\Tcol^{*}$ and $\Trow^{*}$ (from $L^q(\mu)$ to $L^q(\gamma)$) cannot increase $q$-norms.
  \end{enumerate}
\end{fact}
\begin{proof}
We prove both items for $\Tcol$, as the proof for $\Trow$ is exactly analogous.

For the first item, we note that
\[
\norm{\Tcol f}_{L^q(\gamma)}^q
=\Expect{Y\sim \gamma}{\card{\Tcol f(Y)}^q}
=\Expect{Y\sim \gamma}{\card{f({GS}_{{\sf col}}(Y))}^q}
=\Expect{X\sim \mu}{\card{f(X)}^q}
=\norm{f}_{L^q(\mu)}^q.
\]
For the second item, we use Jensen's inequality:
\[
\|\Tcol^*f\|_{L^q(\mu)}^q=\Expect{X\sim \mu}{\left|\Expect{G\sim \text{GMD}}{f(XG)}\right|^q}\leq \Expect{X,G}{|f(XG)|^q}
=\Expect{Y\sim \gamma}{\card{\Tcol f(Y)}^q}=\|f\|_{L^q(\gamma)}^q.\qedhere
\]
\end{proof}

\subsubsection*{Deducing hypercontractivity.}
We now show that the operator $\mathrm{T}_{\rho}$ is hypercontractive, proving Part~3 of Theorem~\ref{thm:hypercontractivity in O_n}.

\begin{lemma}\label{lem:hypercontractivity}
  For all $0\leq \rho\leq \frac{1}{\sqrt{q-1}}$ and $f\colon\SO(n)\to\mathbb{R}$ we have $\norm{\mathrm{T}_{\rho}f}_{L^q(\mu)}\le \norm{f}_{L^2(\mu)}$.
\end{lemma}
\begin{proof}
By the triangle inequality and the fact that $R_V$ preserves the $L^r$ norm for all $r$, it suffices to show that
\[
\norm{\Tcol^* U_\rho \Tcol f}_{L^q(\mu)}\leq \norm{f}_{L^2(\mu)}.
\]
To see that this holds, we apply Fact~\ref{fact:tcol_preserve} and Theorem~\ref{thm:Gaussian_noise_hypercontract}:
\[
\|\Tcol^* U_\rho \Tcol f\|_{L^q(\gamma)}
\leq
\|U_\rho \Tcol f\|_{L^q(\gamma)}
\leq
\|\Tcol f\|_{L^2(\gamma)}
=\norm{f}_{L^2(\gamma)}.
\qedhere
\]
\end{proof}
\section{Comfortable polynomials on $\SO(n)$, and their properties.}

Recall that in Section~\ref{section: compacts are quasi} we defined  the {\em comfortable $d$-juntas} on $\SO(n)$ to be the multilinear polynomials of the form $X\mapsto \sum_{\sigma \in S_d}a_\sigma \prod_{i=1}^{d} x_{i,\sigma(i)}$, for $a_\sigma \in \mathbb{R}$. We also showed that for any irreducible representation $\rho$ of level $d$, where $0 \leq d < n/2$, the Peter-Weyl ideal $W_\rho$ contains a comfortable $d$-junta. In this section we define comfortable polynomials in general (the comfortable $d$-juntas are a special case). We then show that some of them are eigenfunctions of $\Tcol^*$ (or of $\Trow^*$).  

We use comfortable $d$-juntas as they are both easy to work with, and  each low degree eigenspace of $\mathrm{T}_\rho$ contains one; the latter is guaranteed by the following, since $\mathrm{T}_\rho$ commutes with the action of $\SO(n)$ from both sides.
\begin{claim}\label{claim:reduce_to_vddd}
  Let $T$ be a linear operator on $L^2(\SO(n))$ that commutes with the action of $\SO(n)$ from both sides, and let $0 \leq d< n/2.$ Then the space $V_{=d}$ is $T$-invariant, and each eigenspace of $T$ inside $V_{=d}$ contains a comfortable $d$-junta.
\end{claim}
\begin{proof}
A linear map from $L^2(\SO(n))$ to itself that commutes with the action of $\SO(n)$ from both sides is precisely an $\SO(n)\times \SO(n)$-homomorphism. As each $W_{\rho}$ is an irreducible $\SO(n)\times \SO(n)$-module and the $W_{\rho}$ are pairwise non-isomorphic, Schur's lemma implies that any $\SO(n)\times \SO(n)$-homomorphism from $L^2(\SO(n))$ to itself acts as a scalar multiple of the identity when restricted to $W_{\rho}$. Since the linear operator $T$ commutes with the action of $\SO(n)$ from both sides, each $W_{\rho}$ is contained in an eigenspace of $T$. Since the Peter-Weyl ideas span $L^2(\SO(n))$, each eigenspace of $T$ is a direct sum of some of the $W_{\rho}$'s. 

As $V_{=d}$ is a direct sum of finitely many Peter-Weyl ideals (each of which is $T$-invariant), $V_{=d}$ itself is $T$-invariant. The claim now follows from Fact \ref{fact:weyl-useful}, which implies that each of the Peter-Weyl constituents of $V_{=d}$ contains a comfortable $d$-junta. 
\end{proof}

\subsection{Comfortable polynomials}
Claim~\ref{claim:reduce_to_vddd} is important for us as 
it says that if we want to understand the 
eigenvalues of an operator $T$ that commutes
with the action of $\SO(n)$, it suffices 
to understand its action on low-degree polynomials.
 One can already carry out
some non-trivial analysis of our hypercontractive 
operator using this observation, however to push 
our analysis all the way to $d = \Theta(n^{1/2})$ 
we need to work 
with a restricted class of polynomials, 
which we call `comfortable polynomials'. 

\begin{definition}[Comfortable polynomial] We say a multilinear monomial in the matrix entries of $X\in \SO(n)$ is  
\emph{comfortable} if it only contains variables from the top left $\lfloor\frac n2\rfloor \times \lfloor\frac n2\rfloor$ minor of $X$, and it contains at most one variable from each row and at most one variable from each column. A polynomial in the matrix entries is said to be {\em comfortable} if
it is a linear combination of (multilinear) comfortable monomials.
\end{definition}

\label{sec:def_comf}

We may naturally index monomials in $L^2(\mathbb{R}^{n\times n},\gamma)$ by multisets of elements of $[n]\times [n]$.
\remove{
The first of which is $\{\ell_1 \cdot (i_1,j_1) \ldots \ell_r\cdot (i_r,j_r)\}$, by which we mean that the
where the pairs $(i_k,j_k)$ are distinct and $(i_k,j_k)$ appear in the multi-set $\ell_k$ times.

The first notation is
}
Given such a multiset $S=\{(i_1,j_1),\ldots,(i_r,j_r)\}$, where each pair $(i_k,j_k)$ may appear multiple times,
the corresponding monomial is
$H_{S}(X) := \prod\limits_{r=1}^{k}x_{i_r,j_r}$.

We now define the notion of comfortable polynomial.
\begin{definition}
  Let $S=\{(i_1,j_1),\ldots,(i_d,j_d)\}$ be a multiset of elements of $[n] \times [n]$; we define its {\em transpose} by $S^t:= \{(j_1,i_1),\ldots ,(j_d,i_d)\}$.
  \begin{enumerate}
    \item We say $H_S$  is \emph{row comfortable} if $i_1,\ldots ,i_d$ are distinct and all $\leq n/2$.
    \item We say $H_S$ is \emph{column comfortable} if $H_{S^t}$ is row comfortable, i.e.\ if  $j_1,\ldots,j_d$ are distinct and all $\leq n/2$.
    \item Finally, we say $H_S$ is {\em comfortable} if it is both row comfortable and column comfortable.
  \end{enumerate}
\end{definition}
\begin{definition}
  We say a polynomial is {\em row comfortable} if it lies in the span of the row comfortable monomials; similarly, we say it is {\em column comfortable} if it lies in the span of the column comfortable monomials, and that it is  {\em comfortable} if it is both row comfortable and column comfortable.
\end{definition}

We also need to define row- and column- comfortable $d$-juntas, generalizing the notion of comfortable $d$-juntas.
\begin{definition}
A {\em row- (respectively\ column-) comfortable $d$-junta} is a row (respectively column) comfortable polynomial whose monomials contain variables only from the top left $d\times d$ minor.
\end{definition}

Finally, we need to define row comfortable $d$-row-juntas and column comfortable $d$-column-juntas, which are slight weakenings of the notion of a junta that will be useful.

\begin{definition}
A {\em row comfortable $d$-row-junta} is a row comfortable polynomial whose monomials contain variables only from the first $d$ rows. A {column comfortable $d$-column-junta} is a column comfortable polynomial whose monomials contain variables only from the first $d$ columns.
\end{definition}


\subsubsection*{Comfortable monomials are eigenvectors of $\Tcol^{*}$/$\Trow^{*}$}

The following claim shows that the column/row comfortable monomials are eigenvectors of $\Tcol^{*}$/$\Trow^{*}$.
\begin{claim}\label{claim:comf_ev}
  There exists an absolute constant $C\geq 1$ such that the following holds.
  For each multiset $S=\{(i_1,j_1),\ldots, (i_d,j_d)\}$ of elements of $[n]\times [n]$, there exists $\lambda_S>0$ such that:
  \begin{enumerate}
    \item If $H_S$ is column comfortable (i.e.\ the $j_k$ are distinct), then $\Tcol^{*} H_{S} = \lambda_{S} H_{S}$.
    \item If $H_S$ is row comfortable, then $\Trow^{*} H_{S^t} = \lambda_{S} H_{S^t}$.
    \item $C^{-d} \leq \lambda_S \leq 1$ for all $S$.
    \item If $S$ is supported on $[n]\times [d]$, then $\card{\lambda_S-1}= O(d^2/n)$.
  \end{enumerate}
\end{claim}
\begin{proof}
  It suffices to prove the first, third and fourth items, as the second follows from the first by taking transposes. By Lemma~\ref{lem:formula_for_adjoint}
  we have $(\Tcol^* H_S)(X)=\Expect{G\sim\text{GMD}}{H_S(XG)}$.
Using the fact that the entries of $XG$ corresponding to different columns are independent and that each $j_k$ appears at most once, we obtain
  \begin{align*}
  (\Tcol^*H_S)(X)&=\Expect{G\sim\text{GMD}}{H_S(XG)}\\
  &= \Expect{G\sim\text{GMD}}{\prod_{k=1}^{d}(XG)_{i_k,j_k}}\\
  & =\prod_{k=1}^d \Expect{G\sim\text{GMD}}{(XG)_{i_k,j_k}}\\
  & =\prod_{k=1}^d \left(X\Expect{G\sim\text{GMD}}{G}\right)_{i_k,j_k}.
  \end{align*}
  Observe that $\mathbb{E}_{G \sim \text{GMD}}[G]$ is a diagonal matrix with $(\mathbb{E}[G])_{j,j}$ being equal to $1/\sqrt{n}$ times the expectation of the length (= Euclidean norm) of an $(n-j+1)$-dimensional standard Gaussian random vector, for each $j < n$. For $m \in \mathbb{N}$, let $N(0,I_m)$ denote an $m$-dimensional standard Gaussian random vector, and let $\|N(0,I_m)\|_{\ell^2}$ denote its Euclidean norm. We have
   \[
  (\Tcol^*H_S)(X) = \prod_{k=1}^d X_{i_k,j_k}\left(\Expect{G\sim\text{GMD}}{G}\right)_{j_k,j_k} = H_S(X) \prod_{k=1}^d \left(\Expect{G\sim\text{GMD}}{G}\right)_{j_k,j_k} = \lambda_S H_S(X),\]
  where
  $$\lambda_S = n^{-d/2} \prod_{k=1}^{d} \mathbb{E} [\|N(0,I_{n-j_k+1})\|_{\ell^2}].$$
  To estimate the eigenvalues $\lambda_S$ we  need the following fact.

  \begin{fact}\label{fact:too_lazy_to_chernoff} For any $m \in \mathbb{N}$, we have
  $\sqrt{m} - \frac{1}{2\sqrt{m}} \leq \Expect{}{\norm{N(0, I_m)}_{\ell^2}}\leq \sqrt{m}$.
\end{fact}
\begin{proof}
  Let $Z_1,\ldots,Z_m\sim N(0,1)$ be independent and $Z = \sum\limits_{i=1}^{m} Z_i^2$ so that $\norm{N(0, I_m)}_{\ell^2} = \sqrt{Z}$.
  Then $\Expect{}{Z} = m$, so $\mathbb{E}[\sqrt{Z}] \leq \sqrt{\mathbb{E}[Z]} = \sqrt{m}$ by Cauchy-Schwarz, proving the upper bound. Secondly ${\sf var}(Z) = \sum\limits_{i=1}^{m} {\sf var}(Z_i^2) =m$,
  which implies that
  \[
  \Expect{}{(\sqrt{Z} - \sqrt{m})^2}
  =
  \Expect{}{\frac{(Z - m)^2}{(\sqrt{Z} + \sqrt{m})^2}}
  \leq
  \frac{1}{m}{\sf var}(Z)
  =1.
  \]
  Thus,
  \[
  2m - 2\sqrt{m}\Expect{}{\sqrt{Z}} = \Expect{}{Z} + m - 2\sqrt{m}\Expect{}{\sqrt{Z}}\leq 1,
  \]
  implying that
  \[
  \Expect{}{\sqrt{Z}}\geq \sqrt{m} - \frac{1}{2\sqrt{m}},
  \]
 proving the lower bound.
\end{proof}
Continuing the proof of the claim, the above fact yields $C^{-d} \leq \lambda_S \leq 1$ for all $S$. If $S$ is supported upon $[n]\times [d]$, then it yields
$$ \lambda_S \geq n^{-d/2}\left(\sqrt{n-d+1}-\frac{1}{2\sqrt{n-d+1}}\right)^d \geq 1-O(d^2/n),$$
completing the proof of Claim~\ref{claim:comf_ev}.
\end{proof}

\subsubsection*{Projections onto comfortable subspaces}

Define the operator $\Pi_{{\sf comf}}\colon L^2(\mathbb{R}^{n\times n},\gamma)\to L^2(\mathbb{R}^{n\times n},\gamma)$ to be orthogonal projection onto the linear subspace of comfortable polynomials. This projection has a neat Fourier expression:
\[
(\Pi_{{\sf comf}} f)(X) = \sum\limits_{H_\alpha\text{ a comfortable monomial}}\widehat{f}(\alpha) H_{\alpha}(X),
\]
where $\hat{f}(\alpha):=\langle f,H_{\alpha}\rangle_{L^2(\gamma)}$ for each Hermite polynomial $H_{\alpha}$. We also define the operators $\Pi_{{\sf comf}, d}$, $\Pi_{{\sf comf}, {\sf col}, d}$ and 
$\Pi_{{\sf comf}, {\sf row}, d}$ to be the orthogonal projections onto
the space of comfortable, row comfortable and column comfortable polynomials of degree at most $d$, respectively. Finally, if $S \subset [n]$ is a set of rows, we define $\Pi_{{\sf comf}, =S}$ to be the projection onto the subspace spanned by the comfortable homogeneous monomials of degree $|S|$ which depend only upon variables from the rows in $S$.

\subsection{Reducing Theorem~\ref{thm:hypercontractivity in O_n} to a statement about low-degree truncations of $\Tcol$.}
In this section, we prove Theorem \ref{thm:hypercontractivity in O_n} modulo the following lemma, asserting that on
the space of comfortable $d$-juntas, the operator $\Tcol R_V$ (for a typical $V$) preserves some of the mass of a function even after projection onto $V_d$.
\begin{lemma}\label{lem:main_lemma}
  There exist absolute constants $C>0$ and $\delta>0$ such that the following holds for all $d\leq \delta n^{1/2}$. For all comfortable $d$-juntas $f:\mathbb{R}^{n \times n} \to \mathbb{R}$, we have
  \[
  \Expect{V\sim\SO(n)}{\norm{(\Tcol R_V f)^{\leq d}}_{L^2(\gamma)}^2}
  \geq C^{-d}\norm{f}_{L^2(\mu)}^2.
  \]
\end{lemma}

We now show how to complete the proof of Theorem~\ref{thm:hypercontractivity in O_n}, assuming Lemma \ref{lem:main_lemma}.
\begin{proof}[Proof of Theorem~\ref{thm:hypercontractivity in O_n} (assuming Lemma~\ref{lem:main_lemma})]
Lemmas \ref{lem:commuting} and \ref{lem:hypercontractivity} give the first two items, so it remains to prove the third item. Namely, letting $f\in V_d$ we want to show that $\| \mathrm{T}_{\rho} f\|_{L^2(\mu)} \ge (c\rho)^{d}\| f\|_{L^2(\mu )}$. By  Claim~\ref{claim:reduce_to_vddd} we may assume that $f$ is a comfortable $d$-junta. By Cauchy--Schwarz, it is enough to show that
$\langle \mathrm{T}_{\rho} f, f \rangle_{L^2(\mu)} \ge (c\rho)^{d}\|f\|_2^2$, and we next show that the last assertion follows by Lemma~\ref{lem:main_lemma}.

Using the fact that $U_{\rho}=U_{\sqrt{\rho}}^*U_{\sqrt{\rho}}$, we have
\begin{align*}
\langle \mathrm{T}_\rho f,f\rangle_{L^2(\mu)}
= \Expect{V\sim \SO(n)}{\| U_{\sqrt{\rho}} \Tcol R_V f \|_{L^2(\gamma)}^2}
&\geq
\Expect{V\sim \SO(n)}{\| (U_{\sqrt{\rho}} \Tcol R_V f)^{\leq d} \|_{L^2(\gamma)}^2}\\
&\geq
\rho^{d}\Expect{V\sim \SO(n)}{\| (\Tcol R_V f)^{\leq d} \|_{L^2(\gamma)}^2}\\
&\geq
\rho^d C^{-d}\norm{f}_{L^2(\mu)}^2,
\end{align*}
 where we used Lemma~\ref{lem:main_lemma} and the fact that for each function $g$ of degree $\le d$ it holds that  \[\| U_{\sqrt{\rho}}g\|_{L^2(\gamma)}\ge \rho^{d/2}\|g\|_{L^2(\gamma)}.\]
\end{proof}



 \section{Low-degree truncations of $\Tcol$: Proof of Lemma~\ref{lem:main_lemma}.}\label{sec:ingredients}
 In this section, we prove 
 Lemma~\ref{lem:main_lemma}. Our main idea is to show that one can approximate the $L^2$-norms of (row) comfortable $d$-juntas with respect to $(\sqrt{n}\SO(n),\mu)$ by their $L^2$-norms with respect to $(\mathbb{R}^{n\times n}, \gamma)$.  

\subsection{Comparing $L^2(\mu)$ and $L^2(\gamma)$.}
The following lemmas assert that the $2$-norm in $L^2(\sqrt{n} \SO(n))$ of a row comfortable $d$-junta is roughly bounded by its $2$-norm in Gaussian space. We defer the proofs of the lemmas to Sections~\ref{sec:closeness1} and \ref{sec:closeness2}. 
\begin{lemma}\label{lem:one sided bound}
  For all $\eps>0$, there exists $\delta>0$ such that if $d\leq \delta n^{1/2}$, then for any function $f:\mathbb{R}^{n \times n} \to \mathbb{R}$ that is either (i) a column comfortable, $d$-column-junta or (ii) a row comfortable, $d$-row-junta, we have \[\|f\|_{L^{2}\left(\mu\right)}\le (1+\eps)\|f\|_{L^{2}\left(\gamma\right)}.\]
\end{lemma}

In the case where the function $f$ is comfortable, we show that in fact the same inequality with the measures swapped.
\begin{lemma}\label{lem:the other side}
  For all $\eps>0$, there exists $\delta>0$ such that if $d\le \delta n^{1/2}$ and $f:\mathbb{R}^{n \times n} \to \mathbb{R}$ is a comfortable $d$-junta, then
  \[
  \|f\|_{L^2(\gamma)}\le (1+\eps)\|f\|_{L^2(\mu)}.
  \]
\end{lemma}

\subsection{The main argument for proving Lemma~\ref{lem:main_lemma}.}
We now give the proof of 
Lemma~\ref{lem:main_lemma}, assuming Lemmas \ref{lem:one sided bound} and \ref{lem:the other side}.
\subsubsection*{Swapping $\Trow$ and $\Tcol$.}
Our first step is to show that on the left-hand side of Lemma~\ref{lem:main_lemma}, we can 
replace $\Tcol$ by $\Trow$. The benefit of this 
exchange is that $\Trow$ and $R_V$ commute.  

\begin{lemma}\label{lem:swap}
There exist absolute constants $C>0,\delta >0$ such that the following holds. Let  $d< \delta n$ and let $f:\mathbb{R}^{n \times n} \to \mathbb{R}$ be a comfortable $d$-junta.
Then 
\[
\Expect{V\sim \SO(n)}{\|(\Tcol R_V f)^{\le d}  \|_{L^2(\gamma)}^2 } \ge C^{-d} \Expect{V\sim \SO(n)}{ \|\Pi_{\sf comf, d}\Trow R_V f  \|_{L^2(\gamma)}^2}.
\]
\end{lemma}
\begin{proof}
\remove{By using the adjointness and} Applying Claim \ref{claim:comf_ev}, we have 
\begin{align*}
\|(\Tcol R_V f)^{\le d}  \|_{L^2(\gamma)}^2 
&\ge \sum_{H_S \text{ comfortable of degree } d} \langle \Tcol R_Vf, H_S \rangle_{L^2(\gamma)}^2 \\
&= \sum_{H_S \text{ comfortable of degree } d} \langle R_Vf, \Tcol^{*} H_S \rangle_{L^2(\mu)}^2\\
&=\lambda_S^2\sum_{H_S \text{ comfortable of degree } d} \langle R_Vf, H_S \rangle_{L^2(\mu)}^2\\
&=\frac{\lambda_S^2}{\lambda_{S^{t}}^2}
\sum_{H_S \text{ comfortable of degree } d} \langle R_Vf, \Trow^{*} H_S \rangle_{L^2(\mu)}^2\\
&\ge C^{-d} \sum_{H_S \text{ comfortable of degree } d} \langle \Trow R_Vf, H_S \rangle_{L^2(\gamma)}^2. 
\end{align*}
The lemma follows by plugging in the definition of $\Pi_{{\sf comf}, d}$.
\end{proof}
\subsubsection*{$\Trow f$ is close to $f$.}
We have now reduced our task to understanding 
the average of the square of the $2$-norm of $\Pi_{\sf comf, d}\Trow R_V f = \Pi_{\sf comf, d}R_V \Trow f$ (this equality holds because $\Trow$ and 
$R_V$ commute). The following claim further simplifies
our task and shows that $\Trow f$ is close to $f$, 
thereby effectively reducing our task to estimating the 
$2$-norm of $\Pi_{\sf comf, d} R_V f$. (Though some care is required to make this precise, as 
we are applying a projection operator on top, 
which may decrease norms considerably.)
\begin{claim}\label{claim:ccomf_ev2}
For any $\epsilon>0$, there exists $\delta>0$ such that the following holds. Let $d\leq\delta n^{1/2}$, and let $f:\mathbb{R}^{n \times n} \to \mathbb{R}$ be a comfortable $d$-junta. Then 
  \[
  \norm{\Trow f - f}_{L^2(\gamma)}^2
  \leq   \epsilon \norm{f}_{L^2(\gamma )}^2.
  \]
\end{claim}
\begin{proof}
Let $g$ be a comfortable $d$-junta  satisfying $\Trow^*g=f$, i.e.\ writing $f = \sum\limits_{S} \alpha_S H_S$, we take $g=\sum \lambda_{S^t}^{-1}\alpha_S H_S$, where $\lambda_S$ is as in Claim \ref{claim:comf_ev}. Then by Parseval, we have
\[\|g-f\|_{L^2(\gamma)} \le \eps \|f\|_{L^2(\gamma)},\]
provided $\delta$ is sufficiently small. Hence, using Cauchy--Schwarz, we have
\begin{align*}
\|f\|_{L^2(\mu)}^2 = \langle f, \trow ^* g\rangle 
&= \langle \Trow f,g \rangle = \langle \Trow f, f \rangle + \langle \Trow f, g-f \rangle \\
&\le \langle \Trow f, f \rangle + \| \Trow f \|_{L^2(\gamma)} \| g-f\|_{L^2(\gamma)},
\end{align*}
implying that
\[
\langle \Trow f,f \rangle \ge \|f\|_{L^2(\mu)}^2-\eps \|\trow f\|_{L^2(\gamma)}\|f\|_{L^2(\gamma)} =  \|f\|_{L^2(\mu)}^2-\eps \|f\|_{L^2(\mu)}\|f\|_{L^2(\gamma)}.
\]
Thus, we obtain
\begin{align*}
 \|\Trow f-f \|_{L^2(\gamma)}^2 &=  \|f\|_{L^2(\gamma)}^2 + \|\trow f\|_{L^2(\gamma)}^2 - 2 \langle \Trow f,f \rangle 
 \\ &\le  \|f\|_{L^2(\gamma)}^2 -\|f\|_{L^2(\mu)}^2 + 2\eps \|f\|_{L^2(\mu)}\|f\|_{L^2(\gamma)}\\
 & \leq 3\epsilon \|f\|_{L^2(\gamma)}^2,
\end{align*}
using Lemmas \ref{lem:one sided bound} and \ref{lem:the other side} (again provided $\delta$ is sufficiently small depending on $\epsilon$). This completes the proof.
 \end{proof}

\subsubsection*{The projection of $R_V\Trow f$ onto 
the subspace of comfortable polynomials}
At this point, it would seem 
that to finish the proof of Lemma~\ref{lem:main_lemma} it suffices to 
estimate the typical $2$-norm of $\Pi_{{\sf comf}, d} R_V f$. While this is indeed the case, 
some care is needed, as one cannot 
switch particularly smoothly from $\Trow f$ to $f$ in the previous statement. 
To address this, we must be able to estimate 
the $2$-norm of $\Pi_{{\sf comf}, d} R_V f$ under
the weaker hypothesis that $f$ is a 
 {\em row-}comfortable $d$-junta. This is the content of the following claim.
\begin{claim} \label{claim: L_U works properly}
  Let $d\le n/2$ and let $f:\mathbb{R}^{n \times n} \to \mathbb{R}$ be a row-comfortable $d$-row-junta. Then $$\mathbb{E}_{V\sim \SO(n)} \|\Pi_{{\sf comf}, d}R_V f\|_{L^2(\gamma)}^2\geq \frac{(n/2)!}{2^d n^d((n/2)-d)!}\|f\|_{L^2(\mu)}^2.$$
  \end{claim}
\begin{proof}
We first note that if $H_{\alpha}$ and $H_{\beta}$ are row-comfortable monomials such that the set of rows appearing in $\alpha$ is different from the set of rows appearing in $\beta$, then we have $\langle H_{\alpha},H_{\beta} \rangle_{L^2(\gamma)} = \langle H_{\alpha},H_{\beta} \rangle_{L^2(\mu)}=0$. Indeed, the first equality is immediate from the pairwise orthogonality of the Hermite polynomials with respect to $L^2(\gamma)$, and the second equality follows from the fact that, if $H_{\alpha}$ contains a variable from the $i$th row and $H_{\beta}$ does not, then flipping the sign of both the $i$th row and the $(\lfloor n/2\rfloor+1)$th row of a matrix $X \in \SO(n)$ yields a (Haar-)measure-preserving map on $\SO(n)$ that sends $H_{\alpha}(X)H_{\beta}(X)$ to its negative, and therefore $\langle H_{\alpha},H_{\beta} \rangle_{L^2(\mu)}=\mathbb{E}_{X \sim \mu} [H_{\alpha}(X)H_{\beta}(X)]=0$.

For $S \subset [n]$, let $W_{S}$ denote the linear span of the row-comfortable homogeneous monomials of degree $|S|$ which depend only upon variables from the rows in $S$ (and therefore depend upon precisely one variable from each of the rows in $S$). It is clear that the $W_S$ are pairwise orthogonal with respect to $L^2(\gamma)$, and by the above observation they are also pairwise orthogonal with respect to $L^2(\mu)$. 

By the above remarks, for any row-comfortable $d$-row junta $g:\mathbb{R}^{n \times n} \to \mathbb{R}$, we may write $g$ as an orthogonal direct sum,
$$g = \sum_{S \subset [d]}g_{(=S)},$$
where $g_{(=S)}$ denotes the orthogonal projection of $g$ onto $W_S$ (with respect to $L^2(\gamma)$). It is clear that $(R_Vg)_{(=S)} = R_V(g_{(=S)})$ for any $V \in \SO(n)$ and any row-comfortable $d$-row-junta $g$.

Now let $f:\mathbb{R}^{n \times n} \to \mathbb{R}$ be a row-comfortable $d$-row-junta. Since the $f_{(=S)}$ are pairwise orthogonal with respect to $L^2(\mu)$ as well as with respect to $L^2(\gamma)$, we have
$$\|f\|_{L^2(\mu)}^2 = \sum_{S \subset [d]}\|f_{(=S)}\|_{L^2(\mu)}^2,$$
and therefore by averaging, there exists $S \subset [d]$ such that 
\begin{equation}\label{eq:avg} \|f_{(=S)}\|_{L^2(\mu)}^2 \geq \frac{1}{2^d}\|f\|_{L^2(\mu)}^2.\end{equation}

For brevity, write $h = f_{(=S)}$, $d':=|S|$ and $S = \{i_1,\ldots,i_{d'}\}$. Our next aim is to show that
\begin{equation}\label{eq:h-eqn} \mathbb{E}_{V\sim \SO(n)} \|\Pi_{{\sf comf}, =S}R_V h\|_{L^2(\gamma)}^2\geq \frac{(n/2)!}{n^{d'}((n/2)-d')!}\|h\|_{L^2(\mu)}^2.
\end{equation}
In proving (\ref{eq:h-eqn}), we may and shall assume (without loss of generality) that $S = \{1,2,\ldots,d'\}$. We first assert that, in this case, 
$$\mathbb{E}_{V\sim \SO(n)} \|\Pi_{{\sf comf},=S}R_V h\|_{L^2(\gamma)}^2= \frac{(n/2)!}{((n/2)-d')!}\mathbb{E}_V\langle R_V h,H_{\{(1,1),\ldots,(d',d')\}}\rangle_{L^2(\gamma)}^2.$$
Indeed, this follows from the fact that we may write $V$ as the product of a random $\SO(n)$ matrix $V'$ and a random permutation matrix $P_{\sigma}$, $V = P_{\sigma}V'$ say. Now the $S_n$-orbit $\{R_{P_{\sigma}}m:\ \sigma \in S_n\}$ of a monomial $m=m(X)$ of the form $\prod_{i=1}^{d'} X_{i, v_i}$ with the $v_i$ all distinct, consists of all the monomials of the form $\prod_{i=1}^{d'} X_{i, w_i}$, with the $w_i$ all distinct.

Now write $h=\sum_{\alpha=((1,i_1),\ldots ,(d', i_{d'}))}  \hat{h}(\alpha) H_\alpha$. Then for each $\alpha$ of the form $\{(1,j_1),\ldots ,(d',j_{d'})\}$ (for $j_1,\ldots,j_{d'}$ all distinct), we have
$$R_V H_\alpha(X) = H_\alpha (XV) = \prod_{k=1}^{d} [XV]_{k,j_k}=\sum_{r_1,\ldots, r_{d'}} \prod_{k=1}^{d'}X_{k,r_k}V_{r_k,j_k}.$$
Therefore,
$$\langle R_V H_\alpha(X), H_{\{(1,1),\ldots,(d',d')\}}\rangle_{L^2(\gamma)} =  H_\alpha(V)=n^{-d'/2}H_\alpha(\sqrt{n} V).$$
Expanding $h=\sum_\alpha \hat{h}(\alpha)H_\alpha$ and taking 2-norms with respect to $\sqrt{n}V\sim \mu$, we see that indeed,
$$\mathbb{E}_{V\sim \SO(n)} \|\Pi_{{\sf comf}, =S}R_V h\|_{L^2(\gamma)}^2\geq \frac{(n/2)!}{n^{d'}((n/2)-d')!}\|h\|_{L^2(\mu)}^2,$$
proving (\ref{eq:h-eqn}).

Combining this with (\ref{eq:avg}) and using the fact that $d' \leq d$, yields
$$\mathbb{E}_{V\sim \SO(n)} \|\Pi_{{\sf comf}, =S}R_V f_{(=S)}\|_{L^2(\gamma)}^2 \geq \frac{(n/2)!}{2^d n^d((n/2)-d)!}\|f\|_{L^2(\mu)}^2.$$

Finally, since 
\begin{align*}\Pi_{{\sf comf},d}R_Vf &= \Pi_{{\sf comf},d}R_V\left(\sum_{S \subset [d]}f_{(=S)}\right) \\
&= \Pi_{{\sf comf},d} \left(\sum_{S \subset [d]}(R_Vf)_{(=S)}\right)\\
&= \sum_{S \subset [d]} \Pi_{{\sf comf},=S}\left((R_Vf)_{(=S)}\right),
\end{align*}
and since the last sum is an orthogonal direct sum (again using the pairwise orthogonality of the $W_S$'s with respect to $L^2(\gamma)$), we have
\begin{align*}\mathbb{E}_{V\sim \SO(n)} \|\Pi_{{\sf comf}, d}R_V f\|_{L^2(\gamma)}^2
& \geq \mathbb{E}_{V\sim \SO(n)} \|\Pi_{{\sf comf}, =S}\left((R_V f)_{(=S)}\right)\|_{L^2(\gamma)}^2\\
&= \mathbb{E}_{V\sim \SO(n)} \|\Pi_{{\sf comf}, =S}R_V (f_{(=S)})\|_{L^2(\gamma)}^2 \\
& \geq \frac{(n/2)!}{2^d n^d((n/2)-d)!}\|f\|_{L^2(\mu)}^2,\end{align*}
proving the claim.
\end{proof}

We now show how to lower-bound the right-hand side of the expression in the statement of Lemma~\ref{lem:swap}.
\begin{lemma}\label{lem:randomize__comf}
  There exist $\delta>0$ and $C>0$ such that the following holds. Let $d\le \delta n^{1/2}$. Then for all comfortable $d$-juntas $f:\mathbb{R}^{n \times n} \to \mathbb{R}$, we have 
  \[
  \Expect{V\sim \SO(n)} 
  {{\norm{\Pi_{{\sf comf}, d}  R_V \Trow f}}_{L^2(\gamma)}^2}
  \ge C^{-d} \norm{f}_{L^2(\mu)}^2.
  \]
\end{lemma}
\begin{proof}
  We write $\Pi_{{\sf comf}, d} = \Pi_{{\sf comf ,}d} \Pi_{{\sf comf,row ,}d}$, and set $g=\Pi_{{\sf comf, row}, d} \trow f$. Note that $g$ is a row-comfortable $d$-row-junta. We have
  \[
  \Pi_{{\sf comf}, d}  R_V \Trow f = 
  \Pi_{{\sf comf}, d}  R_V g,
  \]
  where we used the fact that $\Pi_{{\sf comf}, {\sf row}, d}$ commutes with $R_V$. Taking the squares of the 2-norms and expectations over $V$ we may apply Claim~\ref{claim: L_U works properly} to $g$, obtaining
  \begin{equation}\label{eq:randomize_comf}
  \Expect{V}{\|\Pi_{{\sf comf}, d}\trow  R_V f\|_{L^2(\gamma)}^2} = \frac{(n/2)!}{n^d ((n/2)-d)!}\|g\|_{L^2(\mu)}^2
  \geq C^{-d}\|g\|_{L^2(\mu)}^2.
  \end{equation}
   By Lemma \ref{lem:one sided bound}, as $g-f$ is a row comfortable, $d$-row-junta, we have
   \[ 
   \|g\|_{L^2(\mu)} \ge \|f\|_{L^2(\mu)}-\|g-f\|_{L^2(\mu)} \ge \|f\|_{L^2(\mu)}-2\|g-f\|_{L^2(\gamma)}
   \]
   As $f$ is comfortable, we can apply Claim \ref{claim:ccomf_ev2} to it to obtain 
   \[
   \|g-f\|_{L^2(\gamma)}
   =\|\Pi_{{\sf comf, row}, d}
   (\Trow f -f)\|_{L^2(\gamma)}
   \le \| \Trow f - f \|_{L^2(\gamma)}\le 
   \epsilon \|f\|_{L^2(\gamma)}.
   \]
   On the other hand Lemma \ref{lem:the other side} shows that $\|f\|_{L^2(\gamma)}\le 2\|f\|_{L^2(\mu)}$. Putting these three facts together, we obtain $\|g\|_{L^2(\mu)}> ( 1-2\epsilon) \|f\|_{L^2(\mu)}$. Plugging this into~\eqref{eq:randomize_comf}, and adjusting the value of $C$, completes the proof.
  \end{proof}
\subsubsection*{Finishing the proof of Lemma~\ref{lem:main_lemma}.}
 Using Lemma~\ref{lem:swap}, the left-hand side of Lemma~\ref{lem:main_lemma} is at least 
 \[
 C^{-d} \Expect{V\sim \SO(n)}{ \|\Pi_{\sf comf, d}\Trow R_V f  \|_{L^2(\gamma)}^2}
 =C^{-d} \Expect{V\sim \SO(n)}{ \|\Pi_{\sf comf, d}
 R_V \Trow  f  \|_{L^2(\gamma)}^2},
 \]
 and using Lemma~\ref{lem:randomize__comf}
 the last quantity is at least $C'^{-d}\norm{f}_{L^2(\mu)}^2$, as required.\hfill\qedsymbol

\subsection{$L^2(\mu)$ is dominated by $L^2(\gamma)$ on column or row comfortable juntas: Proof of Lemma~\ref{lem:one sided bound}.}\label{sec:closeness1}
Our remaining tasks are to prove Lemmas \ref{lem:one sided bound} and \ref{lem:the other side}. First, we show that the operator $\Tcol^{*}$ is close to the identity on column comfortable $d$-column-juntas.

\begin{claim} \label{claim: T_col star is close to the identity}
  For all $\eps>0$ there exists $\delta>0$ such that if $d\leq \delta n^{1/2}$ and $f$ is a column comfortable $d$-column-junta, then
   \[
   \|\Tcol^{*}f-f\|_{L^{2}\left(\gamma\right)}\le \eps\|f\|_{L^{2}\left(\gamma\right)}.
   \]
\end{claim}
\begin{proof}
The lemma follows immediately from Parseval and Claim \ref{claim:comf_ev} as $|\lambda_S - 1|< \epsilon$ for each $S$ as above, provided $\delta$ is sufficiently small (similarly to in the proof of Claim~\ref{claim:ccomf_ev2}).
\end{proof}
Claim~\ref{claim: T_col star is close to the identity} is particularly useful  as it
implies that $\Tcol^{*}$ is invertible on 
the space spanned by column-comfortable $d$-column-juntas, and thus gives us a natural way of going from $L^2(\mu)$ to $L^2(\gamma)$.
We are now ready to prove Lemma~\ref{lem:one sided bound}, restated below.
\begin{replemma}{lem:one sided bound}
For all $\eps>0$, there exists $\delta>0$ such that if $d\leq \delta n^{1/2}$, then for any function $f:\mathbb{R}^{n \times n} \to \mathbb{R}$ that is either (i) a column comfortable, $d$-column-junta or (ii) a row comfortable, $d$-row-junta, we have \[\|f\|_{L^{2}\left(\mu\right)}\le (1+\eps)\|f\|_{L^{2}\left(\gamma\right)}.\]
\end{replemma}
\begin{proof}
Let $\delta = \delta(\epsilon)>0$ be as in the statement of Claim \ref{claim: T_col star is close to the identity}. Without loss of generality, we assume that $f$ is a column comfortable $d$-column-junta, i.e., is contained in the subspace
$$W:=\text{Span}\{H_S:\ H_S \text{ is column comfortable and }S\text{ is supported on }[n] \times [d]\},$$
otherwise we may 
consider $f'(X) = f(X^t)$. Let $g:\mathbb{R}^{n \times n} \to \mathbb{R}$ be a function such that  $\Tcol^{*}g=f.$ Then as $\Tcol^*$ is a contraction, we have $\|f\|_{L^{2}\left(
\mu\right)}\le\|g\|_{L^{2}\left(\gamma\right)}.$
In order to complete the proof, we note (using Claim~\ref{claim: T_col star is close to the identity}) that $\Tcol^*(W)=W$, $\Tcol^*$ is invertible on 
the subspace $W$, and that
\begin{equation}\label{eq:norm-ineq} \|(\Tcol^*|_W)^{-1}\|_{L^2(\gamma)\to L^2(\gamma)} \leq 1+2\eps.\end{equation}
Indeed, the facts that $\Tcol^*(W)=W$ and that $\Tcol^*$ is invertible on 
the subspace $W$, follow immediately from Claim~\ref{claim:comf_ev}. Moreover, if (\ref{eq:norm-ineq}) fails, then there exists $h \in W$ with $\|h\|_{L^2(\gamma)}<1/(1+2\epsilon)$ and $\norm{(\Tcol^{*}|_W)^{-1} h}_{L^2(\gamma)}=1$, and then for $h' := (\Tcol^{*}|_W)^{-1} h$ (which lies in $W$), we obtain
\begin{align*}
\norm{\Tcol^{*} h' - h'}_{L^2(\gamma)}& =\norm{h - (\Tcol^{*}|_W)^{-1} h}_{L^2(\gamma)}\\
&
\geq \norm{(\Tcol^{*}|_W)^{-1} h}_{L^2(\gamma)} - \norm{h}_{L^2(\gamma)}\\
& > 1-1/(1+2\epsilon)\\
&\geq \eps\\
& = \eps \| h'\|_{L^2(\gamma)}
\end{align*}
provided $\epsilon\leq 1/2$, contradicting Claim~\ref{claim: T_col star is close to the identity}.

We have $g=(\Tcol^{*}|_{W})^{-1} f$, 
and so
$$\|f\|_{L^2(\mu)} \leq \norm{g}_{L^2(\gamma)} = \norm{(\Tcol^{*}|_W)^{-1} f}_{L^2(\gamma)}\leq (1+2\eps)\norm{f}_{L^2(\gamma)},$$
completing the proof.
\end{proof}

\subsection{$L^2(\gamma)$ is dominated by $L^2(\mu)$ on comfortable juntas: Proof of Lemma \ref{lem:the other side}.}\label{sec:closeness2}
To prove Lemma \ref{lem:the other side}, we define 
an auxiliary distribution on $\mathbb{R}^{n\times n}$, which we refer to as the `over-Gaussian' distribution.
\begin{definition}
  Let $G\sim \text{GMD}$, and choose $Y\sim \gamma$ independently.
  We define the distribution $\nu$ to be the distribution of $YG$, and call it the {\em over-Gaussian distribution}.
\end{definition}
We refer to $\nu$ by this name since it can be produced taking $X\sim \mu$ (i.e., at random according to the Haar measure on $\SO(n)$), and then multiplying $X$ by two independent copies of $\text{GMD}$, thereby `overshooting' the Gaussian distribution.

In the following section, we show that the distribution
$\nu$ is close to $\gamma$ in the sense that the expectation of a certain kind of test function is roughly the same under both measures. In fact, the relevant test functions are the squares of the comfortable $d$-juntas, and their expectations are roughly the same even if we allow the degree $d$ to be as large as $\Theta(\sqrt{n})$. The following lemma asserts that if $f$ is a comfortable $d$-junta, then its over-Gaussian $2$-norm
cannot be much larger than its Gaussian $2$-norm.
\begin{lemma}\label{lem:contiguity}
  For all $\eps>0$ there exists $\delta>0$ such that if $d \leq \delta n^{1/2}$, then for all comfortable $d$-juntas $f:\mathbb{R}^{n \times n} \to \mathbb{R}$, we have
  \[
  \|f\|_{L^2(\nu)} \leq (1+\eps)\|f\|_{L^2(\gamma)}.
  \]
\end{lemma}

To prove Lemma~\ref{lem:contiguity}, we need some more notation, and two technical claims. For a permutation $\sigma\in S_d$, we write 
 $x_\sigma: = x_{1,\sigma (1)}\cdots x_{d,\sigma(d)}$; this is a function on $\mathbb{R}^{d\times d}$. 
Fix a comfortable $d$-junta $f:\mathbb{R}^{d \times d} \to \mathbb{R}$ and write $f=\sum_{I=\left(i_{1},\ldots,i_{d}\right)}a_{I}x_{I}$ 
where $x_{I}:=x_{1,i_{1}}x_{2,i_{2}}\cdots x_{d,i_{d}}$, the sum ranging over all $I$ such that $i_1,\ldots, i_d\in [d]$ are distinct. (As usual, we write $(S)_d$ for the set of ordered $d$-tuples of distinct elements of the set $S$; using this notation we may write $I \in ([d])_d$.) We have
$$\|f\|_{L^{2}\left(\nu\right)}^{2}  =\sum_{I}a_{I}^{2}\|x_{I}\|_{L^{2}\left(\nu\right)}^{2}+\sum_{I\ne J}a_{I}a_{J}\left\langle x_{I},x_{J}\right\rangle _{\nu}.$$
We now need two technical claims, which handle respectively the diagonal terms
and the off-diagonal terms in the above expansion of $\norm{f}_{L^2(\nu)}^2$. To handle the diagonal terms, we will use the following claim.
\begin{claim}\label{claim:diagonal_terms_compute}
  If $i_1,\ldots,i_d \in [n]$ are distinct, then $x_I := x_{1,i_{1}}x_{2,i_{2}}\cdots x_{d,i_{d}}$ satisfies $$\|x_{I}\|_{L^{2}\left(\nu\right)}^{2} = 1.$$
\end{claim}
For $I,J \in [n]^d$ we let 
$d(I,J):=\card{\sett{r}{i_r\neq j_r}}$ denote the Hamming distance from $I$ to $J$. To handle the off-diagonal terms, we need the following claim. 
\begin{claim}\label{claim:off_diagonal_terms_compute}
For any $I,J$ such that $d(I,J) = \ell$, we have $\card{\left\langle x_{I},x_{J}\right\rangle} _{L^{2}\left(\nu\right)}\le \varepsilon_{\ell}$, where
$$\varepsilon_{\ell} := 2^{\ell+4}n^{-\ell/2} 2^{d\ell/\sqrt{n}}.$$
\end{claim}

\subsubsection*{Claims~\ref{claim:diagonal_terms_compute} and~\ref{claim:off_diagonal_terms_compute} imply 
Lemma~\ref{lem:contiguity}}
We first show how to deduce Lemma~\ref{lem:contiguity} from Claims~\ref{claim:diagonal_terms_compute} and~\ref{claim:off_diagonal_terms_compute}. Writing $f = \sum_I a_I x_I$, we have
\begin{align*}
\|f\|_{L^{2}\left(\nu\right)}^{2} & =\sum_{I}a_{I}^{2}\|x_{I}\|_{L^{2}\left(\nu\right)}^{2}+\sum_{I\ne J}a_{I}a_{J}\left\langle x_{I},x_{J}\right\rangle _{\nu}\\
 & \le\|f\|_{L^{2}\left(\gamma\right)}^{2}+\sum_{I\ne J}\frac{a_{I}^{2}+a_{J}^{2}}{2}\card{\left\langle x_{I},x_{J}\right\rangle _{\nu}}\\
 & \le\|f\|_{L^{2}\left(\gamma\right)}^{2}+
 \sum\limits_{\ell=1}^{d}\sum_{I}a_{I}^{2}\card{\left\{ J:d\left(J,I\right)=\ell\right\}} \cdot\epsilon_{\ell}\\
 & \le\|f\|_{L^{2}\left(\gamma\right)}^{2}+\|f\|_{L^{2}\left(\gamma\right)}^{2}\sum_{\ell=1}^{d}\epsilon_{\ell}\binom{d}{\ell} \ell !\\
 & =\|f\|_{L^{2}\left(\gamma\right)}^{2}\left(1+\sum_{\ell=1}^{d}\epsilon_{\ell} d^{\ell}\right).
\end{align*}
Using the upper bound on $\eps_{\ell}$, we obtain
\[
\sum_{\ell=1}^{d}\epsilon_{\ell} d^{\ell}
\leqslant
\sum_{\ell=1}^{d}2^{\ell+4}n^{-\ell/2} 2^{d\ell/\sqrt{n}} d^{\ell}
\leqslant
16\sum_{\ell=1}^{\infty}\left(\frac{2d\cdot 2^{d/\sqrt{n}}}{\sqrt{n}} \right)^{\ell} \leq \frac{\eps}{2},
\]
where we used $d\leq \delta n^{1/2}$ and the fact that $\delta$ is sufficiently small compared to $\eps$.

\subsubsection*{Proof of Claims~\ref{claim:diagonal_terms_compute} and \ref{claim:off_diagonal_terms_compute}.}
To prove the two claims, we need the following simple fact about the Gaussian maker distribution.
\begin{claim}
\label{claim:gmd lemma} Let $I=(i_{1},\ldots,i_{d}) \in [n]^d$ and $J=(j_1,\ldots,j_d) \in [n]^d$ be such that $d(I,J) = \ell$, and such that in the product
$$G_{i_{1}j_{1}}G_{i_{2}j_{2}}\cdots G_{i_{d}j_{d}},$$
no matrix entry of $G$ appears more than twice. Then
\[
\card{\mathbb{E}_{G\sim\text{GMD}}\left[G_{i_{1}j_{1}}G_{i_{2}j_{2}}\cdots G_{j_{d}i_{d}}\right]}\le\left(\frac{1}{n}\right)^{\ell/2}.
\]
\end{claim}

\begin{proof}
If, in the product
$$G_{i_{1}j_{1}}G_{i_{2}j_{2}}\cdots G_{i_{d}j_{d}},$$
some off-diagonal matrix entry of $G$ appears exactly once, then the expectation of the product is zero. We may therefore assume that every off-diagonal matrix entry of $G$ appears either exactly twice, or not at all, in the above product. If there are exactly $\ell$ values of $r$ such that $i_r \neq j_r$, then the above expectation factorises into a product of the expectations of the squares of $\ell/2$ off-diagonal and of the squares of $(d-\ell)/2$ diagonal entries: 
$$\prod_{k \in \mathcal{D}}\mathbb{E}[G_{k,k}^2]\prod_{(i,j) \in \mathcal{E}}\mathbb{E}[G_{i,j}^2],$$
where $\mathcal{E} \subset [n]^2\setminus \{(k,k):\ k \in [n]\}$, $|\mathcal{D}| = (d-\ell)/2$ and $|\mathcal{E}| = \ell/2$.
We have $\mathbb{E}[G_{i,j}^2] = 1/n$ for all $(i,j) \in \mathcal{E}$ and $\mathbb{E}[G_{k,k}^2] = (n-k+1)/n \leq 1$ for all $k \in \mathcal{D}$, proving the claim.
\end{proof}

We are now ready to prove Claim~\ref{claim:diagonal_terms_compute}.
\begin{proof}[Proof of Claim~\ref{claim:diagonal_terms_compute}]

Let $x_I = x_{1,i_1}x_{2,i_2}\cdots x_{d,i_d}$, where $i_1,\ldots,i_d \in [n]$ are distinct. We have
$$\|x_{I}\|_{L^2(\nu)}^2 = \mathbb{E}_{G \sim \text{GMD}}\mathbb{E}_{Y \sim \gamma}[(YG)_{1,i_1}^2(YG)_{2,i_2}^2\cdots (YG)_{d,i_d}^2].$$
Since for each $h \in [d]$, $(YG)_{h,i_h} = \sum_{k=1}^{i_h} Y_{h,k}G_{k,i_h}$ involves only entries of $Y$ in row $h$ and entries of $G$ in column $i_h$ (and the $i_h$ are distinct), the random variables $((YG)_{h,i_h}^2:\ h \in [d])$ form a system of independent random variables, and therefore
$$\|x_{I}\|_{L^2(\nu)}^2 = \mathbb{E}_{G \sim \text{GMD}}\mathbb{E}_{Y \sim \gamma}[(YG)_{1,i_1}^2] \mathbb{E}_{G \sim \text{GMD}}\mathbb{E}_{Y \sim \gamma}[(YG)_{2,i_2}^2]\cdots \mathbb{E}_{G \sim \text{GMD}}\mathbb{E}_{Y \sim \gamma}[(YG)_{d,i_d}^2].$$
For each $h \in [d]$, we have
\begin{align*} \mathbb{E}_{G \sim \text{GMD}}\mathbb{E}_{Y \sim \gamma}[(YG)_{h,i_h}^2] & = \mathbb{E}_{G \sim \text{GMD}}\mathbb{E}_{Y \sim \gamma}\left[\left(\sum_{k=1}^{i_h} Y_{h,k}G_{k,i_h}\right)^2\right]\\
& = 2\sum_{1\leq k < k' \leq i_h} \mathbb{E}_{G \sim \text{GMD}}\mathbb{E}_{Y \sim \gamma}[Y_{h,k}Y_{h,k'}G_{k,i_h}G_{k',i_h}]\\
&+ \sum_{k=1}^{i_h} \mathbb{E}_{G \sim \text{GMD}}\mathbb{E}_{Y \sim \gamma} [Y_{h,k}^2G_{k,i_h}^2]\\
& = 0 + \sum_{k=1}^{i_h} \mathbb{E}_{G \sim \text{GMD}}[G_{k,i_h}^2]\mathbb{E}_{Y \sim \gamma} [Y_{h,k}^2]\\
& = \sum_{k=1}^{i_h} \mathbb{E}_{G \sim \text{GMD}}[G_{k,i_h}^2]\cdot 1\\
& = (i_h-1)(1/n)+(n-i_h+1)/n\\
& = 1.
\end{align*}
(Here, for the third equality we use the independence of $Y_{h,k},Y_{h,k'},G_{k,i_h},G_{k',i_h}$ and the fact that $Y_{h,k}$ and $Y_{h,k'}$ both have zero expectation.) Hence, $\|x_{I}\|_{L^2(\nu)}^2 = 1$, as required.
\end{proof}

We now move on to the proof of Claim~\ref{claim:off_diagonal_terms_compute}.
\begin{proof}[Proof of Claim~\ref{claim:off_diagonal_terms_compute}.]
Let $\ell\geq 1$ and fix $I,J \in ([d])_d$ such that $d(I,J)=\ell \geq 1$. Since $G$ is upper-triangular and $i_h,j_h \leq d$ for all $h \in [d]$, we have
$$(YG)_{h,i_h} = \sum_{k=1}^{i_h}Y_{h,k}G_{k,i_h} = \sum_{k=1}^{d} Y_{h,k}G_{k,i_h}$$
and
$$(YG)_{h,j_h} = \sum_{k=1}^{j_h}Y_{h,k}G_{k,j_h} = \sum_{k=1}^{d} Y_{h,k}G_{k,j_h}$$
for all $h \in [d]$. Hence,
\begin{align*}
x_{I}\left(YG\right) & =\prod_{h=1}^{d}(YG)_{h,i_h}= \sum_{K=\left(k_{1},\ldots,k_{d}\right) \in [d]^d}Y_{1,k_{1}}\cdots Y_{d,k_{d}}G_{k_{1},i_{1}}\cdots G_{k_{d},i_{d}}
\end{align*}
 and  
\[
x_{J}\left(YG\right)=\sum_{K=\left(k_{1},\ldots,k_{d}\right) \in [d]^d}Y_{1,k_{1}}\cdots Y_{d,k_{d}}G_{k_{1},j_{1}}\cdots G_{k_{d},j_{d}},
\]
 so, using the fact that, under $\nu$, the $(Y_{i,j}:i,j \in [n])$ are independent and of expectation zero (and are independent of the $G_{i,j}$), we obtain 
\[
\left\langle x_{I},x_{J}\right\rangle _{\nu}=\sum_{K \in [d]^d}\mathbb{E}_{G \sim \text{GMD}}\left[G_{k_{1}i_{1}}G_{k_{1}j_{1}}\cdots G_{k_{d}i_{d}}G_{k_{d}j_{d}}\right].
\]
For a $d$-tuple $K = (k_1,\ldots,k_d)\in [d]^d$ we write $m_{1}=m_1(K) := |\{ r:j_{r}=i_{r},k_{r}\ne i_{r}\}| ,$
$m_{2}=m_2(K) := |\{ r:j_{r}\ne i_{r},k_{r}\notin\{ i_{r},j_{r}\}\}|$
and $m_{3}=m_3(K) :=|\{ r:j_{r}\ne i_{r},k_{r}\in\{ i_{r},j_{r}\}\}|$, and we let $\mathcal{K}\left(m_{1},m_{2},m_{3}\right)$ denote the set of $d$-tuples $K$ with parameters $m_1,m_2$ and $m_3$.
For $K\in\mathcal{K}\left(K_{1},K_{2},K_{3}\right)$, by
Claim~\ref{claim:gmd lemma} we have
\[
|\mathbb{E}_{G \sim \text{GMD}}\left[G_{k_{1}i_{1}}G_{k_{1}j_{1}}\cdots G_{k_{d}i_{d}}G_{k_{d}j_{d}}\right]|\le n^{-\frac{2m_{1}+2m_{2}+m_{3}}{2}}.
\]
Summing over all $K$, we see that $|\langle x_{I},x_{J}\rangle|$ is at most
\begin{align*}
&\sum_{m_{1},m_{2},m_{3}}\sum_{K\in\mathcal{K}\left(m_{1},m_{2},m_{3}\right)}n^{-\frac{2m_{1}+2m_{2}+m_{3}}{2}}\\
 & \le\sum_{m_{1},m_{2},m_{3}}n^{-\frac{2m_{1}+2m_{2}+m_{3}}{2}}\left|\mathcal{K}\left(m_{1},m_{2},m_{3}\right)\right|.
\end{align*}
 Now 
\[
\left|\mathcal{K}\left(m_{1},m_{2},m_{3}\right)\right|\le\binom{d}{m_{1}}d^{m_{1}}\binom{\ell}{m_{2}}d^{m_{2}}2^{m_{3}}\le\frac{d^{2m_{1}+m_{2}}\ell^{m_{2}}2^{m_3}}{m_{1}!m_{2}!}.
\]
Summing over all $m_{1},m_{2},m_{3}$ with $m_{2}+m_{3}=\ell$ completes
the proof. Indeed, 
\begin{align*}
\sum_{m_{1},m_{2},m_{3}}n^{-\frac{2m_{1}+2m_{2}+m_{3}}{2}}\left|\mathcal{K}\left(m_{1},m_{2},m_{3}\right)\right|&\le
\sum_{m_1,m_{2},m_{3}}n^{-\frac{2m_{1}+2m_{2}+m_{3}}{2}} \frac{d^{2m_{1}+m_{2}}\ell^{m_{2}}2^{m_3}}{m_{1}!m_{2}!}\\
&= \sum_{m_{2}+m_{3}=\ell} n^{-\frac{2m_{2}+m_{3}}{2}}\frac{d^{m_{2}}\ell^{m_{2}}2^{m_3}}{m_{2}!} \sum_{m_{1}=0}^{d-\ell}\frac{1}{m_1!} (d^2/n)^{m_1} \\
 & \le2\sum_{m_{2}+m_{3}=\ell} n^{-\frac{2m_{2}+m_{3}}{2}}\frac{d^{m_{2}}\ell^{m_{2}}2^{m_3}}{m_{2}!}\\
 & = 2^{\ell+1}n^{-\ell/2}\sum_{m_2=0}^{\ell} n^{-m_2/2} \frac{d^{m_{2}}\ell^{m_{2}}}{2^{m_2} m_{2}!}\\
 & \leq 2^{\ell+1}n^{-\ell/2}\left(1+\sum_{m_2=1}^{\ell} \left(\frac{ e d \ell}{2\sqrt{n} m_2}\right)^{m_2}\right)\\
 & \leq 2^{\ell+4}n^{-\ell/2} 2^{d\ell/\sqrt{n}},
\end{align*}
as required.
\end{proof}

With Lemma~\ref{lem:contiguity} in hand, the proof of Lemma~\ref{lem:the other side} (restated below) is easy.

\begin{replemma}{lem:the other side}
For all $\eps>0$, there exists $\delta>0$ such that if $d\le \delta n^{1/2}$ and $f:\mathbb{R}^{n \times n} \to \mathbb{R}$ is a comfortable $d$-junta, then $\|f\|_{L^2(\gamma)}\le (1+\eps)\|f\|_{L^2(\mu)}$.
\end{replemma}
\begin{proof}
We may assume that $\epsilon \leq 1/2$. By the triangle inequality, it suffices to show that $\|f(X)-f(XG)\|_{L^2(\mu,\text{GMD})}<\tfrac{\eps}{2}\norm{f}_{L^2(\gamma)}$, where $X\sim \mu$ and $G\sim \text{GMD}$. Since for any fixed, upper triangular $G \in \mathbb{R}^{n \times n}$, the function $g_G$ defined by
$$g_G: X \mapsto f(X)-f(XG)$$
is a row-comfortable $d$-row-junta, by Lemma \ref{lem:one sided bound} we have
\begin{align*}
\|f(X)-f(XG)\|_{L^2(\mu,\text{GMD})}^2 & = \mathbb{E}_{G \sim \text{GMD}} \mathbb{E}_{X \sim \mu}[(f(X)-f(XG))^2]\\
& = \mathbb{E}_{G \sim \text{GMD}} \mathbb{E}_{X \sim \mu}[g_G(X)^2]\\
& = \mathbb{E}_{G \sim \text{GMD}}[\|g_G\|_{L^2(\mu)}^2]\\
& \leq (1+\epsilon)^2\, \mathbb{E}_{G \sim \text{GMD}} \|g_G\|_{L^2(\gamma)}^2\\
& = (1+\epsilon)^2\, \mathbb{E}_{G \sim \text{GMD}} \mathbb{E}_{Y \sim \gamma} [(f(Y)-f(YG))^2]\\
& = (1+\epsilon)^2\, \|f(Y)-f(YG)\|_{L^2(\gamma,\text{GMD})}^2,\end{align*}
and therefore it suffices to show that 
\[
\|f(Y)-f(YG)\|_{L^2(\gamma,\text{GMD})}<\frac{\eps}{4}\|f\|_{L^2(\gamma)},
\] where $Y\sim \gamma$
and $G\sim \text{GMD}$ is independent of $Y$.
Expanding, we note that \[\|f(Y)-f(YG)\|_{L^2(\gamma,\text{GMD})}^2=\|f\|_{L^2(\gamma)}^{2}+\|f\|_{L^2(\nu)}^2 - 2 \langle f, \Tcol^* f\rangle_{L^2(\gamma)}.\]
To handle the cross term, we note that
by Cauchy-Schwarz and Claim~\ref{claim: T_col star is close to the identity}, we have
\[
\card{\langle f, \Tcol^* f-f\rangle_{L^2(\gamma)}}
\leq \norm{f}_{L^2(\gamma)}\norm{\Tcol^* f-f}_{L^2(\gamma)}
\leq \frac{\eps}{10}\norm{f}_{L^2(\gamma)}^2,
\]
so $\langle f, \Tcol^* f\rangle_{L^2(\gamma)} \geq (1-\epsilon/10)\|f\|_{L^2(\gamma)}^2$. Using
Lemma~\ref{lem:contiguity}, we have \[\|f\|_{L^2(\nu)}^2 \leq (1+\epsilon/20)\|f\|_{L^2(\gamma)}^2\] provided $\delta$ is sufficiently small, and so
$$ \|f(Y)-f(YG)\|_{L^2(\gamma,\text{GMD})}^2 
\leq \|f\|_{L^2(\gamma)}^{2}+\|f\|_{L^2(\nu)}^2 - 2 (1-\epsilon/10)\|f\|_{L^2(\gamma)}^{2} \leq(\eps/4)\|f\|_{L^2(\gamma)}^2,$$
completing the proof.
\end{proof}

\bibliographystyle{plain}
	\bibliography{ref}
	

    \appendix
    \section{Bounding the dimensions of high level representations}\label{sec:High dimensions}
    \subsection{The special orthogonal group $\SO(n)$}\label{subsec: quasirandomness of SO_n}
    \begin{replemma}{lemma:lb2son}
Let $n \geq 5$. If $\rho$ is an irreducible representation of $\SO(n)$ of level $d\geq n/2$, then
$$\dim(\rho) \geq \exp(n/32).$$
\end{replemma}
\begin{proof}
This follows from Weyl's original dimension formulae, together with a short computation. First suppose that $n=2k+1$ is odd. As mentioned above, the equivalence classes of irreducible representations of $\SO(2k+1,\mathbb{R})$ are in an explicit one-to-one correspondence with the partitions $\lambda$ (of non-negative integers) whose Young diagrams have at most $k$ rows. Weyl's dimension formula states that for any such partition $\lambda$, the corresponding irreducible representation $\rho_{\lambda}$ of $\SO(2k+1,\mathbb{R})$ has
$$\dim(\rho_{\lambda}) = \prod_{1 \leq i < j \leq k} \frac{\lambda_i-\lambda_j+j-i}{j-i} \prod_{1 \leq i \leq j \leq k}\frac{\lambda_i + \lambda_j+2k+1-i-j}{2k+1-i-j}.$$
Fix $\lambda$ a partition of degree $d\in\mathbb{N}$; trivially,
$$\dim(\rho_{\lambda}) \geq 
\prod_{1 \leq i \leq j \leq k} 
\frac{\lambda_i + 2k+1-i-j}{2k+1-i-j}
\geq \prod_{\substack{1 \leq i\leq k/2\\  k/2 \leq j \leq k}} 
\left(1+\frac{\lambda_i}{2k}\right)
=\prod_{1 \leq i\leq k/2} 
\left(1+\frac{\lambda_i}{2k}\right)^{k/2}.$$
If $\lambda_1\geq k$, then the previous product is at least 
$\left(1+1/2\right)^{k/2}\geq {\sf exp}(n/16)$ and we are done, so 
assume that $\lambda_1<k$ and hence all $\lambda_i$'s are 
smaller than $k$. For all $0\leq \delta<1/2$ we have 
$1+\delta\geq e^{\delta/2}$, so the previous product is at least
\[
\prod_{1 \leq i\leq k/2} 
e^{\frac{\lambda_i}{4k}\cdot \frac{k}{2}}
=e^{\frac{1}{8}\sum\limits_{i=1}^{k/2}\lambda_i}.
\]
Since $\lambda_1\geq \lambda_2\geq\ldots\geq \lambda_k$, we get 
that $\sum\limits_{i=1}^{k/2}\lambda_i\geq \frac{1}{2}\sum\limits_{i=1}^{k}\lambda_i=\frac{d}{2}$, 
so the last expression is at least $e^{d/16}$ and we are done.

The case of even $n$ is very similar. Let $n=2k$; then the equivalence classes of irreducible representations of $\SO(2k,\mathbb{R})$ are in an explicit correspondence with the partitions $\lambda$ (of non-negative integers) whose Young diagrams have at most $k$ rows: this correspondence is one-to-one when the number of rows is less than $k$, but when the number of rows is equal to $k$, the correspondence is two-to-one (each partition $\lambda$ with $k$ rows corresponds to two irreducible representations $\rho_{\lambda}$ and $\tilde{\rho}_{\lambda}$ of the same dimension). Writing $c = c(\lambda) = 1$ if $\lambda$ has less than $k$ rows and $c = c(\lambda) = 1/2$ if $\lambda$ has exactly $k$ rows, the dimensions of the corresponding irreducible representations are given by the formula
$$\dim(\rho_{\lambda}) = c \prod_{1 \leq i < j \leq k} \left(\frac{\lambda_i-\lambda_j+j-i}{j-i}\right) \left(\frac{\lambda_i + \lambda_j+2k-i-j}{2k-i-j}\right).$$
From now on we can apply exactly the same argument as in the case of $n$ odd; we omit the details.
\end{proof}
\subsection{The spin group $\Spin(n)$}\label{subsec:Quasirandomness of spin representations}
\begin{lemma}
Let $\rho_{\lambda}$ be as in Theorem \ref{thm:Spin is quasirandom}. Then $$\dim(\rho_{\lambda}) = \prod_{1\leq i < j \leq k} \frac{\lambda_i-\lambda_j-i+j}{j-i}\prod_{1 \leq i \leq j \leq k}\frac{\lambda_i+\lambda_j+2k+1-i-j}{2k+1-i-j}\ge 2^{-\Omega(n)}.$$
\end{lemma}
\begin{proof}
Indeed, if $a_k$ is odd then we must have $a_k \geq 1$ and therefore $\lambda_i \geq 1/2$ for all $i \in [k]$. Hence, rather crudely, we have
\begin{align*}
\dim(\rho_{\lambda}) & \geq 
\prod_{1 \leq i \leq j \leq k}\frac{\lambda_i+\lambda_j+2k+1-i-j}{2k+1-i-j}\\
& = \prod_{1 \leq i \leq j \leq k}\left( 1+ \frac{\lambda_i+\lambda_j}{2k+1-i-j}\right)\\
& \geq \prod_{1 \leq i \leq k/2,\atop k/2 \leq j \leq k}\left( 1+ \frac{\lambda_i}{2k}\right)\\
& \geq \prod_{1 \leq i \leq k/2}\left( 1+ \frac{\lambda_i}{2k}\right)^{k/2}\\
&\geq \prod_{1 \leq i \leq k/2}\left( 1+ \frac{\min\{\lambda_i,2k\}}{2k}\right)^{k/2}\\
& \geq \prod_{1 \leq i \leq k/2}\exp(\min\{\lambda_i,2k\}/8)\\
& = \exp\left(\tfrac{1}{8}\sum_{1 \leq i \leq k/2}\min\{\lambda_i,2k\}\right)\\
& \geq \exp(k/32)\\
& = \exp((n-1)/64),
\end{align*}
as required, using the fact that $1+x \geq e^{x/2}$ for all $x \leq 1$.

For all $n=2k \geq 6$ even, we have 
$$\dim(\rho_{\lambda}) = \prod_{1 \leq i < j \leq k} \frac{\lambda_i-\lambda_j-i+j}{j-i}\frac{\lambda_i+\lambda_j+2k-i-j}{2k-i-j},$$
where the $k$-tuple $\boldsymbol{\lambda} = (\lambda_1, \lambda_2,\ldots \lambda_k)$ ranges over all $k$-tuples defined by
$$\lambda_i = a_i+a_{i+1}+\ldots+a_{k-2}+\tfrac{1}{2}(a_{k-1}+a_k) \ \forall\ i \leq k-2,\quad \lambda_{k-1} = \tfrac{1}{2}(a_{k-1}+a_k),\quad \lambda_k = \tfrac{1}{2}(a_k-a_{k-1}),$$
for some $(a_i)_{i=1}^{k} \in (\mathbb{N} \cup \{0\})^k$. The case of $a_{k-1}+a_k$ even corresponds to irreducible representations of $\text{Spin}(n)$ that are also irreducible representations of $\SO(n,\mathbb{R})$; the dimensions of these were bounded in the previous subsection. The case of $a_{k-1}+a_k$ odd corresponds to `new' irreducible representations of $\text{Spin}(n)$, but the above equation quickly implies that any such has dimension at least $2^{\Omega(n)}$. Indeed, if $a_{k-1}+a_k$ is odd then we must have $a_{k-1}+a_k \geq 1$ and therefore $\lambda_i \geq 1/2$ for all $i \leq k-1$. Hence, again rather crudely, we have
\begin{align*}\dim(\rho_{\lambda}) & \geq 
\prod_{1 \leq i < j \leq k}\frac{\lambda_i+\lambda_j+2k-i-j}{2k-i-j}\\
& = \prod_{1 \leq i < j \leq k}\left( 1+ \frac{\lambda_i+\lambda_j}{2k-i-j}\right)\\
& \geq \prod_{1 \leq i \leq k/2,\atop k/2 < j \leq k}\left( 1+ \frac{\lambda_i}{2k}\right)\\
& \geq \prod_{1 \leq i \leq k/2}\left( 1+ \frac{\lambda_i}{2k}\right)^{(k-1)/2}\\
&\geq \prod_{1 \leq i \leq k/2}\left( 1+ \frac{\min\{\lambda_i,2k\}}{2k}\right)^{k/4}\\
& \geq \prod_{1 \leq i \leq k/2}\exp(\min\{\lambda_i,2k\}/16)\\
& = \exp\left(\tfrac{1}{16}\sum_{1 \leq i \leq k/2}\min\{\lambda_i,2k\}\right)\\
& \geq \exp(k/64)\\
& = \exp(n/128),
\end{align*}
as required, again using the fact that $1+x \geq e^{x/2}$ for all $x \leq 1$.
\end{proof}

\subsection{The special unitary group $\SU(n)$}\label{quasirandomness of SU_n}
\begin{replemma}{lem:su2lb}
  If $\rho$ is an irreducible representation of $\SU(n)$ of level $D$, then
  ${\sf dim}(\rho)\geq 2^{\Omega(\min(D,n))}$.
\end{replemma}
\begin{proof}
Weyl's dimension formula states that the dimension of the irreducible representation of $\SU(n)$ corresponding to the partition $\lambda$ (with at most $n-1$ parts) is
$$\prod_{1 \leq i < j \leq n}\frac{\lambda_i-\lambda_j+j-i}{j-i}.$$
Assume first that $\lambda_1 \geq n$. If $\lambda_{\lfloor n/2 \rfloor+1} \leq n/2$ then, considering all the terms in the above product corresponding to $i=1$ and $j\geq \lfloor n/2\rfloor +1$, we see that the above product satisfies

\begin{align*}\prod_{1 \leq i < j \leq n}\frac{\lambda_i-\lambda_j+j-i}{j-i} & \geq \prod_{j \geq \lfloor n/2 \rfloor +1}\frac{\lambda_1-\lambda_j+j-1}{j-1}\\
& \geq \prod_{j \geq \lfloor n/2 \rfloor +1}\frac{\lambda_1-\lambda_j+n-1}{n-1}\\
&\geq \prod_{j \geq \lfloor n/2 \rfloor +1}\frac{n-n/2+n-1}{n-1}\\
&= \exp(\Theta(n)).
\end{align*}

If, on the other hand, we have $\lambda_{\lfloor n/2 \rfloor+1} > n/2$, then considering all the terms in the above product corresponding to $j=n$ and $i \leq \lfloor n/2 \rfloor +1$, we obtain

\begin{align*}\prod_{1 \leq i < j \leq n}\frac{\lambda_i-\lambda_j+j-i}{j-i} & \geq \prod_{i \leq \lfloor n/2 \rfloor +1}\frac{\lambda_i-\lambda_n+n-i}{n-i}\\
& \geq \prod_{i \leq \lfloor n/2 \rfloor +1}\frac{\lambda_i-\lambda_n+n-1}{n-1}\\
&\geq \prod_{i \leq \lfloor n/2 \rfloor +1}\frac{n/2-0+n-1}{n-1}\\
&= \exp(\Theta(n)).
\end{align*}

We may henceforth assume that $\lambda_1 < n$. In this case we have $\lambda_i - \lambda_j < n$ for all $i,j$. Hence, the above product satisfies

\begin{align*}\prod_{1 \leq i < j \leq n}\frac{\lambda_i-\lambda_j+j-i}{j-i} & \geq \prod_{1 \leq i < j \leq n}\frac{\lambda_i-\lambda_j+n-1}{n-1}\\
& \geq \prod_{1 \leq i < j \leq n}\left(1+\frac{\lambda_i-\lambda_j}{n-1}\right)\\
& \geq \prod_{1 \leq i < j \leq n} \exp\left(\frac{\lambda_i-\lambda_j}{2(n-1)}\right)\\
& = \exp\left(\frac{1}{2(n-1)}\sum_{1 \leq i < j \leq n}(\lambda_i-\lambda_j) \right)\\
& = \exp\left(\frac{1}{2(n-1)} \sum_{1 \leq i < j \leq n}(a_i+\ldots+a_{j-1}) \right)\\
& = \exp\left(\frac{(n-1)a_1+2(n-2)a_2+\ldots+2(n-2)a_{n-2}+(n-1)a_{n-1}}{2(n-1)}\right)\\
& \geq \exp(\Theta(D)).
\end{align*}
\end{proof}

\section{The required adaptations  for showing that $\Sp(n)$ and $\SU(n)$ are good} 

Here we complete the proof that $\Sp(n)$ and $\SU(n)$ are good. In fact, we will only explain how to adapt the proof for $\Sp(n)$, as $\SU(n)$ is only simpler (one just works with complex matrices rather than quaternionic ones).   

To construct our noise operator $\T_{\rho}$ on $\textup{Sp}(n)$ we use the identification
$$\textup{Sp}(n) = \{X \in \mathbb{H}^{n \times n}:\ XX^{h} = I\} = \{X \in \mathbb{H}^{n \times n}:\ X^hX = I\}$$
of $\textup{Sp}(n)$ with the group of unitary quaternionic matrices. (Here, $X^{h}$ denotes the quaternionic conjugate of $X$, i.e.\ $(X^h)_{p,q} = \overline{X_{q,p}}$, where $\overline{a+b\mathbf{i}+c\mathbf{j}+d\mathbf{k}} = a-b\mathbf{i}-c\mathbf{j}-d\mathbf{k}$ for $a,b,c,d \in \mathbb{R}$.) We couple $\textup{Sp}(n)$ with the space $(\mathbb{H}^{n \times n},\gamma)$ of quaternionic normal random matrices, where the real part, the $\mathbf{i}$-part (= coefficient of $\mathbf{i}$), the $\mathbf{j}$-part and the $\mathbf{k}$-part of each entry are independent normal (real-valued) random variables with mean zero and variance $1/4$, and all the entries are independent. (The following terminology will be useful in the sequel. Recall that a {\em standard normal quaternionic} or $\mathcal{QN}$ random variable is a quaternion-valued random variable where the real part, the $\mathbf{i}$-part, the $\mathbf{j}$-part and the $\mathbf{k}$-part of the value of the random variable are independent normal (real-valued) random variables with mean zero and variance $1/4$; so a quaternionic normal random matrix $\sim (\mathbb{H}^{n \times n},\gamma)$ is simply a matrix where each entry is an independent $\mathcal{QN}$ random variable. A {\em $\mathcal{QN}$ random vector in $n$ dimensions} is a vector of $n$ independent $\mathcal{QN}$ random variables.)

To make this coupling work, we need to define a `Gram-Schmidt' type process (on the columns, and also on the rows) which, when applied to a matrix in $(\mathbb{H}^{n \times n},\gamma)$, yields an element of $\textup{Sp}(n)$ with probability one. We define two inner products on $\mathbb{H}^n$:
$$\langle x,y \rangle:=  \sum_{i=1}^{n} \overline{x_i} y_i, \quad \langle x,y\rangle' : = \sum_{i=1}^{n} x_i \overline{y_i}.$$
It is easy to check that, for $H \in \mathbb{H}^{n \times n}$, we have $\langle Hx,Hy \rangle = \langle x,y \rangle$ for all $x,y \in \mathbb{H}^n$ if and only if $H \in \textup{Sp}(n)$. 

Our Gram-Schmidt process on the columns, ${\sf GS}_{{\sf col}}(X)$ for $X \in \mathbb{H}^{n \times n}$, is defined as follows. If $c_1,c_2,\ldots,c_n$ denote the columns of $X$, then we (inductively) define
$$\gamma_k := c_k - \sum_{\ell < k}\tilde{c}_{\ell} \langle \tilde{c}_{\ell},c_k \rangle$$
and (if $\gamma_k \neq 0$)
$$\tilde{c}_k := \gamma_k/\sqrt{\langle \gamma_k,\gamma_k \rangle},$$
for $1 \leq k \leq n$. If any $\gamma_k = 0$ then ${\sf GS}_{{\sf col}}(X)$ is undefined; otherwise we define ${\sf GS}_{{\sf col}}(X)$ to be $\sqrt{n}$ times the matrix $\tilde{X}$ in $\mathbb{H}^{n \times n}$ with columns $\tilde{c}_1,\tilde{c}_2,\ldots,\tilde{c}_n$; it is easy to see that $\tilde{X}\in \textup{Sp}(n)$. (One checks, by induction on $k$, that $\langle \tilde{c}_{\ell},\tilde{c}_k \rangle = \delta_{j,k}$ for all $1 \leq j \leq k \leq n$; taking conjugates this implies that $\langle \tilde{c}_{k},\tilde{c}_\ell \rangle = \delta_{j,k}$ for all $1 \leq j \leq k \leq n$, and these two statements together imply that $\tilde{X} \in \textup{Sp}(n)$. It is clear that, if $X$ is sampled according to $\gamma$, then all the $\gamma_k$ are non-zero with probability one, so ${\sf GS}_{{\sf col}}(X)$ is defined outside a set of zero probability measure.

We define the Gaussian noise operator $U_{\rho}:L^2(\mathbb{H}^{n \times n},\gamma) \to L^2(\mathbb{H}^{n \times n},\gamma)$ in the obvious way: for $f \in L^2(\mathbb{H}^{n \times n},\gamma)$ we define
$$U_\rho f(X) = \mathbb{E}_{Y \sim \gamma} [f(\rho X + \sqrt{1-\rho^2}Y)].$$
It is easy to check that $U_\rho$ is self-adjoint; indeed,
$$\mathbb{E}_{X \sim \gamma}[\overline{f(X)}U_\rho g(X)] = \mathbb{E}_{X,Z \sim \gamma,\ \rho\text{-correlated}} \overline{f(X)}{g(Z)} = \mathbb{E}_{X \sim \gamma}[\overline{U_\rho f(X)}g(X)].$$

As in the case of $\SO(n)$, we let $\mu$ denote the measure on $\sqrt{n}\Sp(n)$ induced (under dilation by $\sqrt{n}$) from the Haar probability measure on $\Sp(n)$. For $f \in L^2(\textup{Sp}(n))$ we define $\T_{{\sf col}}:L^2(\sqrt{n} \textup{Sp}(n),\mu) \to L^2(\mathbb{H}^{n \times n},\gamma)$ by
$$\T_{{\sf col}}f(X) = f({\sf GS}_{{\sf col}}(X)).$$
We similarly let $T_{{\sf col}}^*$ denote its (Hilbert-space) adjoint.

Again, as in the case of $\SO(n)$, we define
$$\T_{\rho} = \mathbb{E}_{V \sim \textup{Sp}(n)} [R_{V}^*\T_{{\sf col}}^*U_\rho \T_{{\sf col}}R_V].$$

Since $U_\rho$ is self-adjoint, so is $\T_{\rho}$. The fact that $\T_{{\sf col}}$ commutes with $L_{V}$ for all $V \in \Sp(n)$ follows from the fact that ${\sf GS}_{{\sf col}}(VX) = V {\sf GS}_{{\sf col}}(X)$ for all $X \in \mathbb{H}^{n \times n}$ and all $V \in \Sp(n)$, which in turn follows from the fact that $ \langle Vx,Vy \rangle = \langle x,y\rangle$ for all $x,y \in \mathbb{H}^n$ and all $V \in \Sp(n)$. The fact that $L_V$ and $R_V$ both commute with $U_{\rho}$ follows from the fact that, if $X \sim (\mathbb{H}^{n \times n},\gamma)$ and $V \in \Sp(n)$, then $VX \sim (\mathbb{H}^{n \times n},\gamma)$ and $XV \sim (\mathbb{H}^{n \times n},\gamma)$, as in the case of $\SO(n)$ and $(\mathbb{R}^n,\gamma)$ (the proof is very similar). For any $V \in \Sp(n)$, we have $L_V^* = L_{V^{-1}}$ and $R_V^* = R_{V^{-1}}$, and therefore, taking the adjoints of
$$L_V \T_{{\sf col}} = \T_{{\sf col}} L_V,\quad R_V \T_{{\sf col}} = \T_{{\sf col}} R_V$$
yields
$$\T_{{\sf col}}^*L_{V^{-1}} = L_{V^{-1}} \T_{{\sf col}}^*,\quad \T_{{\sf col}}^*R_{V^{-1}} = R_{V^{-1}}\T_{{\sf col}}^*,$$
and therefore $L_V$ and $R_V$ commute with $\T_{{\sf col}}^*$, for all $V \in \Sp(n)$. It follows that $L_V$ commutes with $\T_\rho$. The fact that for any $W \in \Sp(n)$, $R_{W}$ commutes with $\T_\rho$, follows from a simple change of variables:
\begin{align*} R_W \T_\rho & = R_{W} \mathbb{E}_{V \sim \textup{Sp}(n)} [R_{V}^*\T_{{\sf col}}^*U_\rho \T_{{\sf col}}R_V]\\
& = \mathbb{E}_{V \sim \textup{Sp}(n)} [R_{W^{-1}}^*R_{V}^*\T_{{\sf col}}^*U_\rho \T_{{\sf col}}R_V]\\
& = \mathbb{E}_{V \sim \textup{Sp}(n)}[(R_V R_{W^{-1}})^* \T_{{\sf col}}^*U_\rho \T_{{\sf col}}R_V]\\
& = \mathbb{E}_{V \sim \textup{Sp}(n)}[R_{VW^{-1}}\T_{{\sf col}}^*U_\rho \T_{{\sf col}}R_V]\\
& = \mathbb{E}_{V \sim \textup{Sp}(n)}[R_{V}\T_{{\sf col}}^*U_\rho \T_{{\sf col}}R_{VW}]\\
& = \mathbb{E}_{V \sim \textup{Sp}(n)}[R_{V}\T_{{\sf col}}^*U_\rho \T_{{\sf col}}R_{V}R_{W}]\\
& = \T_{\rho}R_{W}.
\end{align*}
Hence, $\T_\rho$ commutes with the action of $\Sp(n)$ from both left and right, as in the $\SO(n)$ case.

We also need to define an analogue of $\T_{{\sf row}}$. However, this is a little different to in the $\SO(n)$ case. For $V \in \mathbb{H}^{n \times n}$, $X \in \Sp(n)$ does not imply that $\langle e_i X, e_j X \rangle = \delta_{i,j}$ for each $i,j \in [n]$ (the latter would be the analogue of `orthonormal rows'). The condition $X \in \Sp(n)$ is, however, equivalent to the condition $\langle e_i X, e_j X \rangle' = \delta_{i,j}$ for each $i,j \in [n]$ (which is in turn equivalent to the condition $\langle e_i \overline{X},e_j \overline{X} \rangle = \delta_{i,j}$ for each $i,j \in [n]$). We therefore define our Gram-Schmidt process on the rows, ${\sf GS}_{{\sf row}}(X)$ for $X \in \mathbb{H}^{n \times n}$, as follows. If $r_1,r_2,\ldots,r_n$ denote the rows of $X$, then we (inductively) define 
$$\delta_k: = r_k - \sum_{\ell < k} \overline{\langle \tilde{r}_{\ell},r_k\rangle'} \tilde{r}_{\ell}$$
and (if $\delta_k \neq 0)$
$$\tilde{r}_k: = \delta_k/\sqrt{\langle \delta_k,\delta_k\rangle'},$$
for $1 \leq k \leq n$. If any $\delta_k = 0$, then ${\sf GS}_{{\sf row}}(X)$ is undefined; otherwise we define ${\sf GS}_{{\sf row}}(X)$ to be $\sqrt{n}$ times the matrix $\tilde{X} \in \mathbb{H}^{n \times n}$ with rows $\tilde{r}_1,\tilde{r}_2,\ldots,\tilde{r}_n$; it is easy to see that $\tilde{X} \in \Sp(n)$, similarly to in the case of ${\sf GS}_{{\sf col}}(X)$. Again, as with ${\sf GS}_{{\sf col}}(X)$, if $X$ is sampled according to $\gamma$, then all the $\delta_k$ are non-zero with probability one, so ${\sf GS}_{{\sf row}}(X)$ is defined outside a set of zero probability measure.

For $f \in L^2(\textup{Sp}(n))$, we define $\T_{{\sf row}}:L^2(\sqrt{n} \textup{Sp}(n),\mu) \to L^2(\mathbb{H}^{n \times n},\gamma)$ by
$$\T_{{\sf row}}f(X) = f({\sf GS}_{{\sf row}}(X)),$$
and of course, we let $T_{{\sf row}}^*$ denote its (Hilbert-space) adjoint.

As in the $\SO(n)$ case, we observe that $\T_{{\sf row}}$ commutes with $R_{V}$ for all $V \in \Sp(n)$; this follows from the fact that ${\sf GS}_{{\sf row}}(XV) =  {\sf GS}_{{\sf row}}(X)V$ for all $X \in \mathbb{H}^{n \times n}$ and all $V \in \Sp(n)$, which in turn follows from the fact that $ \langle xV,yV \rangle' = \langle x,y\rangle'$ for all $x,y \in \mathbb{H}^n$ and all $V \in \Sp(n)$.

Similarly to as in the $\SO(n)$ case, if $Y \sim (\mathbb{H}^{n \times n},\gamma)$, and $X = {\sf GS}_{{\sf col}}(Y)$, we obtain $Y = XG$, where $g_{i,i}$ is $1/\sqrt{n}$ times the (Euclidean) length of a $\mathcal{QN}$ random vector in $n-i+1$ dimensions, $g_{i,j} = 0$ for all $i > j$, $g_{i,j}$ is $1/\sqrt{n}$ times a $\mathcal{QN}$ random variable, the entries of $G$ are independent, and independent of all the entries of $X$. This (distribution over) quaternionic upper-triangular matrices $G$ is our Gaussian Maker Distribution (or GMD, for short) in the $\Sp(n)$ case. To generalise the $\SO(n)$ proof, the only (important) facts we need are the fact that $\Sp(n)$ acts transitively (from either the left or the right) on the set $\mathcal{S} = \{v \in \mathbb{H}^n:\ \langle v,v \rangle = 1\}$ of quaternionic vectors of unit norm, and that $(\mathbb{H}^{n \times n},\gamma)$ is invariant under both left and right actions of $\Sp(n)$. 

We can therefore write
$$\T_{{\sf col}}^*f(X) = \mathbb{E}_{G \sim \text{GMD}} f(XG)\quad \forall X \in \sqrt{n}\Sp(n).$$
Similarly, we can write
$$\T_{{\sf row}}^*f(X) = \mathbb{E}_{G \sim \text{GMD}} f(G^{T}X)\quad \forall X \in \sqrt{n}\Sp(n).$$

Our `nice' functions on $L^2(\mathbb{H}^{n \times n},\gamma)$ are polynomials (with complex coefficients) where each variable is a real-part, an $\mathbf{i}$-part, a $\mathbf{j}$-part or a $\mathbf{k}$-part of one of the matrix entries. We say a monomial in these variables is {\em comfortable} if no two of its variables come from the same row, no two of its variables come from the same column, and all variables come from the first $\lfloor n/2\rfloor$ rows and the first $\lfloor n/2 \rfloor$ columns. Similarly, we say it is {\em row comfortable} (respectively {\em column comfortable} if no two of its variables come from the same row (respectively column), and all of its variables come from the first $\lfloor n/2\rfloor$ rows (respectively columns). A comfortable polynomial (respectively row-comfortable polynomial or column-comfortable polynomial) is a complex linear combination of comfortable (respectively row-comfortable or column-comfortable) monomials. Such a polynomial is said to be a {\em $d$-junta} if it is homogeneous of degree $d$ and depends only upon the top-left $d$ by $d$ minor, it is said to be a {\em $d$-row-junta} if it depends only upon the first $d$ rows, and it is said to be a {\em $d$-column-junta} if it depends only upon the first $d$ columns.

As in the $\SO(n)$-case, we define $\Pi_{{\sf comf}}\colon L^2(\mathbb{H}^{n\times n},\gamma)\to L^2(\mathbb{H}^{n\times n},\gamma)$ to be orthogonal projection onto the linear subspace of comfortable polynomials. Similarly, we define the operators $\Pi_{{\sf comf}, d}$, $\Pi_{{\sf comf}, {\sf col}, d}$ and 
$\Pi_{{\sf comf}, {\sf row}, d}$ to be the orthogonal projections onto
the space of comfortable, row comfortable and column comfortable polynomials of degree at most $d$, respectively. Finally, if $S \subset [n]$ is a set of rows, we define $\Pi_{{\sf comf}, =S}$ to be the projection onto the subspace spanned by the comfortable homogeneous monomials of degree $|S|$ which depend only upon variables from the rows in $S$.

Since the real-part, the $\mathbf{i}$-part, the $\mathbf{j}$-part and the $\mathbf{k}$-part of each matrix entry of $X \sim (\mathbb{H}^{n \times n},\gamma)$ is $N(0,1/4)$-distributed rather than $N(0,1)$-distributed, to guarantee orthonormality we must multiply by a factor of $2^d$, so a `generic' row-comfortable monomial of degree $d$ is of the form
$$2^d (X_{i_1,j_1})_{q_1\text{-part}} \cdot (X_{i_2,j_2})_{q_2\text{-part}} \cdot \ldots \cdot (X_{i_d,j_d})_{q_d\text{-part}},$$
where $q_k \in \{\text{real},\mathbf{i},\mathbf{j},\mathbf{k}\}$ for all $k \in [d]$ and $i_1,\ldots,i_d$ are distinct integers between $1$ and $n/2$; for brevity we denote this by $H_{\alpha}$ where $\alpha = \{(i_1,j_1;q_1),(i_2,j_2;q_2),\ldots,(i_d,j_d;q_d)\}$.

The proof (and statement) of Claim \ref{claim:comf_ev} is readily adapted. If
$$S =\{(i_1,j_1;q_1),(i_2,j_2;q_2),\ldots,(i_d,j_d;q_d)\}$$
is such that the $j_k$ are all distinct, then we have
\begin{align*}
    \T_{{\sf col}}^*f(X) &= \mathbb{E}_{G \sim \text{GMD}} H_S(XG)\\
    & = 2^d\mathbb{E}_{G \sim \text{GMD}}\left[ \prod_{k=1}^{d}((XG)_{i_k,j_k})_{q_k\text{-part}}\right]\\
    & = 2^d\mathbb{E}_{G \sim \text{GMD}}\left[ \prod_{k=1}^{d}\left(\sum_{\ell=1}^{j_k} (X_{i_k,\ell} G_{\ell,j_k})_{q_k\text{-part}}\right)\right]\\
    & = 2^d \prod_{k=1}^{d}\mathbb{E}_{G \sim \text{GMD}}\left[\sum_{\ell=1}^{j_k} (X_{i_k,\ell} G_{\ell,j_k})_{q_k\text{-part}}\right]\\
    & = 2^d \prod_{k=1}^{d}\left(\sum_{\ell=1}^{j_k} (X_{i_k,\ell} \mathbb{E}_G[G_{\ell,j_k}])_{q_k\text{-part}}\right)\\
    & = 2^d \prod_{k=1}^{d} (X_{i_k,j_k} \mathbb{E}_G[G_{j_k,j_k}])_{q_k\text{-part}}\\
    & = 2^d \prod_{k=1}^{d} (X_{i_k,j_k})_{q_k\text{-part}} \mathbb{E}_G[G_{j_k,j_k}]\\
    & = \lambda_S H_S(X),
\end{align*}
where 
$$\lambda_S: = \prod_{k=1}^{d} \mathbb{E}_G[G_{j_k,j_k}].$$
(Note that we used the fact that $\mathbb{E}_G[G_{\ell,j_k}] = 0$ for all $\ell \neq j_k$, and that $\mathbb{E}_G[G_{j_k,j_k}] \in \mathbb{R}$.) The rest of the proof is almost exactly the same as before.

The analogue of Claim \ref{claim: L_U works properly} is that, for any row-comfortable $d$-row-junta $f:\mathbb{H}^{n \times n} \to \mathbb{C}$, we have
$$\mathbb{E}_{V\sim \Sp(n)} \|\Pi_{{\sf comf}, d}R_V f\|_{L^2(\gamma)}^2 \geq  \frac{(n/2)!}{2^{d} n^d((n/2)-d)!}\|f\|_{L^2(\mu)}^2.$$

The proof of Claim \ref{claim: L_U works properly} is readily adapted. We first note that, as before, if $H_{\alpha}$ and $H_{\beta}$ are row-comfortable monomials such that the set of rows appearing in $\alpha$ is different from the set of rows appearing in $\beta$, then we have $\langle H_{\alpha},H_{\beta} \rangle_{L^2(\gamma)} = \langle H_{\alpha},H_{\beta} \rangle_{L^2(\mu)}=0$; the same proof works as before.

For $S \subset [n]$, let $W_{S}$ denote the (complex) linear span of the row-comfortable monomials of degree exactly $|S|$ which depend only upon variables from the rows in $S$. As before, the $W_S$ are pairwise orthogonal with respect to both $L^2(\gamma)$ and $L^2(\mu)$. 

Again as before, for any row-comfortable $d$-row junta $g:\mathbb{H}^{n \times n} \to \mathbb{C}$, we may write $g$ as an orthogonal direct sum,
$$g = \sum_{S \subset [d]}g_{(=S)},$$
where $g_{(=S)}$ denotes the orthogonal projection of $g$ onto $W_S$ (with respect to $L^2(\gamma)$). It is clear that $(R_Vg)_{(=S)} = R_V(g_{(=S)})$ for any $V \in \Sp(n)$ and any row-comfortable $d$-row-junta $g$.

Now let $f:\mathbb{H}^{n \times n} \to \mathbb{C}$ be a row-comfortable $d$-row-junta. Since the $f_{(=S)}$ are pairwise orthogonal with respect to $L^2(\mu)$ as well as with respect to $L^2(\gamma)$, we have
$$\|f\|_{L^2(\mu)}^2 = \sum_{S \subset [d]}\|f_{(=S)}\|_{L^2(\mu)}^2,$$
so there exists $S \subset [d]$ such that $\|f_{(=S)}\|_{L^2(\mu)}^2 \geq \frac{1}{2^d}\|f\|_{L^2(\mu)}^2$.

For brevity, write $h = f_{(=S)}$, $d':=|S|$ and $S = \{i_1,\ldots,i_{d'}\}$. We need to show that
\begin{equation}\label{eq:h-eqn} \mathbb{E}_{V\sim \Sp(n)} \|\Pi_{{\sf comf}, =S}R_V h\|_{L^2(\gamma)}^2\geq \frac{(n/2)!}{n^{d'}((n/2)-d')!}\|h\|_{L^2(\mu)}^2.
\end{equation}
In proving this, we may assume without loss of generality that $S = \{1,2,\ldots,d'\}$. In this case, we may write
\begin{equation}\label{eq:h-expansion} h = \sum_{\alpha = \{(1,j_1;q_1),\ldots,(d,j_d;q_d)\}} \hat{f}(\alpha)H_{\alpha}.\end{equation}
For brevity, write $H_{\alpha_0}: = H_{\{(1,1;\mathbb{R}),(2,2;\mathbb{R}),\ldots,(d',d';\mathbb{R})\}}$, i.e., $H_{\alpha_0}$ denotes the degree-$d'$ monomial
$$X \mapsto 2^d (X_{1,1})_{\text{real-part}} \cdot (X_{2,2})_{\text{real-part}} \cdot \ldots \cdot (X_{d',d'})_{\text{real-part}}.$$

If $\alpha = \{(1,j_1;q_1),(2,j_2;q_2),\ldots,(d',j_{d'};q_{d'})\}$ is such that the $j_k$ are all distinct, then $H_\alpha = R_{V_\alpha} H_{\alpha_0}$, where $V_\alpha  = \Sigma D$, $\Sigma$ is the permutation matrix corresponding to some permutation $\sigma \in S_n$ satisfying $\sigma^{-1}(i)=j_i$ for all $i \in [d']$ ($\Sigma_{i,j} := \delta_{\{j = \sigma(i)\}}$ for all $i,j \in [n]$), and $D$ is a diagonal matrix with $D_{i,i} = \overline{q_i}$ for all $i \in [d']$. It follows that $H_{\alpha_0} = R_{V_\alpha^{-1}} H_\alpha$. Since $V_{\alpha} \in \Sp(n)$ for any such $V_\alpha$, and since $V_\alpha V \sim \Sp(n)$ for $V \sim \Sp(n)$,
we have
\begin{align*} \mathbb{E}_{V\sim \Sp(n)} \|\Pi_{{\sf comf}, d}R_V f\|_{L^2(\gamma)}^2 &= \sum_{\alpha:\ H_\alpha \text{ comfortable}} \mathbb{E}_{V \sim \Sp(n)} |\langle R_{V} f,H_{\alpha}\rangle|^2\\
&= \sum_{\alpha:\ H_\alpha \text{ comfortable}} \mathbb{E}_{ V \sim \Sp(n)} |\langle R_{V_\alpha V} f,H_{\alpha}\rangle|^2\\
& = \sum_{\alpha:\ H_\alpha \text{ comfortable}} \mathbb{E}_{ V \sim \Sp(n)} |\langle R_{V_\alpha}R_{V} f,H_{\alpha}\rangle|^2\\
& = \sum_{\alpha:\ H_\alpha \text{ comfortable}} \mathbb{E}_{V_{\alpha} V \sim \Sp(n)} |\langle R_V f,R_{V_\alpha^{-1}} H_{\alpha}\rangle|^2\\
& = 4^{d'}\frac{(n/2)!}{((n/2)-d')!}\mathbb{E}_{V \sim \Sp(n)} |\langle R_V f,H_{\alpha_0}\rangle|^2.
\end{align*}
Now observe that
$$R_V H_\alpha (X) = H_\alpha(XV) = 2^{d'} \prod_{k=1}^{d'} ((XV)_{k,j_k})_{q_k\text{-part}} = 2^{d'} \prod_{k=1}^{d'}\left(\sum_{i=1}^{n}X_{k,i}V_{i,j_k}\right)_{q_k\text{-part}},$$
and therefore
$$\langle R_V H_\alpha(X),H_{\alpha_0}(X)\rangle = \prod_{k=1}^{d'}(V_{k,j_k})_{q_k\text{-part}} = 2^{-{d'}}H_{\alpha}(V).$$
Using the expansion (\ref{eq:h-expansion}) of $h$, it follows that
$$\langle R_V h,H_{\alpha_0}\rangle = 2^{-d'} \sum_{\alpha}\hat{f}(\alpha)H_\alpha(V) = 2^{-d'}f(V),$$
and therefore
\begin{align*}\mathbb{E}_{V\sim \Sp(n)} \|\Pi_{{\sf comf}, =S}R_V h\|_{L^2(\gamma)}^2 & = 4^{d'}\frac{(n/2)!}{((n/2)-d')!}\mathbb{E}_{V \sim \Sp(n)} 4^{-d'}|h(V)|^2\\
&= \frac{(n/2)!}{n^{d'}((n/2)-d')!}\|h\|_{L^2(\mu)}^2.
\end{align*}
The rest of the proof is exactly as before.

\remove{The proof (and statement) of Claim \ref{claim:comf_ev} is readily adapted. If
$$S =\{(i_1,j_1;q_1),(i_2,j_2;q_2),\ldots,(i_d,j_d;q_d)\}$$
is such that the $j_k$ are all distinct, then we have
\begin{align*}
    \T_{{\sf col}}^*f(X) &= \mathbb{E}_{G \sim \text{GMD}} H_S(XG)\\
    & = 2^d\mathbb{E}_{G \sim \text{GMD}}\left[ \prod_{k=1}^{d}((XG)_{i_k,j_k})_{q_k\text{-part}}\right]\\
    & = 2^d\mathbb{E}_{G \sim \text{GMD}}\left[ \prod_{k=1}^{d}\left(\sum_{\ell=1}^{j_k} (X_{i_k,\ell} G_{\ell,j_k})_{q_k\text{-part}}\right)\right]\\
    & = 2^d \prod_{k=1}^{d}\mathbb{E}_{G \sim \text{GMD}}\left[\sum_{\ell=1}^{j_k} (X_{i_k,\ell} G_{\ell,j_k})_{q_k\text{-part}}\right]\\
    & = 2^d \prod_{k=1}^{d}\left(\sum_{\ell=1}^{j_k} (X_{i_k,\ell} \mathbb{E}_G[G_{\ell,j_k}])_{q_k\text{-part}}\right)\\
    & = 2^d \prod_{k=1}^{d} (X_{i_k,j_k} \mathbb{E}_G[G_{j_k,j_k}])_{q_k\text{-part}}\\
    & = 2^d \prod_{k=1}^{d} (X_{i_k,j_k})_{q_k\text{-part}} \mathbb{E}_G[G_{j_k,j_k}]\\
    & = \lambda_S H_S(X),
\end{align*}
where 
$$\lambda_S: = \prod_{k=1}^{d} \mathbb{E}_G[G_{j_k,j_k}].$$
(Note that we used the fact that $\mathbb{E}_G[G_{\ell,j_k}] = 0$ for all $\ell \neq j_k$, and that $\mathbb{E}_G[G_{j_k,j_k}] \in \mathbb{R}$.) The rest of the proof is almost exactly the same as before.}

We now need analogues of some of the results of Section \ref{sec:ingredients}. (We omit those whose statements and proofs are trivial to adapt.) As in the $\SO(n)$ case, we define the `over-Gaussian' distribution $\nu$ to be the distribution of $YG$, where $Y \sim \gamma$ and $G \sim \text{GMD}$. Fix a comfortable $d$-junta  $f$ and write
$$f = \sum_{S} a_I (2^d x_S),$$
where for $S=\{(1,i_1;q_1),\ldots,(d,i_d;q_d)\}$ (with $i_1,\ldots,i_d \in [n]$ distinct, and $q_1,\ldots,q_d \in \{\text{real},\mathbf{i},\mathbf{j},\mathbf{k}\}:=  \mathcal{R}$), we write
$$x_S: = \prod_{h=1}^{d} (x_{h,i_h})_{q_h\text{-part}}.$$
Note that $\{2^d x_S\}_S$ forms an orthonormal set of vectors in $L^2(\gamma)$, where $S$ ranges over tuples of the above form.

We need the following analogue of Claim \ref{claim:diagonal_terms_compute}.

\begin{claim}
Let $S = \{(1,i_1;q_1),\ldots,(d,i_d;q_d)\}$ where $i_1,\ldots,i_d \in [n]$ are distinct and $q_1,\ldots,q_d \in \{\text{real},\mathbf{i},\mathbf{j},\mathbf{k}\}:=  \mathcal{R}$, and let $x_S: = \prod_{h=1}^{d} (x_{h,i_h})_{q_h\text{-part}}$. Then
$$\|2^d x_S \|^2_{L^2(\nu)} = 1.$$
\end{claim}
\begin{proof}
In what follows, for $q \in \{\text{real},\mathbf{i},\mathbf{j},\mathbf{k}\}$ and $h \in \mathbb{H}$, we define $(h)_{- q\text{-part}}: = -(h)_{q\text{-part}}$, for notational convenience.
Observe that
$$\|2^dx_{S}\|_{L^2(\nu)}^2 = 4^d\mathbb{E}_{G \sim \text{GMD}}\mathbb{E}_{Y \sim \gamma}[(((YG)_{1,i_1})_{q_1\text{-part}})^2(((YG)_{2,i_2})_{q_2\text{-part}})^2\cdots (((YG)_{d,i_d}))_{q_d\text{-part}}^2].$$
Since for each $h \in [d]$, $((YG)_{h,i_h})_{q_h\text{-part}} = \sum_{k=1}^{i_h}\sum_{r \in \mathcal{R}} (Y_{h,k})_{r\text{-part}}(G_{k,i_h})_{r^{-1}q_h\text{-part}}$ involves only entries of $Y$ in row $h$ and entries of $G$ in column $i_h$ (and the $i_h$ are distinct), the random variables $\{(((YG)_{h,i_h})_{q_h\text{-part}})^2:\ h \in [d]\}$ form a system of independent random variables, and therefore
$$\|2^dx_{S}\|_{L^2(\nu)}^2 = 4^d\prod_{h=1}^{d}\mathbb{E}_{G \sim \text{GMD}}\mathbb{E}_{Y \sim \gamma}[(((YG)_{h,i_h})_{q_h\text{-part}})^2].$$
For each $h \in [d]$, we have
\begin{align*} & \mathbb{E}_{G\sim \text{GMD}}\mathbb{E}_{Y\sim \gamma}[(((YG)_{h,i_h})_{q_h\text{-part}})^2] \\
&= \mathbb{E}_{G}\mathbb{E}_{Y}\left[\left(\sum_{k=1}^{i_h}\sum_{r \in \mathcal{R}} (Y_{h,k})_{r\text{-part}}(G_{k,i_h})_{r^{-1}q_h\text{-part}}\right)^2\right]\\
& = \sum_{(k,r) \neq (k',r')} \mathbb{E}_{G}\mathbb{E}_{Y}[(Y_{h,k})_{r\text{-part}}(Y_{h,k'})_{r'\text{-part}}(G_{k,i_h})_{r^{-1}q_h\text{-part}}(G_{k',i_h})_{(r')^{-1}q_h\text{-part}}]\\
&+ \sum_{k=1}^{i_h} \mathbb{E}_{G \sim \text{GMD}}\mathbb{E}_{Y \sim \gamma} [Y_{h,k}^2G_{k,i_h}^2]\\
& = 0 + \sum_{k=1}^{i_h}\sum_{r \in \mathcal{R}} \mathbb{E}_{G \sim \text{GMD}}[((G_{k,i_h})_{r^{-1}q_h\text{-part}})^2]\mathbb{E}_{Y \sim \gamma} [((Y_{h,k})_{r\text{-part}})^2]\\
& = \sum_{k=1}^{i_h} \sum_{r \in \mathcal{R}}\mathbb{E}_{G \sim \text{GMD}}[((G_{k,i_h})_{r^{-1}q_h\text{-part}})^2]\cdot \tfrac{1}{4}\\
& = \tfrac{1}{4}(4(i_h-1)(1/(4n))+(n-i_h+1)/n)\\
& = 1/4.
\end{align*}
(Here, for the third equality we use the independence of $$(Y_{h,k})_{r\text{-part}},(Y_{h,k'})_{r'\text{-part}},(G_{k,i_h})_{r^{-1}q_h\text{-part}},(G_{k',i_h})_{(r')^{-1}q_h\text{-part}}$$
and the fact that $(Y_{h,k})_{r\text{-part}}$ and $(Y_{h,k'})_{r'\text{-part}}$ both have zero expectation.) Hence, $\|2^d x_{S}\|_{L^2(\nu)}^2 = 1$, as required.
\end{proof}

Similarly, we need the following analogue of Claim \ref{claim:off_diagonal_terms_compute}. For $S = \{(1,i_1;q_1),\ldots,(d,i_d;q_d)\}$ and $T = \{(1,j_1;p_1),\ldots,(d,j_d;p_d)\}$ we set
$$d(S,T): = |\{h \in [d]:\ i_h \neq j_h \text{ or } q_h \neq p_h\}|.$$

\begin{claim}
For any $S,T$ such that $d(S,T) = \ell$, we have $\card{\left\langle 2^d x_{S},2^d x_{T}\right\rangle} _{L^{2}\left(\nu\right)}\le \varepsilon_{\ell}$, where
$$\varepsilon_{\ell} := 2^{\ell+4}n^{-\ell/2} 2^{d\ell/\sqrt{n}}.$$
\end{claim}

To prove this we first need the following simple analogue of Claim \ref{claim:gmd lemma}.
\begin{claim}
Let $(i_{1},\ldots,i_{d}) \in [n]^d$ and $(j_1,\ldots,j_d)\in [n]^d$ be such that $|\{h \in [d]:\ i_h \neq j_h\}| = \ell$ and such that in the product
$$(G_{i_{1}j_{1}})_{r_1\text{-part}}(G_{i_{2}j_{2}})_{r_2\text{-part}}\cdots (G_{i_{d}j_{d}})_{r_d\text{-part}},$$
no matrix entry of $G$ appears more than twice. Then
\[
\card{\mathbb{E}_{G\sim\text{GMD}}\left[(G_{i_{1}j_{1}})_{r_1\text{-part}}(G_{i_{2}j_{2}})_{r_2\text{-part}}\cdots (G_{i_{d}j_{d}})_{r_d\text{-part}}\right]}\le\left(\frac{1}{4n}\right)^{\ell/2}.
\]
\end{claim}

\begin{proof}
If, in the product
$$(G_{i_{1}j_{1}})_{r_1\text{-part}}(G_{i_{2}j_{2}})_{r_2\text{-part}}\cdots (G_{i_{d}j_{d}})_{r_d\text{-part}},$$
some off-diagonal matrix entry of $G$ appears exactly once, then the expectation of the product is zero. We may therefore assume that every matrix entry of $G$ appears either exactly twice, or not at all, in the above product. If there are exactly $\ell$ values of $h$ such that $i_h \neq j_h$, then the above expectation factorises into a product of the expectations of the squares of $\ell/2$ off-diagonal and of the squares of $(d-\ell)/2$ diagonal entries: 
$$\prod_{k \in \mathcal{D}}\mathbb{E}[((G_{k,k})_{q_k\text{-part}})^2]\prod_{(i,j) \in \mathcal{E}}\mathbb{E}[((G_{i,j})_{r_i\text{-part}})^2],$$
where $\mathcal{E} \subset [n]^2\setminus \{(k,k):\ k \in [n]\}$, $|\mathcal{D}| = (d-\ell)/2$, $|\mathcal{E}| = \ell/2$ and $q_k,r_i \in \mathcal{R}$ for all $i$ and $k$. We have $\mathbb{E}[((G_{i,j})_{r_i\text{-part}})^2] = 1/(4n)$ for all $(i,j) \in \mathcal{E}$ and $\mathbb{E}[((G_{k,k})_{q_k\text{-part}})^2] \leq 1$ for all $k \in \mathcal{D}$, proving the claim.
\end{proof}

\begin{proof}
Let $\ell\geq 1$ and fix $S = \{(1,i_1;q_1),\ldots,(d,i_d;q_d)\}$ and $T = \{(1,j_1;p_1),\ldots,(d,j_d;p_d)\}$ such that $d(S,T)=\ell \geq 1$. Since $G$ is upper-triangular and $i_h,j_h \leq d$ for all $h \in [d]$, we have
$$((YG)_{h,i_h})_{q_h\text{-part}} = \sum_{k=1}^{i_h}\sum_{r \in \mathcal{R}} (Y_{h,k})_{r\text{-part}}(G_{k,i_h})_{r^{-1}q_h\text{-part}} = \sum_{k=1}^{d} \sum_{r \in \mathcal{R}} (Y_{h,k})_{r\text{-part}}(G_{k,i_h})_{r^{-1}q_h\text{-part}}$$
and
$$((YG)_{h,j_h})_{p_h\text{-part}} = \sum_{k=1}^{j_h} \sum_{r \in \mathcal{R}} (Y_{h,k})_{r\text{-part}}(G_{k,j_h})_{r^{-1}p_h\text{-part}} = \sum_{k=1}^{d} \sum_{r \in \mathcal{R}} (Y_{h,k})_{r\text{-part}}(G_{k,j_h})_{r^{-1}p_h\text{-part}}$$
for all $h \in [d]$. Hence,
\begin{align*}
x_{S}\left(YG\right) & =\prod_{h=1}^{d}((YG)_{h,i_h})_{q_h\text{-part}}\\
&= \sum_{K=\left(k_{1},\ldots,k_{d}\right) \in [d]^d,\atop
R = (r_1,\ldots,r_d) \in \mathcal{R}^d}(Y_{1,k_{1}})_{r_1\text{-part}}\cdots (Y_{d,k_{d}})_{r_d\text{-part}}(G_{k_{1},i_{1}})_{r_1^{-1}q_1\text{-part}}\cdots (G_{k_{d},i_{d}})_{r_d^{-1}q_d\text{-part}}
\end{align*}
 and  
\begin{align*}
x_{T}\left(YG\right)&=\sum_{K=\left(k_{1},\ldots,k_{d}\right) \in [d]^d,\atop
R = (r_1,\ldots,r_d) \in \mathcal{R}^d}(Y_{1,k_{1}})_{r_1\text{-part}}\cdots (Y_{d,k_{d}})_{r_d\text{-part}}(G_{k_{1},j_{1}})_{r_1^{-1}p_1\text{-part}}\cdots (G_{k_{d},j_{d}})_{r_d^{-1}p_d\text{-part}},
\end{align*}
 so, using the fact that, under $\nu$, the $((Y_{i,j})_{r\text{-part}}:i,j \in [n],\ r \in \mathcal{R})$ are independent and of expectation zero (and are independent of the $G_{i,j}$), we obtain 
\begin{align}\label{eq:master}
&\left\langle 2^d x_{S},2^d x_{T}\right\rangle _{\nu}\nonumber\\
&=\sum_{K \in [d]^d,\atop R \in \mathcal{R}^d}\mathbb{E}_{G \sim \text{GMD}}\left[(G_{k_{1}i_{1}})_{r_1^{-1}q_1\text{-part}}(G_{k_{1}j_{1}})_{r_1^{-1}p_1\text{-part}}\cdots (G_{k_{d}i_{d}})_{r_d^{-1}q_d\text{-part}}(G_{k_{d}j_{d}})_{r_d^{-1}p_d\text{-part}}\right].
\end{align}
For a $d$-tuple $(K;R) = (k_1;r_1\ldots,k_d;r_d)\in ([d] \times \mathcal{R})^d$, we write $m_{1}=m_1(K) = m_1(K;R) := |\{ h \in [d]:\ j_{h}=i_{h},\ k_{h}\ne i_{h}\}|$,
$m_{2}=m_2(K) = m_2(K;R) := |\{h \in [d]:\ j_{h}\ne i_{h},\ k_{h}\notin\{ i_{h},j_{h}\}\}|$
and $m_{3}=m_3(K;R) :=|\{h \in [d]:\ j_{h}\ne i_{h},\ k_{h}\in \{ i_{h},j_{h}\}\}|$; note that these quantities depend only on $K$ and not on $R$. We let $\mathcal{K}\left(m_{1},m_{2},m_{3}\right)$ denote the set of $d$-tuples $(K;R)$ with parameters $m_1,m_2$ and $m_3$. For $(K;R)\in\mathcal{K}\left(m_{1},m_{2},m_{3}\right)$, by
Claim~\ref{claim:gmd lemma} we have
\[
\left|\mathbb{E}_{G}\left[(G_{k_{1}i_{1}})_{r_1^{-1}q_1\text{-part}}(G_{k_{1}j_{1}})_{r_1^{-1}p_1\text{-part}}\cdots (G_{k_{d}i_{d}})_{r_d^{-1}q_d\text{-part}}(G_{k_{d}j_{d}})_{r_d^{-1}p_d\text{-part}}\right]\right|\le (4n)^{-\frac{2m_{1}+2m_{2}+m_{3}}{2}}.
\]
We further note that, for $(K;R) \in \mathcal{K}(m_1,m_2,m_3)$, the expectation in the above inequality is zero unless the following four conditions hold:
\begin{itemize}
    \item Whenever $i_h = j_h$ and $k_h \neq i_h$, we have $p_h=q_h$.
    \item Whenever $i_h = j_h = k_h$ we have $p_h = q_h$ and $r_h^{-1}p_h = \text{real}$.
    \item Whenever $i_h \neq j_h$ and $k_h =i_h$ we have $r_h^{-1}q_h = \text{real}$.
    \item Whenever $i_h \neq j_j$ and $k_h = j_h$ we have $r_h^{-1}p_h = \text{real}$.
\end{itemize}

In view of this we let $\mathcal{K}^*(m_1,m_2,m_3)$ be the set of all $d$-tuples $(K;R) \in \mathcal{K}(m_1,m_2,m_3)$ such that the above four conditions hold. For $(K;R) \in \mathcal{K}^*(m_1,m_2,m_3)$ we have $m_2+m_3= \ell$.

Summing over all $K$, we see that $|\langle x_{S},x_{T}\rangle|$ is at most
\begin{align*}
&\sum_{m_{1},m_{2},m_{3}}\sum_{K\in\mathcal{K}^*\left(m_{1},m_{2},m_{3}\right)}(4n)^{-\frac{2m_{1}+2m_{2}+m_{3}}{2}}\\
 & \le\sum_{m_{1},m_{2},m_{3}}(4n)^{-\frac{2m_{1}+2m_{2}+m_{3}}{2}}\left|\mathcal{K}^*\left(m_{1},m_{2},m_{3}\right)\right|.
\end{align*}
 Now 
\[
\left|\mathcal{K}^*\left(m_{1},m_{2},m_{3}\right)\right|\le\binom{d}{m_{1}}d^{m_{1}}\binom{\ell}{m_{2}}d^{m_{2}}2^{m_{3}}\cdot 4^{m_1+m_2}\le\frac{d^{2m_{1}+m_{2}}\ell^{m_{2}}2^{m_3}}{m_{1}!m_{2}!}\cdot 4^{m_1+m_2};
\]
note that the only difference with the corresponding expression in the proof of Claim \ref{claim:gmd lemma} is the extra factor of $4^{m_1+m_2}$, which comes from the fact that $r_h$ can vary freely over $\mathcal{R}$ (and still satisfy the above conditions) when $i_h \neq j_h$ and $k_h \notin \{i_h,j_h\}$, or when $i_h=j_h$ and $k_h \neq i_h$, but in no other cases.

Summing over all $m_{1},m_{2},m_{3}$ with $m_{2}+m_{3}=\ell$ completes
the proof, just as in the proof of Claim \ref{claim:off_diagonal_terms_compute}; the extra factor of $4^{-\frac{2m_1+2m_2+m_3}{2}}$ cancels out (or more than cancels out) the extra factor of $4^{m_1+m_2}$.
\end{proof}

The analogue of Lemma \ref{lem:contiguity} is proven from the above claims in a very similar way. Writing
$$f = \sum_{S}\alpha_{S}(2^d x_{S}),$$
we obtain
\begin{align*}
\|f\|_{L^{2}\left(\nu\right)}^{2} & \leq \sum_{S}|\alpha_{S}|^{2}\|2^d x_{S}\|_{L^{2}\left(\nu\right)}^{2}+\sum_{S\ne T}|\alpha_{S}||\alpha_{T}|\left|\left\langle 2^d x_{S},2^d x_{T}\right\rangle _{\nu}\right|\\
 & \le\|f\|_{L^{2}\left(\gamma\right)}^{2}+\sum_{S\ne T}\frac{|a_{S}|^{2}+|a_{T}|^{2}}{2}\card{\left\langle 2^d x_{S},2^d x_{T}\right\rangle _{\nu}}\\
 & \le\|f\|_{L^{2}\left(\gamma\right)}^{2}+
 \sum\limits_{\ell=1}^{d}\sum_{S}|a_{S}|^{2}\card{\left\{ T:d\left(T,S\right)=\ell\right\}} \cdot\epsilon_{\ell}\\
 & \le\|f\|_{L^{2}\left(\gamma\right)}^{2}+\|f\|_{L^{2}\left(\gamma\right)}^{2}\sum_{\ell=1}^{d}\epsilon_{\ell}
 \binom{d}{\ell} \ell ! 4^\ell \\
 & =\|f\|_{L^{2}\left(\gamma\right)}^{2}\left(1+\sum_{\ell=1}^{d}\epsilon_{\ell} (4d)^{\ell}\right),
 \end{align*}
 and the rest of the proof is essentially unchanged, up to reducing the value of $\delta$ by a constant factor.
\section{Bounding the eigenvalues of the Laplace-Beltrami operator}

\subsection{Bounding the eigenvalues of the Laplace-Beltrami operator in $\SO(n)$}
In this section, we obtain our desired bound on the eigenvalues of the Laplace-Beltrami operator in $\SO_n$, which we need in order to prove fineness.
\begin{lemma}\label{lem:lb_eigenval_curv}
  Let $\rho \in \hat{\SO}(n)$ be of level $D$, then the corresponding eigenvalue $\lambda_{\rho}$ of the Laplacian $\Delta$ satisfies
  \[
  \lambda_{\rho} \geq -C(nD+D^2).
  \]
\end{lemma}

To prove this, we will use~\cite[Theorems 2.3, 2.4]{berti-procesi}. It is well-known that the 
matrix coefficients of the irreducible representations of $\SO(n)$ are 
eigenvectors of the Laplace--Beltrami operator.
The aforementiond theorems in \cite{berti-procesi} give a formula for the eigenvalues, in terms of the fundamental weights of the irreducible representations. (The reader is referred to \cite[p. 219]{brian-hall} for a relatively concise definition of the fundamental weights.) The following statement (which suffices for our purposes) combines Theorem 2.3 and 2.4 from~\cite{berti-procesi}.
 \begin{thm}\label{thm:ev_formula_rep}
 Let $G$ be a simply connected Lie group of rank $k$. Then the eigenvalues of the Laplace-Beltrami operator $\Delta:L^2(G) \to L^2(G)$ are in one-to-one correspondence with the equivalence classes of complex irreducible representations of $G$. Furthermore, there is a set of vectors $\{w_1,\ldots,w_k\}$ in a finite-dimensional inner-product space, known as a {\em system of fundamental weights for $G$}, such that there is an explicit one-to-one correspondence 
 between the equivalence classes of irreducible 
 representations of $G$, and the elements of the discrete cone
 \[
 \sett{\sum\limits_{i=1}^k r_i w_i}{r_1,\ldots,r_k\in \mathbb{N}}.
 \]
 Furthermore, defining $\rho := \sum\limits_{i=1}^{k} w_i$, the 
 eigenvalue $\lambda_v$ of $\Delta$ corresponding to $v=\sum\limits_{i=1}^k r_i w_i$ satisfies 
 $$\lambda_v=-\norm{v+\rho}_2^2 + \norm{\rho}_2^2 = -2\langle v,\rho\rangle-\|v\|_2^2.$$
 \end{thm}
 
 We will choose a known system of fundamental weights for $\SO(n)$ as in~\cite[Section 5.1]{procesi}. We denote by 
 $E_{i,j}\in\mathbb{R}^{n\times n}$ the $n$ by $n$ matrix with a one in the $(i,j)$-th entry and zeros elsewhere. (We omit $n$ from
  notation, as it will always be clear from context.) The system of fundamental weights depends on whether $n$ is even or odd, and we deal with the two cases separately.
 \subsubsection*{The case of odd $n$}
 Let $n=2k+1$.  In this case, 
 the rank of $\SO(n)$ is $k$, and a system 
 of fundamental weights $\{w_1,w_2,\ldots,w_k\}$ is given by $w_i = \sum\limits_{j=1}^i u_j$ for $1\leq i\leq k-1$ and $w_{k} = \frac{1}{2}\sum\limits_{j=1}^{k} u_j$, 
 where $u_j := E_{j,k+j} - E_{k+j,j}$ for each $1 \leq j \leq k$.
 
 The equivalence class of representations corresponding to 
 $v=\sum\limits_{i=1}^{k} r_i w_i$ 
 is indexed by Young diagrams $\lambda = (\lambda_1,\ldots,\lambda_k)$, where $r_i = \lambda_i - \lambda_{i+1}$
 for each $1 \leq i\leq k-1$ and $r_k = 2\lambda_k$. The corresponding degree is therefore
 \begin{equation}\label{eq:degree_eq}
 D
 =\sum\limits_{i=1}^{k}\lambda_i
 =\sum\limits_{i=1}^{k-1}
 \left(\sum\limits_{i\leq j\leq k-1} r_j + \frac{r_k}{2}\right)
 +\frac{r_k}{2}
 =\sum\limits_{1\leq j\leq k-1}jr_j
 +\frac{k}{2} r_k.
 \end{equation}
We will use this equality to estimate the 
eigenvalue given in Theorem~\ref{thm:ev_formula_rep}. 
Using the formula therein, the corresponding 
eigenvalue is
\[
\lambda_v = -\norm{v+\rho}_2^2+\norm{\rho}_2^2,
\]
where here and henceforth we view the matrices $v$ and $\rho$ as vectors of length $n^2$, using the the natural flattening, and we use the unnormalized inner product on $\mathbb{R}^{n^2}$ (equivalently, 
$\inner{A}{B} = {\sf Tr}(A^tB)$). Then
\[
\lambda_v = -2\inner{v}{\rho}-\norm{v}_2^2.
\]
It remains to estimate the norm of $v$ and the inner
product of $v$ and $\rho$. To estimate
the norm of $v$, we write 
\[
v 
= \sum\limits_{i=1}^{k}w_i
= \sum\limits_{i=1}^{k-1}r_i
\sum\limits_{j=1}^{i}u_j
+\frac{r_k}{2}\sum\limits_{j=1}^{k}u_j
=\sum\limits_{j=1}^{k}
\left(\frac{r_k}{2}+\sum\limits_{i=j}^{k-1}r_i\right)u_j,
\]
and since the $u_j$'s are mutually orthogonal and each has $\|u_j\|_2^2=2$, 
we obtain
\[
\norm{v}_2^2
=2\sum\limits_{j=1}^{k}
\left(\frac{r_k}{2}+\sum\limits_{i=j}^{k-1}r_i\right)^2
\leq
2
\left(\sum\limits_{j=1}^{k}
\left(\frac{r_k}{2}+\sum\limits_{i=j}^{k-1}r_i\right)\right)^2
=2
\left(\sum\limits_{j=1}^{k}
\frac{kr_k}{2}+\sum\limits_{i=1}^{k-1}ir_i\right)^2,
\]
which is at most $2D^2$ by~\eqref{eq:degree_eq}.

To bound the inner product of $v$ and $\rho$, note that $\rho$ is the vector $v$ 
in which we take all the $r_i$'s to be $1$, hence 
by the computation above, we have
\[
\rho = \sum\limits_{i=1}^{k}w_i
=\sum\limits_{j=1}^{k}\left(k-j+\frac{1}{2}\right)u_j,
\]
and so
\[
\inner{\rho}{v}
=2\sum\limits_{j=1}^k
\left(k-j+\frac{1}{2}\right)
\left(\frac{r_k}{2}+\sum\limits_{i=j}^{k-1}r_i\right)
\leq 2k
\sum\limits_{j=1}^k
\left(\frac{r_k}{2}+\sum\limits_{i=j}^{k-1}r_i\right)
\leq 2kD.
\]
where in the last inequality we used~\eqref{eq:degree_eq}.
Overall, we see that the eigenvalue 
$\lambda_v$ satisfies $\lambda_v\geq -2D^2-nD$, as required.

\subsubsection*{The case of even $n$}
 Let $n=2k$. In this case, 
 the rank of $\SO(n)$ is $k$, and a system 
 of fundamental weights $\{w_1,\ldots,w_k\}$ is given by $w_i = \sum\limits_{j=1}^i u_j$ for $1 \leq i\leq k-2$, $w_{k-1} = \frac{1}{2}\sum\limits_{j=1}^{k} u_j$
 and $w_k = w_{k-1}-u_k$, 
 where again we define $u_j := E_{j,k+j} - E_{k+j,j}$ for $1 \leq j \leq k$. 
 
 Consider the the equivalence class 
 of representations corresponding to 
 $v=\sum\limits_{i=1}^{k} r_i w_i$.
 We now need to inspect the corresponding Young diagram to relate the $r_i$'s to 
 the degree of the representation, and we
 recall that the corresponding Young diagram 
 $\lambda = (\lambda_1,\ldots,\lambda_k)$ maybe either have $k$ non-zero 
 rows or at most $k-1$ non-zero rows. 
 
 \paragraph{Young diagrams with at most $k-1$ rows.}
 In this case we have $r_i = \lambda_{i}-\lambda_{i+1}$ for each $1\leq i \leq k$, and so the degree is 
 \[
 D = 
 \sum\limits_{i=1}^{k}\lambda_i
 =\sum\limits_{i=1}^{k}\sum\limits_{j=i}^{k-1}
 r_j
 =\sum\limits_{j=1}^{k-1}jr_j.
 \]
 Using the same estimates as before, 
 we obtain 
 \[
 v=
 \sum\limits_{i=1}^{k-2}r_i\sum\limits_{j=1}^{i}u_j
 +\sum\limits_{j=1}^{k}\frac{r_{k-1}}{2} u_j
 =\sum\limits_{j=1}^{k}
 \left(\sum\limits_{i=j}^{k-2}r_i+\frac{r_{k-1}}{2}\right)u_j,
 \]
 so
 \[
 \norm{v}_2^2
 \leq 2\sum\limits_{j=1}^{k}\left(\sum\limits_{i=j}^{k-2}r_i+\frac{r_{k-1}}{2}\right)^2
 \leq 
 2\left(\sum\limits_{i=j}^{k-2}ir_i
 +\frac{kr_{k-1}}{2}\right)^2
 \leq 2D^2.
 \]
 We also obtain $\rho = \sum\limits_{j=1}^{k-2} (k-\frac{1}{2}-j)u_j+\frac{1}{2}u_{k-1}+\frac{1}{2}u_{k}$, and so
 \[
 \inner{v}{\rho}
 =2\sum\limits_{j=1}^{k-2}
 \left(k-\frac{1}{2}-j\right)
 \left(\sum\limits_{i=j}^{k-2}r_i+\frac{r_{k-1}}{2}\right)
 +r_{k-1}
 \leq 2k\sum\limits_{j=1}^{k-2}
 \left(\sum\limits_{i=j}^{k-2}r_i+\frac{r_{k-1}}{2}\right)
 +r_{k-1}
 \leq 2k D + D,
 \]
 which is at most $2nD$, as desired.
 
 \paragraph{Young diagrams with $k$ rows.}
 In this case, we have $r_i = \lambda_{i}-\lambda_{i+1}$ for each $1 \leq i\leq k-2$, 
 and $(r_{k-1}, r_{k})$ 
 is either $(\lambda_{k-1}-\lambda_k, \lambda_{k-1}+\lambda_{k})$
 or $(\lambda_{k-1}+\lambda_k, \lambda_{k-1}-\lambda_{k})$. 
 The computation in both cases is similar and goes along the same lines as the computations so far, hence we deal only with the case where
 $(r_{k-1}, r_k) = (\lambda_{k-1}-\lambda_k, \lambda_{k-1}+\lambda_{k})$.
 
 The degree here is 
 \[
 D = 
 \sum\limits_{i=1}^{k}\lambda_i
 =
 \sum\limits_{i=1}^{k-2}
 \left(\sum\limits_{j=i}^{k-2}r_j+\frac{1}{2}(r_{k-1}+r_k)\right)
 =\sum\limits_{j=1}^{k-2}jr_j
 +\frac{k-2}{2}r_{k-1}
 +\frac{k-2}{2}r_{k}.
 \]
 We have the same formulae for $v$ and $\rho$ as before, and so
 \[
 \norm{v}_2^2
 \leq 2\sum\limits_{j=1}^{k}\left(\sum\limits_{i=j}^{k-2}r_i+\frac{r_{k-1}}{2}\right)^2
 \leq 
 2\left(\sum\limits_{i=j}^{k-2}ir_i
 +\frac{kr_{k-1}}{2}\right)^2
 \leq 2D^2,
 \]
 and
 \[
 \inner{v}{\rho}
 =2\sum\limits_{j=1}^{k-2}
 \left(k-\frac{1}{2}-j\right)
 \left(\sum\limits_{i=j}^{k-2}r_i+\frac{r_{k-1}}{2}\right)
 +r_{k-1}
 \leq 2k\sum\limits_{j=1}^{k-2}
 \left(\sum\limits_{i=j}^{k-2}r_i+\frac{r_{k-1}}{2}\right)
 +r_{k-1}
 \leq 2k D + D,
 \]
 which is at most $2nD$, as desired.

\subsection{Bounding the eigenvalues of the Laplace--Beltrami operator in $\SU(n)$}
The following lemma, analogous to Lemma~\ref{lem:lb_eigenval_curv}, gives our desired bound on the eigenvalues of the Laplace--Beltrami operator in $\SU(n)$.
\begin{lemma}\label{lem:lb_eigenval_curv_sun}
  For $\rho \in \widehat{\SU(n)}$ of level $D$, the corresponding eigenvalue $\lambda_{\rho}$ of $\Delta$ satisfies
  \[
  \lambda_{\rho} \geq -C(nD+D^2),
  \]
  where $C$ is an absolute constant.
\end{lemma}
\begin{proof}
The proof proceeds by a similar computation to the proof of Lemma~\ref{lem:lb_eigenval_curv}. We apply Theorem~\ref{thm:ev_formula_rep} in the case of $G=\SU(n)$, which has rank $k=n-1$, and use a system of fundamental weights from~\cite{berti-procesi}. 
The system of fundamental weights $\{w_1,\ldots,w_{n-1}\}$ is defined by
$$w_i = \sum_{j=1}^{i} e_j - \tfrac{i}{n}\sum_{j=1}^{n}e_j\ (1 \leq i \leq n-1),$$
where $\{e_i\}_{i=1}^{n}$ is the standard basis of $\mathbb{R}^n$; here, $e_i$ corresponds to
\begin{equation}\label{eq:cartan} \mathbf{i}E_{i,i},\end{equation}
where $\mathbf{i} = \sqrt{-1}$ and $E_{i,j}$ is the matrix with a one in the $(i,j)$th-entry and zeros elsewhere, as before. For each $1 \leq k \leq l \leq n-1$, we have
$$\langle w_k,w_l \rangle = k(n-l)/n.$$
Set $\sigma := \sum_{i=1}^{n-1}w_i$. For a partition $\lambda$ whose Young diagram has less than $n$ rows, the corresponding weight vector is
$$v = \sum_{i=1}^{n-1}a_iw_i,$$
where $a_i = \lambda_i - \lambda_{i+1}$ for all $i \in [n-1]$ and $\lambda_{n}:=0$; the level $D$ of the corresponding representation is given by
$$D = \sum_{i=1}^{n-1}a_i\min\{i,n-i\}.$$

It follows that, if $v = \sum_{k=1}^{n-1}a_k w_k$, then 
\begin{align*} \langle v,\mathbf{\sigma} \rangle & = \left\langle \sum_{k=1}^{n-1}a_k w_k, \sum_{k=1}^{n-1}w_k \right\rangle \\
& = \sum_{1 \leq k \leq l \leq n-1} a_k k(n-l)/n + \sum_{1 \leq l < k \leq n-1} a_k l(n-k)/n\\
& = \sum_{k=1}^{n-1} (n-k)(n-k+1)ka_k/(2n) + \sum_{k=1}^{n-1} k(k+1)l(n-k)a_k/(2n)\\
& \leq nD,
\end{align*}
and
\begin{align*} \langle v,v \rangle & = \left\langle \sum_{k=1}^{n-1}a_k w_k, \sum_{k=1}^{n-1}a_k w_k \right\rangle \\
& = \sum_{1 \leq k \leq l \leq n-1} a_k a_l k(n-l)/n + \sum_{1 \leq l < k \leq n-1} a_k a_l l(n-k)/n\\
& \leq 2D^2.
\end{align*}
Hence, by Theorem \label{thm:ev_formula_rep}, the corresponding eigenvalue $\lambda_v$ of $\Delta$ satisfies
$$\lambda_v = -2\langle v,\sigma \rangle - \langle \sigma,\sigma \rangle \geq -2nD - 2D^2.$$

\end{proof}



   \subsection{Bounding the eigenvalues of the Laplace-Beltrami operator in $\Sp(n)$}

    For bounds on the eigenvalues of the Laplace-Beltrami operator of $\textup{Sp}(n)$, we use the following system of fundamental weights (which are a scalar multiple of those in \cite{fulton-harris}). Setting $u_j = \mathbf{i}(E_{j,j} - E_{n+j,n+j})$ for each $j \in [n]$, we let 
$$w_i = \sum_{j=1}^{i}u_i$$
be the $i$th fundamental weight, for each $i \in [n]$. The irreducible representation corresponding to $v = \sum_{i=1}^{n}r_i w_i$ also corresponds to the Young diagram $\lambda = (\lambda_1,\ldots,\lambda_n)$, where $r_i = \lambda_i - \lambda_{i+1}$ for all $1 \leq i \leq n$ and $\lambda_{n+1}:=0$; the corresponding degree is
\begin{equation}\label{eq:degree_eq_symp}
 D
 =\sum\limits_{i=1}^{n}\lambda_i
 = \sum\limits_{j=1}^{n}jr_j.
 \end{equation}
From here on, the analysis is very similar to that for $\SO(n)$ (for $n$ odd). Using the formula in Theorem~\ref{thm:ev_formula_rep}, the 
eigenvalue of $\Delta$ corresponding to $v = \sum_{i=1}^{n}r_i w_i$ is
\[
\lambda_v = -\norm{v+\rho}_2^2+\norm{\rho}_2^2,
\]
where here we view the matrices $v$ and $\rho$ as vectors of length $(2n)^2$, using the natural flattening, and we use the unnormalized inner product on $\mathbb{R}^{(2n)^2}$ (equivalently, 
$\inner{A}{B} = {\sf Tr}(A^tB)$). Then
\[
\lambda_v = -2\inner{v}{\rho}-\norm{v}_2^2.
\]
It remains to estimate the norm of $v$ and the inner
product of $v$ and $\rho$. To compute
the norm of $v$, we write 
\[
v 
= \sum\limits_{i=1}^{n}w_i
= \sum\limits_{i=1}^{n}r_i
\sum\limits_{j=1}^{i}u_j
=\sum\limits_{j=1}^{n}
\left(\sum\limits_{i=j}^{n}r_i\right)u_j,
\]
and since the $u_j$'s are mutually orthogonal and each has $2$-norm-squared equal to $2$, 
we get that
\[
\norm{v}_2^2
=2\sum\limits_{j=1}^{n}
\left(\sum\limits_{i=j}^{n}r_i\right)^2
\leq
2
\left(\sum\limits_{j=1}^{n}
\left(\sum\limits_{i=j}^{n}r_i\right)\right)^2
=2
\left(\sum\limits_{j=1}^{n}
\sum\limits_{i=1}^{n}ir_i\right)^2,
\]
which is at most $2D^2$ by~\eqref{eq:degree_eq_symp}.

To bound the inner product of $v$ and $\rho$, note that
\[
\rho = \sum\limits_{i=1}^{n}w_i
=\sum\limits_{j=1}^{n}\left(n-j+1\right)u_j,
\]
and so
\[
\inner{\rho}{v}
=2\sum\limits_{j=1}^n
\left(n-j+1\right)
\left(\sum\limits_{i=j}^{n}r_i\right)
\leq 2n
\sum\limits_{j=1}^n
\left(\sum\limits_{i=j}^{n}r_i\right)
\leq 2nD.
\]
where in the last inequality we used~\eqref{eq:degree_eq_symp}.
Overall, we see that the eigenvalue 
$\lambda_v$ satisfies $\lambda_v\geq -2D^2-nD$, as desired.

\end{document}